\documentclass[reqno,12pt,article]{amsart}
\usepackage[applemac]{inputenc}
\usepackage{amsmath,amssymb,amsthm,graphicx,mathrsfs,url}
\usepackage[usenames,dvipsnames]{color}
\usepackage[colorlinks=true,linkcolor=Red,citecolor=Green]{hyperref}
\usepackage{enumitem}
\usepackage[russian, francais, english]{babel}
\usepackage{csquotes}
\usepackage[all]{xy}
\usepackage{tabularx}

\def\?[#1]{\textbf{[#1]}\marginpar{\Large{\textbf{??}}}}

\newtheorem{theo}{Theorem}
\newtheorem{prop}{Proposition}[section]

\newtheorem{lemm}[prop]{Lemma}
\newtheorem{corr}[theo]{Corollary}
\theoremstyle{definition}

\newtheorem{rem}[prop]{Remark}

\numberwithin{equation}{section}

\newcommand{\vol}{\mathrm{vol}}

\newcommand{\dd}{\mathrm{d}}
\newcommand{\dist}{\mathrm{dist}}
\newcommand{\Lie}{\mathcal{L}}

\newcommand{\e}{\mathrm{e}}

\newcommand{\strf}{{\tr}_\mathrm{s}^\flat}

\newcommand{\dom}{\mathcal{O}}

\newcommand{\R}{\mathbb{R}}
\newcommand{\Z}{\mathbb{Z}}
\newcommand{\Sc}{\mathcal{S}}
\renewcommand{\C}{\mathbb{C}}
\newcommand{\N}{\mathbb{N}}
\newcommand{\Pcal}{\mathcal{P}}

\newcommand{\Ccal}{\mathcal{C}}

\newcommand{\Mcal}{\mathcal{M}}
\newcommand{\grm}{\mathrm{g}}
\newcommand{\wl}{\mathrm{wl}}

\newcommand{\Cscr}{\mathscr{C}}
\newcommand{\nbf}{\mathbf{n}}
\newcommand{\ibf}{\mathbf{i}}
\newcommand{\diam}{\operatorname{diam}}

\DeclareMathOperator{\rank}{rank}

\let\Re=\Real
\DeclareMathOperator{\sgn}{sgn}

\DeclareMathOperator{\supp}{supp}

\DeclareMathOperator{\WF}{WF}

\DeclareMathOperator{\tr}{tr}

\DeclareMathOperator{\Id}{Id}

\title{Closed geodesics with prescribed intersection numbers}
\author{Yann Chaubet}

\begin{document}
\maketitle

\begin{abstract}
{Let $(\Sigma, g)$ be a closed, oriented, negatively curved surface, and fix pairwise disjoint simple closed geodesics $\gamma_{\star,1}, \dots \gamma_{\star, r}$. We give an asymptotic growth as $L \to +\infty$ of the number of primitive closed geodesic of length less than $L$ intersecting $\gamma_{\star,j}$ exactly $n_j$ times, where $n_1, \dots, n_r$ are fixed nonnegative integers. This is done by introducing a dynamical scattering operator associated to the surface with boundary obtained by cutting $\Sigma$ along $\gamma_{\star,1}, \dots, \gamma_{\star, r}$ and by using the theory of Pollicott-Ruelle resonances for open systems.}
\end{abstract}

\section{Introduction}

Let $(\Sigma,g)$ be a closed oriented negatively curved Riemannian surface and denote by $\Pcal$ the set of its oriented primitive closed geodesics. For $L > 0$ define
$$N(L) = \#\{\gamma \in \Pcal,~\ell(\gamma) \leqslant L\},$$
where $\ell(\gamma)$ is the length of a geodesic $\gamma$. Then a classical result obtained by Margulis \cite{margulis1969applications} states that
$$
N(L) \sim \frac{\e^{hL}}{hL}
$$
as $L \to +\infty$, where $h > 0$ is the topological entropy of the geodesic flow of $(\Sigma, g)$.

The purpose of this paper is to understand the asymptotic behavior of the quantity
$$
N(\nbf, L) = \# \left\{ \gamma \in \mathcal{P},~\ell(\gamma) \leqslant L,~i(\gamma,\gamma_{\star,j}) = n_j,~j=1,\dots,r \right\}
$$
as $L \to +\infty$, where $\gamma_{\star,1}, \dots, \gamma_{\star, r}$ are some pairwise disjoint simple closed geodesics, $\nbf = (n_1, \dots, n_r) \in \N^r$, and $i(\gamma, \gamma_{\star_j})$ is the geometric intersection number between $\gamma$ and $\gamma_{\star,j}$. The main result goes as follows.

\begin{theo}\label{thm:multi} Let $\nbf = (n_1, \dots, n_r) \in \N^{r}$. If $N(\nbf, L) > 0$ for some $L > 0$, then there are $C_\nbf > 0, d_\nbf \in \N$ and $h_{\nbf} \in ]0, h[$ such that
$$
N(\nbf, L) \sim C_\nbf L^{d_\nbf - 1} \e^{h_\nbf L}, \quad L \to +\infty.
$$
\end{theo}
In fact, a similar statement holds if we aditionnaly prescribe the order in which we want the intersections to occur, as follows. Let us denote by $\Sigma_1, \dots, \Sigma_q$ the connected components of the surface $\Sigma_\star = \Sigma \setminus (\gamma_{\star,1} \cup \cdots \cup \gamma_{\star, r})$ obtained by cutting $\Sigma$ along $\gamma_{\star, 1}, \dots, \gamma_{\star, r}$.  Let $\gamma \in \Pcal$ intersecting at least one $\gamma_{\star, i}$. For each $\gamma \in \Pcal$, we denote by $\omega(\gamma)$ the pair $(u,v)$ of cyclically ordered sequences $u = (u_1, \dots, u_N)$ and $v = (v_1, \dots, v_N)$ such that $\gamma$ goes through $\Sigma_{v_1}, \dots, \Sigma_{v_N}$ (in this order) and passes from $\Sigma_{v_{k}}$ to $\Sigma_{v_{k+1}}$ by crossing $\gamma_{\star, u_k}$, where $v_{N+1} = v_1$ (see Figure \ref{fig:intro}) ; those sequence are well defined modulo cyclic permutations.
Any pair of finite sequences $\omega$ will be called an \textit{admissible path} if $\omega \sim \omega(\gamma)$ for some $\gamma \in \Pcal$, where $\omega \sim \omega(\gamma)$ means that $\omega(\gamma)$ is a cyclic permutation of $\omega$ (the permutation being the same for both components of $\omega$). 
%For two admissible paths $\omega =(u,v)$ and $\omega'=(u',v')$, we will write $\omega \sim \omega'$ if $\omega'$ is a cyclic permutation of $\omega$,  meaning that for some cyclic permutation $\sigma$ of $\{1, \dots, N\}$ (here $N$ is the length of $u$ and $v$) we have $(u_{\sigma(1)}, \dots u_{\sigma(N)}) = u'$ and $(v_{\sigma(1)}, \dots v_{\sigma(N)}) = v'$.

\begin{figure}\label{fig:intro}
\includegraphics{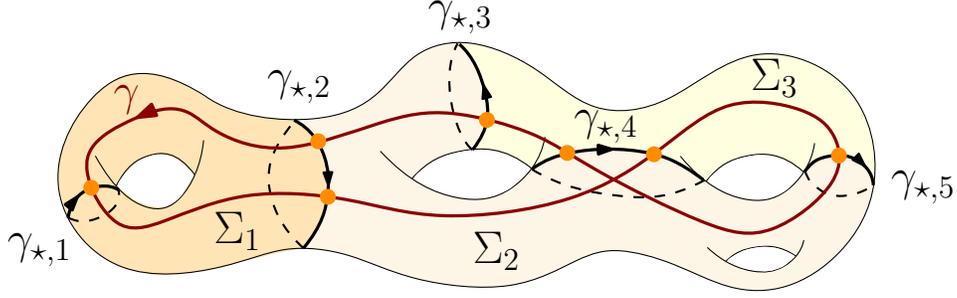}
\caption{A closed geodesic $\gamma$ on $\Sigma$. Here we have $r = 5$, $q = 3$, and $\omega(\gamma) \sim (u,v)$ where $u = (1, 2, 4, 5, 4, 3, 2)$ and $v = (1, 1, 2, 3, 2, 3, 2)$ (the starting point of $\gamma$ is the red arrow).}
\end{figure}

Denote by $S\Sigma$  the unit tangent bundle of $(\Sigma, g)$ and by $(\varphi_t)_{t \in \R}$ the associated geodesic flow, acting on $S\Sigma$. Let $\pi : S\Sigma \to \Sigma$ be the natural projection. We denote by $h_j > 0$ ($j = 1, \dots, q$) the entropy of the open system $(\Sigma_j, g|_{\Sigma_j})$, that is, the topological entropy of the flow $\varphi$ restricted to the trapped set
$$
K_j = \overline{\{(x,v) \in S\Sigma,~\pi(\varphi_t(x,v)) \in \Sigma_j,~t \in \R\}}
$$
where the closure is taken in $S\Sigma$.

For any admissible path $\omega = (u,v)$ of length $N$, we set
$$
h_\omega= \max \{h_{v_k},~ k = 1, \dots, N\}, \quad d_\omega = \#\{k = 1, \dots, N,~h_{v_k} = h_\omega\}.
$$
The number $h_\omega$ is the maximum of the entropies of the surfaces encountered by any $\gamma \in \Pcal$ satisfying $\omega(\gamma) \sim \omega$, while $d_\omega$ is the number of times any such $\gamma$ will encounter a surface whose entropy is equal to $h_\omega$ (for example, in Figure \ref{fig:intro}, if the entropy $h_2$ of $\Sigma_2$ is the greatest, we have $h(\omega) = h_2$ and $d(\omega) = 3$, as $\gamma$ travels three times through $\Sigma_2$). 

In fact, the numbers $h_\omega$ and $d_\omega$ depend only on 
$
\nbf(\omega) = (n_1, \dots, n_r)
$
where $n_j = \#\{k = 1, \dots, N,~u_k = j\}$ (see \S\ref{sec:multi}) ; we will thus refer to them by $h_{\nbf(\omega)}$ and $d_{\nbf(\omega)}$ respectively.

\begin{theo}\label{thm:path}
Let $\omega$ be an admissible path. Then there is $c(\omega) > 0$ such that
$$
\#\{\gamma \in \Pcal,~\ell(\gamma) \leqslant L,~\omega(\gamma) \sim \omega\} \sim c(\omega)  L^{d_{\nbf(\omega)}- 1} \e^{h_{\nbf(\omega)}L}, \quad L \to +\infty.
$$
\end{theo}
Note that Theorem \ref{thm:multi} can be deduced from Theorem \ref{thm:path} by summing over admissible paths $\omega$ with $\nbf(\omega) = \nbf$ where $\nbf \in \N^r$ is fixed. We refer to \S\ref{sec:multi} for a slightly more precise statement.

For the sake of simplicity, and to make the exposition clearer, we will deal in the major part of this article with the case $r = 1$. The case $r > 1$ will be then obtained by identical techniques, as described in \S\ref{sec:multi}. Thus from now on and unless stated otherwise, we will assume that we are given only a simple closed geodesic $\gamma_\star$ and we set
$$
N(n,L) = \# \left\{ \gamma \in \mathcal{P},~\ell(\gamma) \leqslant L,~i(\gamma,\gamma_{\star}) = n\right\}.
$$
In this context, our result reads as follows.

\begin{theo}\label{thm:main} Let $\gamma_\star$ be a nontrivial simple closed geodesic of $(\Sigma,g).$
\begin{enumerate}[label=(\alph*)]
\item Suppose that $\gamma_\star$ is not separating, that is $\Sigma \setminus \gamma_\star$ is connected. Then there exists $c_\star>0$ such that for each $n \in \N$,
$$N(n,L) \sim \frac{(c_\star L)^n}{n!} \frac{\e^{h_\star L}}{h_\star L}, \quad L \to + \infty,$$
where $h_\star \in ]0, h[$ is the entropy of the geodesic flow of the open system $(\Sigma \setminus \gamma_\star)$.
\\

\item Suppose that $\gamma_\star$ separates $\Sigma$ in two surfaces $\Sigma_1$ and $\Sigma_2$. Let $h_j \in ]0, h[$ denote the entropy of the open system $(\Sigma_j, g|_{\Sigma_j})$ and set $h_\star = \max(h_1, h_2)$. Then there is $c_\star>0$ such that for each $n \in \N$ we have as $L \to +\infty,$
$$
N(2n,L) \sim 
\left\{ 
\begin{matrix} 
\displaystyle{\frac{(c_\star L)^n}{n!} \frac{\e^{h_\star L}}{h_\star L}} \quad &\text{ if } h_1 \neq h_2,\vspace{0.3cm}  \\ 
\displaystyle{2\frac{\left(c_\star L^2\right)^n}{(2n)!}\frac{\e^{h_\star L}}{h_\star L}} \quad  &\text{ if }~ h_1 = h_2,
\end{matrix}
\right.
$$
\end{enumerate}
\end{theo}

As before, the entropy $h_\star$ is defined as the topological entropy of the geodesic flow restricted to the trapped set
$$
K_\star = \overline{\{(x,v) \in S\Sigma,~\pi(\varphi_t(x,v)) \in \Sigma \setminus \gamma_\star~t \in \R\}}
$$
where the closure is taken in $S\Sigma$.

\begin{rem}
\begin{enumerate}[label=(\roman*)]
\item The case $n = 0$ is well known and follows from the growth rate of periodic orbits of Axiom A flows obtained by Parry-Pollicott \cite{parry1983analogue} (see \S\ref{subsec:cutting}). However, to the best of our knowledge, the result is new for $n > 0$.
\item Using a classical large deviation result by Kifer \cite{kifer1994large} and Bonahon's intersection form \cite{bonahon1986bouts}, we are in fact able to show that a typical closed geodesic $\gamma$ satisfies $i(\gamma, \gamma_\star) \approx I_\star \ell(\gamma)$  for some $I_\star > 0$ not depending on $\gamma$ (see Proposition \ref{lem:deviation} below for a precise statement). In particular Theorem \ref{thm:main} is a statement about very uncommon closed geodesics.
\end{enumerate}
\end{rem}

We also have an equidistribution result, as follows. We still denote by $(\varphi_t)_{t \in \R}$ the geodesic flow of $(\Sigma, g)$ acting on the unit tangent bundle $S\Sigma$. We set
$$
\tilde \partial = \{(x,v) \in S\Sigma,~x \in \gamma_\star\} \quad \text{and} \quad \Gamma = S\gamma_\star \cup \left\{z \in \tilde \partial,~\varphi_t(z) \in S\Sigma \setminus \tilde \partial,~t  > 0 \right\}
$$
where $S\gamma_\star = \{(x,v) \in \tilde \partial,~v \in T_x\gamma_\star\}.$ We define the Scattering map $S : \tilde \partial \setminus \Gamma \to \tilde \partial$ by 
$$
S(z) = \varphi_{\ell(z)}(z), \quad \ell(z) = \inf \{t > 0,~\varphi_t(z) \in \tilde \partial\}, \quad z \in \tilde \partial \setminus \Gamma.
$$
For any $n \in \N_{\geqslant 1}$ we set
$$
\Gamma_n = \tilde \partial \setminus \left\{z \in \tilde \partial \setminus \Gamma,~S^k(z) \in \tilde \partial \setminus \Gamma,~k = 1, \dots,~n - 1\right\}
$$
which is a closed set of Lebesgue measure zero, and
$$
\ell_n(z) = \ell(z) + \dots + \ell(S^{n-1}(z)), \quad z \in \tilde \partial \setminus \Gamma_n.
$$
\begin{theo}\label{thm:equidistribution}
Let $n \geqslant 1$. For any $f \in C^\infty(\tilde \partial)$ the limit
$$
\lim_{L \to +\infty} \frac{1}{n N(n,L)} \sum_{\substack{\gamma \in \Pcal \\ i(\gamma, \gamma_\star) = n}} \sum_{z \in I_\star(\gamma)} f(z)
$$
exists, where for any $\gamma \in \Pcal$, 
$
I_\star(\gamma) = \{(x,v) \in S\gamma,~x \in \gamma_\star\}
$
is the set of incidence vectors of $\gamma$ along $\gamma_\star$. This formula defines a probability measure $\mu_n$ on $\tilde \partial$, whose support is contained in $\Gamma_n.$
\end{theo}
We will give a full description of $c_\star$ and $\mu_n$ in terms of Pollicott-Ruelle resonant states of the geodesic flow of $(\Sigma_\star, g)$ for the resonance $h_\star$ in \S\ref{sec:bowen}. Here $\Sigma_\star$ is the compact surface with boundary obtained by cutting $\Sigma$ along $\gamma_\star$ (see \S\ref{subsec:cutting}).

\subsection*{Strategy of proof}
A key ingredient used in the proof of Theorems \ref{thm:main} and \ref{thm:equidistribution} is the Scattering operator $\Sc(s) : C^\infty(\tilde \partial) \to C^\infty(\tilde \partial \setminus \Gamma)$ which is defined by
$$
\Sc(s)f(z) = f(S(z)) \e^{-s\ell(z)}, \quad z \in \tilde \partial \setminus \Gamma, \quad s \in \C.
$$
As a first step (which is of independent interest, see Corollary \ref{cor:mero}), we prove that the family $s \mapsto \Sc(s)$ extends to a meromorphic family of operators $\Sc(s) : C^\infty(\tilde \partial) \to \mathcal{D}'(\tilde \partial)$ on the whole complex plane (here $\mathcal{D}'(\tilde \partial)$ denotes the space of distributions on $\tilde \partial$), whose poles are contained in the set of Pollicott--Ruelle resonances of the geodesic flow of the surface with boundary $(\Sigma_\star, g)$ (see \S\ref{subsec:resolv} for the definition of those resonances). In this context, the existence of such resonances follows from the work of Dyatlov--Guillarmou \cite{dyatlov2016pollicott}. By using the microlocal structure of the resolvent of the geodesic flow provided by \cite{dyatlov2016pollicott}, we are moreover able to prove that for any $\chi \in C^\infty_c(\tilde \partial \setminus S\gamma_\star)$, the composition $(\chi \Sc(s))^n$ is well defined for any $n \geqslant 1$, as well as its super flat trace (meaning that we also look at the action of $\Sc(s)$ on differential forms, see \S\ref{subsec:flattrace}) which reads
$$
\strf[(\chi \Sc(s))^n] = n \sum_{i(\gamma, \gamma_\star) = n} \frac{\ell^\#(\gamma)}{\ell(\gamma)} \e^{-s\ell(\gamma)} \prod_{z \in I_\star(\gamma)} \chi(z),
$$
where the products runs over all closed geodesics (not necessarily primitive) $\gamma$ with $i(\gamma, \gamma_\star) = n$ and $\ell^\#(\gamma)$ is the primitive length of $\gamma$ ; this formula is a consequence of the Atiyah-Bott trace formula \cite{atiyah1967lefschetz}. Furthermore, using \textit{a priori} bounds on the growth of $N(n,L)$ (obtained in \S\ref{sec:apriori}), we prove that $s \mapsto \strf[(\chi \Sc(s))^n]$ has a pole of order $n$ at $s = h$, provided that $\chi$ has enough support. Then letting the support of $1-\chi$ being very close to $S\gamma_\star$, and estimating the growth of geodesics intersecting $n$ times $\gamma_\star$ with at least one small angle, we are able to derive Theorem \ref{thm:main} from a classical Tauberian theorem of Delange \cite{delange1954generalisation}.
 
\subsection*{Application to geodesic billards}

We finally state a result on the growth number of periodic trajectories of the billard problem associated to a negatively curved surface with totally geodesic boundary, which follows from the methods used to prove Theorem \ref{thm:main}.

\begin{corr}\label{thm:3}
Let $(\Sigma', g')$ be a negatively curved surface with totally geodesic boundary. For any $n \in \N$ and $L > 0$ we denote by $N(n,L)$ the number of closed billiard trajectories on $(\Sigma', g')$ (that is, geodesic trajectories that bounce on $\partial \Sigma'$ according to Descartes' law) with exactly $n$ rebounds, and with length not greater than $L$. Then there is $c' > 0$ such that
$$
N(n,L) \sim \frac{(c'L)^n}{ n!} \frac{\e^{h'L}}{h'L}, \quad L \to +\infty,
$$
where $h'$ is the entropy of the open system $(\Sigma', g')$. \end{corr}

\subsection*{Related works}

As mentioned before, the case $n = 0$ follows from the work Parry--Pollicott \cite{parry1983analogue} which is based on important contributions of Bowen \cite{bowen1972equidistribution, bowen1973symbolic}, as the geodesic flow on $(\Sigma_\star, g)$ can be seen as an Axiom A flow (see Lemma \ref{lem:Sigma} below and \cite[\S6.1]{dyatlov2016pollicott}). For counting results on non compact Riemann surfaces, see also Sarnak \cite{sarnak1980prime}, Guillop\'e \cite{guillope1986distribution}, or Lalley \cite{lalley1989renewal}. We refer to the work of Paulin--Pollicott--Schapira \cite{paulin2012equilibrium} for counting results in more general settings.

We also mention a result by Pollicott \cite{pollicott1985asymptotic} which says that, if $(\Sigma, g)$ is of constant curvature $-1$ and if $\gamma_\star$ is not separating,
\begin{equation}\label{eq:pollicott}
\frac{1}{N(L)} \sum_{\substack{\gamma \in \Pcal \\ \ell(\gamma) \leqslant L}} i(\gamma, \gamma_\star) \sim I_\star L
\end{equation}
for some $I_\star > 0$, which means that, the average intersection number between $\gamma_\star$ and geodesics of length not greater than $L$ is about $I_\star L$. We show that this also holds in our context (see \S\ref{sec:bonahon}).

Lalley \cite{lalley1988closed}, Pollicott \cite{pollicott1991homology} and Anantharaman \cite{anantharaman2000precise} investigated the asymptotic growth of the number of closed geodesics satisfying some homological constraints (see also Philips--Sarnak \cite{phillips1987geodesics} and Katsuda--Sunada \cite{katsuda1988homology} for the constant curvature case). They show that for any homology class $\xi \in H_1(\Sigma, \Z)$, we have
$$
\#\{\gamma \in \Pcal,~ \ell(\gamma) \leqslant L,~[\gamma] = \xi\} \sim C \e^{hL} / L^{\grm + 1}
$$
for some $C >0$ independent of $\xi$, where $\grm$ is the genus of $\Sigma$ and $h > 0$ is the entropy of the geodesic flow of $(\Sigma, g)$. Such asymptotics are obtained by studying $L$-functions associated to some characters of $H_1(\Sigma, \Z)$. However our problem is very different in nature; indeed, fixing a constraint in homology boils down to fixing \textit{algebraic} intersection numbers whereas here we are interested in \textit{geometric} intersection numbers. This makes $L$-funtions not well suited for this situation.

In the context of hyperbolic surfaces, Mirzhakani \cite{mirzakhani2008growth, mirzakhani2016counting} computed the asymptotic growth of closed geodesics with prescribed self intersection numbers. Namely, for any $k \in \N$, we have
$$
\#\{\gamma \in \Pcal,~\ell(\gamma) \leqslant L,~ i(\gamma, \gamma) = k\} \sim c_k L^{6(\grm - 1)},
$$
where $i(\gamma, \gamma)$ denote the self-intersection number of $\gamma$ (see also \cite{erlandsson2016counting}).
\subsection*{Organization of the paper}

The paper is organized as follows. In \S\ref{sec:geo} we introduce some geometrical and dynamical tools. In \S\ref{sec:scat} we introduce the dynamical scattering operator which is a central object in this paper and we compute its flat trace. In \S\ref{sec:apriori} we prove a priori bounds on $N(n,L)$. In \S\ref{sec:tauberian} we use a Tauberian argument to estimate certain quantities. In \S\ref{sec:proofthm} we prove Theorem \ref{thm:main}. In \S\ref{sec:bowen} we prove Theorem \ref{thm:equidistribution}.  In \S\ref{sec:billard} we explain how the methods described above apply to the billard problem. In \S\ref{sec:deviation} we show that a typical closed geodesic $\gamma$ satisfies $i(\gamma, \gamma_\star) \approx I_\star \ell(\gamma)$ for some $I_\star > 0.$ Finally in \S\ref{sec:multi} we extend the results to the case where we are given more than one closed geodesic.

\subsection*{Acknowledgements}  I am grateful to Colin Guillarmou for a lot of insightful discussions and for his careful reading of many versions of the present article. I also thank Fr\'ed\'eric Paulin for his help concerning the geometrical lemma \ref{lem:0}. This project has received funding from the European Research Council (ERC) under the European Unions Horizon 2020 research and innovation programme (grant agreement No. 725967).

\section{Geometrical preliminaries}\label{sec:geo}

We recall here some classical geometrical and dynamical notions, and introduce the Pollicott-Ruelle resonances that arise in our setting.

\subsection{Structural equations}
Here we recall some classical facts from \cite[\S7.2]{singer1976lecture} about geometry of surfaces. Denote by $M = S \Sigma = \{(x,v) \in T\Sigma,~\|v\|_g = 1\}$  the unit tangent bundle of $\Sigma$, by $X$ the geodesic vector field on $M$, that is the generator of the geodesic flow $\varphi = (\varphi_t)_{t\in \R}$ of $(\Sigma, g)$, acting on $M$. We have the Liouville one-form $\alpha$ on $M$ defined by
$$
\langle \alpha(z), \eta\rangle = \langle \dd_{(x,v)} \pi(\eta), v \rangle, \quad (x,v) \in M, \quad \eta \in T_{(x,v)}M. 
$$
Then $\alpha$ is a contact form (that is, $\alpha \wedge \dd \alpha$ is a volume form on $M$) and it turns out that $X$ is the Reeb vector field associated to $\alpha$, meaning that
$$
\iota_X \alpha = 1, \quad \iota_X \dd \alpha = 0,
$$ 
where $\iota$ denote the interior product.

We also set $\beta = R_{\pi/2}^*\alpha$ where for $\theta \in \R$, $R_{\theta} : M \to M$ is the rotation of angle $\theta$ in the fibers; finally we denote by $\psi$ the connection one-form, that is the unique one-form on $M$ satisfying
$$
\iota_V \psi = 1, \quad \dd \alpha = \psi \wedge \beta, \quad \dd \beta = - \psi \wedge \alpha,
$$
where $V$ is the vertical vector field, that is, the vector field generating $(R_\theta)_{\theta \in \R}$. Then $(\alpha, \beta, \psi)$ is a global frame of $T^*M$, and we denote $H$ the vector field on $M$ such that $(X, H, V)$ is the dual frame of $(\alpha, \beta, \psi)$. We then have the commutation relations 
$$
[V, X] = H, \quad [V, H] = -X, \quad [X, H] = (\kappa \circ \pi) V,
$$
where $\kappa$ is the Gauss curvature of $(\Sigma, g)$.

\subsection{The Anosov property}\label{subsec:theanosovproperty}
It is well known \cite{anosov1967geodesic} that the flow $(\varphi_t)$ has the Anosov property, that is, for any $z \in M$, there is a splitting 
$$
T_zM = \R X(z) \oplus E_s(z) \oplus E_u(z)
$$
which depends continuously on $z$, and with the following property. 
For any norm $\|\cdot\|$ on $TM$, there exists $C, \nu > 0$ such that
$$
\left\|\dd \varphi_t(z) v\right\| \leqslant C \e^{-\nu t} \|v\| , \quad v \in E_s(z), \quad t \geqslant 0, \quad z \in M,
$$
and 
$$
\left\|\dd \varphi_{-t}(z) v\right\| \leqslant C \e^{-\nu t} \|v\|, \quad v \in E_u(z), \quad t \geqslant 0, \quad z \in M
$$
In fact $E_s(z) \oplus E_u(z) = \ker \alpha(z)$ and there exists two continuous functions $r_\pm : M \to \R$ such that $\pm r_\pm > 0$ and
$$
E_s(z) = \R(H(z) + r_-V(z)), \quad E_u(z) = \R (H(z) + r_+ V(z)), \quad z \in M.
$$
Moreover $r_\pm$ satisfy the Ricatti equation
$$
\pm X r_\pm + r_\pm^2 + \kappa \circ \pi = 0,
$$
where $\kappa$ is the curvature of $\Sigma$.

We will denote by $T^*M = E_0^* \oplus E^*_s \oplus E_u^*$ the splitting defined by (here the bundle $\R X$ is denoted by $E_0$)
$$
E_0^*(E_u \oplus E_s) = 0, \quad E_s^*(E_s \oplus E_0) = 0,  \quad E_u^*(E_u \oplus E_0) = 0.
$$
Then we have $E_0^* = \R \alpha$ and
\begin{equation}\label{eq:e^*}
E_s^* = \R(r_-\beta - \psi), \quad E_u^* = \R(r_+ \beta - \psi).
\end{equation}
Note that this decomposition does not coincide with the usual dual decomposition, but it is motivated by the fact that covectors in $E_s^*$ (resp. $E_u^*$) are exponentially contracted in the future (resp. in the past). Also, we will often consider the symplectic lift of $\varphi_t,$
\begin{equation}\label{eq:Phi_t}
\Phi_t(z,\xi) = (\varphi_t(z),~\dd\varphi_t(z)^{-\top}\cdot\xi), \quad (z,\xi) \in T^*M, \quad t \in \R,
\end{equation}
where $^{-\top}$ denotes the inverse transpose. We have the following lemma (see \cite[\S3.2]{dang2020poincar}).
\begin{lemm}\label{lem:phitbeta} For any $\pm t > 0$ we have
$
\iota_V\Phi_t(\beta) \neq 0$ and $\iota_H\Phi_t(\psi) \neq 0$.
\end{lemm}

\subsection{A nice system of coordinates}
In what follows we denote 
$$\tilde \partial = \{(x,v) \in M,~x \in \gamma_\star\} = S\Sigma|_{\gamma_\star}.$$

\begin{lemm}\label{lem:coordinates}
There exists a tubular neighborhood $U$ of $\tilde \partial$ in $M$ and coordinates $(\tau, \rho, \theta)$ on $U$ with
$$
U \simeq (\R/\ell_0\Z)_\tau \times (-\delta, \delta)_\rho \times (\R/{2\pi \Z})_\theta,
$$
such that 
$$|\rho(z)| = \dist_g(\pi(z), \gamma_\star), \quad S_z \Sigma = \{(\tau(z), \rho(z), \theta),~\theta \in \R/2\pi\Z\}, \quad z \in U.$$
Moreover in these coordinates, we have, on $\{\rho = 0\}$,
$$
X =  \cos(\theta) \partial_\tau + \sin(\theta) \partial_\rho, \quad H = -\sin(\theta) \partial_\tau + \cos(\theta) \partial_\rho, \quad V = \partial_\theta,
$$
and
$$ 
\alpha = \cos(\theta) \dd \tau + \sin(\theta) \dd \rho, \quad \beta = -\sin(\theta) \dd \tau + \cos(\theta) \dd \rho, \quad \psi = \dd \theta.
$$
\end{lemm}

\begin{proof}
For $\tau \in \R/\ell_0\Z$ we set $(x_\tau, v_\tau) = \varphi_\tau(\gamma_\star(0), \dot \gamma_\star(0)).$ We now define, for $\delta > 0$ small enough,
$$
\psi(\tau, \rho, \theta) = R_{\theta - \pi/2} \varphi_\rho(x_\tau, \nu(x_\tau)), \quad (\tau, \rho, \theta) \in \R/\ell_0\Z \times (-\delta, \delta) \times \R/2\pi \Z,
$$
where $R_\eta : S\Sigma \to S\Sigma$ is the rotation of angle $\eta$ and $\nu(x_\tau) = R_{\pi /2}v_\tau$. As $V = \partial_\theta$ and $\iota_V \alpha = \iota_V \beta = 0$, we may write $\alpha(\tau, 0, \theta) = a(\tau, \theta) \dd \tau + b(\tau, \theta) \dd \rho$ and $\beta(\tau, 0, \theta) = a'(\tau, \theta) \dd \tau + b'(\tau, \theta) \dd \rho$ for some smooth functions $a,a',b,b'$. Now since $\dd \alpha = \psi \wedge \beta$ we obtain $\Lie_V \alpha = \iota_V \dd \alpha = \beta$, and similarly $\Lie_V \beta = - \alpha.$ Thus we obtain $a' = \partial_\theta a$, $b' = \partial_\theta b$ and
$$
\partial_\theta^2 a + a =0, \quad \partial_\theta^2 b + b = 0.
$$
In consequence we have $a(\tau, \theta) = a_1(\tau) \cos \theta + a_2(\tau) \sin \theta$ and $b(\tau, \theta) = b_1(\tau) \cos \theta + b_2(\tau) \sin \theta$ for some smooth functions $a_1, a_2, b_1, b_2.$ Moreover, by definition of the coordinates $(\tau, \rho, \theta)$, one has
\begin{equation}\label{eq:xtheta}
X(\tau, 0, 0) = \partial_\tau \quad \text{ and }\quad  X(\tau, 0, \pi / 2) = \partial_\rho.
\end{equation} 
Therefore $a_1 = b_2 = 1$ and $a_2 = b_1 = 0$. We thus get the desired formulas for $\alpha$ and $\beta$. Now writing $\psi = a'' \dd\tau + b''\dd \rho + \dd \theta$ and using $\Lie_V \psi = 0$, we obtain $\partial_\theta a'' = \partial_\theta b'' = 0$. As $\iota_X \psi = 0$ we obtain $a'' = b'' = 0$ by (\ref{eq:xtheta}). The formulae for $X, H, V$ follow.
\end{proof}

\begin{rem}\label{rem:nstar}
If $\partial = \{\rho = 0\}$, we get for any $z = (\tau, 0, \theta) \in \partial$
$$
T_z\partial = \R V(z) \oplus \R(\cos(\theta)X(z) - \sin(\theta) H(z)), \quad N^*_z\partial = \R(\sin(\theta)\alpha(z) + \cos(\theta)\beta(z)).
$$
\end{rem}

\subsection{Cutting the surface along $\gamma_\star$}\label{subsec:cutting}

As mentioned in the introduction, we may see $\Sigma \setminus \gamma_\star$ as the interior of a compact surface $\Sigma_\star$ with boundary consisting of two copies of $\gamma_\star$. By gluing two copies of the annulus $U$ obtained in the preceding subsection on each component of the boundary of $\Sigma_\star$, we construct a slightly larger surface $\Sigma_\delta \supset \Sigma_\star$ whose boundary is identified with the boundary of $U$ (see Figure \ref{fig:Sigma}). 

\begin{lemm}\label{lem:Sigma}
The surface $\Sigma_\delta$ has strictly convex boundary, in the sense that the second fundamental form of the boundary $\partial \Sigma_\delta$ with respect to its outward normal pointing vector is strictly negative.
\end{lemm}

\begin{proof}
In the coordinates defined $(\tau, \rho)$ given by Lemma \ref{lem:coordinates}, the metric $g$ has the form
\begin{equation}\label{eq:metric}
 \dd \rho^2 + f(\rho) \dd \tau^2,
 \end{equation}
for some $f > 0$ satisfying $f'(0) = 0$ and one can check that the scalar curvature writes $\kappa(\tau, \rho) = -f''(\rho)/f(\rho)$. Thus $f'' > 0$ which gives $\pm f'(\rho) > 0$ if $\pm \rho > 0$. Now if $\nabla$ is the Levi-Civita connexion we have
$$
-\langle \nabla_{\partial_\tau} \partial_\rho, \partial_\tau \rangle = -f(\rho) \Gamma_{\rho \tau}^\tau = -f'(\rho) / 2,
$$
which concludes, since $\partial_\rho$ is outward pointing (resp. inward pointing) for at $\{\rho = \delta\}$ (resp. $\{\rho = -\delta\}$).
\end{proof}

\begin{figure}
\includegraphics[scale=0.65]{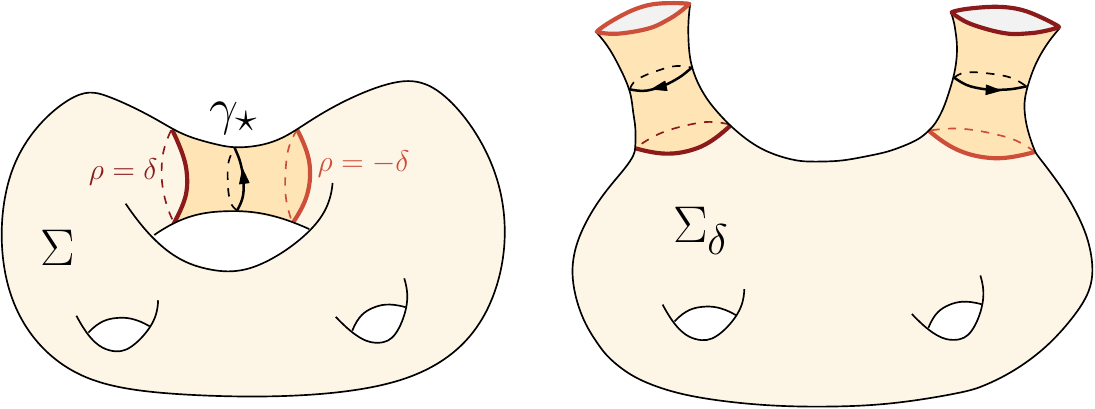}
\caption{The surfaces $\Sigma$ (on the left) and $\Sigma_\delta$ (on the right), in the case where $\gamma_\star$ is not separating. In $\Sigma$, the darker region corresponds to the neighborhood $\pi(U)$ of $\gamma_\star$.}
\label{fig:Sigma}
\medskip
\small{}
\end{figure}

\subsection{The resolvent of the geodesic vector field for open systems}\label{subsec:resolv}
In what follows, we denote by $\Omega^\bullet(M_\delta)$ the set of differential forms on $M_\delta$ and by $\Omega^\bullet_c(M_\delta)$ the elements of $\Omega^\bullet(M_\delta)$ whose support is contained in the interior of $M_\delta$. The set of currents on $M_\delta$, denoted by $\mathcal{D}'^\bullet(M_\delta)$, is defined as the dual of $\Omega_c^\bullet(M_\delta)$ with respect to the pairing
$$
(u ,v) = \int_{M_\delta} u \wedge v, \quad u,v \in \Omega^\bullet(M_\delta).
$$
The geodesic flow $\varphi$ on $M$ induces a flow on $M_\delta = S\Sigma_\delta$ which we still denote by $\varphi$. We define
$$
\ell_{\pm,\delta}(z) = \inf\{t>0,~\varphi_{\pm t}(z) \in \partial M_{\delta}\}, \quad z \in M_{\delta},
$$
the first exit times in the future and in the past. We also set
$$
\Gamma^{\pm}_{\delta} = \{z \in M_{\delta},~\ell_\mp(z) = +\infty\},\quad K_{\delta} = \Gamma^+_{\delta} \cap \Gamma^{-}_{\delta}
$$
and we define the operators $R_{\pm, \delta}(s)$ by
$$
R_{\pm,\delta}(s)\omega(z) = \pm \int_{0}^{\ell_{\mp,\delta}(z)} \varphi_{\mp t}^*\omega(z) \e^{-ts} \dd t, \quad z \in M_{\delta}, \quad \omega \in \Omega^\bullet_c(M_{\delta}),
$$
which are well defined whenever $\Re(s) \gg 1$ (note that our convention of $R_{\pm, \delta}(s)$ differs from \cite{guillarmou2017lens}). Then
$$
(\Lie_X \pm s) R_{\pm,\delta}(s) = \Id_{\Omega_c^\bullet(M_{\delta})},
$$
and for any $(u,v) \in \Omega^\bullet_c(M_{\delta} \setminus \Gamma_{-,\delta}) \times \Omega^\bullet_c(M_{\delta} \setminus \Gamma_{+,\delta})$ we have
$$
\int_{M_{\delta}} \left(R_{+,\delta}(s) u\right) \wedge v = - \int_{M_{\delta}} u \wedge R_{-,\delta}(s) v.
$$
Because the boundary of $\Sigma_{\delta}$ is strictly convex, it follows from \cite[Proposition 6.1]{dyatlov2016pollicott} that the family of operators $R_{\pm}(s)$ extends to a meromorphic family of operators
$$
R_{\pm,\delta}(s) : \Omega^\bullet_c(M_{\delta}) \to \mathcal{D}'^\bullet(M_{\delta}),
$$
satisfying
$$
\WF'(R_{\pm,\delta}(s)) \subset \Delta(T^*M_{\delta}) \cup \Upsilon^{\pm}_{\delta} \cup (E_{\pm, \delta}^* \times E_{\mp, \delta}^*),
$$
where $\Delta(T^*M_{\delta})$ is the diagonal in $ T^*M_\delta \times T^*M_\delta$,
$$
\Upsilon^\pm_{\delta} = \{(\Phi_t(z, \xi), (z, \xi)) \in T^*(M_{\delta} \times M_{\delta}),~  \pm t \geqslant 0,~\langle X(z), \xi \rangle = 0\}, 
$$
and
$$
\quad E^*_{+, \delta} = E_u^*|_{\Gamma^+_{\delta}}, \quad E_{-,\delta}^* = E_s^*|_{\Gamma^-_{\delta}}.
$$
Here, we denoted
$$
\WF'(R_{\pm, \delta}(s)) = \{(z, \xi, z', \xi') \in T^*(M_{\delta} \times M_{\delta}),~(z, \xi, z', -\xi') \in \WF(R_{\pm, \delta}(s))\},
$$
where $\WF$ is the classical H\"ormander wavefront set \cite{hor1}. Near any $s_0 \in \C$, we have the development
$$
R_{\pm, \delta}(s) = Y_{\pm, \delta}(s) + \sum_{j=1}^{J(s_0)} \frac{(X \pm s_0)^{j-1}\Pi_{\pm, \delta}(s_0)}{(s-s_0)^j},
$$
where $Y_{\pm, \delta}(s)$ is holomorphic near $s = s_0$, and $\Pi_{\pm,\delta}(s_0)$ is a finite rank projector satisfying
$$
\WF'(\Pi_{\pm, \delta}(s_0)) \subset E_{\pm, \delta}^* \times E_{\mp, \delta}^*, \quad \supp(\Pi_{\pm, \delta}(s_0)) \subset \Gamma_\delta^{\pm} \times \Gamma_\delta^{\mp},
$$
where we identified $\Pi_{\pm, \delta}(h)$ and its Schwartz kernel.

\subsection{Restriction of the resolvent on the geodesic boundary}

For $\varepsilon > 0$ we will use the slight abuse of notation $\varphi_{\pm \varepsilon}^* \equiv \varphi_{\pm \varepsilon}^* 1_{\{\ell_{\mp,\delta} > \varepsilon\}} : \Omega^\bullet(M_{\delta}) \to \Omega^\bullet(M_{\delta}).$ Let 
$$
\partial = \partial (S\Sigma_\star) = \{(x,v) \in M_\delta,~x \in \gamma_\star \sqcup \gamma_\star\},
$$
and $\partial_0 = S\gamma_\star \sqcup S\gamma_\star \subset \partial.$

\begin{lemm}\label{lem:wfset}
For any $\varepsilon > 0$ small enough, we have
$$
\WF(\varphi_{\mp \varepsilon}^*R_{\pm, \delta}(s)) \cap N^*(\partial \times \partial) = \emptyset,
$$
where 
$$N^*(\partial \times \partial) = \{(z',\xi',z,\xi) \in T^*(M_{\delta} \times M_{\delta}),~\langle \xi',~T_{z'}\partial\rangle = \langle \xi, T_z\partial \rangle = 0\}.$$
\end{lemm}
\begin{proof}
We prove the statement for $R_{+,\delta}(s)$. We have by the preceding subsection that
\begin{equation}\label{eq:wfsetepsilon}
\WF'(\varphi_{-\varepsilon}^*R_{+,\delta}(s)) \subset \Delta_\varepsilon \cup \Upsilon_\varepsilon \cup (E_{+,\delta}^* \times E_{-,\delta}^*),
\end{equation}
where
$$
\Delta_\varepsilon = \left\{(\Phi_\varepsilon(z, \xi), (z,\xi)),~(z, \xi) \in T^*M_{\delta}\right\}
$$
and
$$
\Upsilon_\varepsilon = \{(\Phi_t(z, \xi), (z, \xi)),~ t \geqslant \varepsilon,~\langle X(z), \xi \rangle = 0\}.
$$
Now assume that there is $\Xi = (z', \xi', z, \xi)$ lying in 
$$
N^*(\partial \times \partial) \cap \left(\Delta_\varepsilon \cup \Upsilon_\varepsilon \cup (E_{+,\delta}^* \times E_{-,\delta}^*)\right).
$$
If $\Xi \in \Delta_\varepsilon$, then necessarily we have $z,z' \in \partial_0$, because $\varphi_{\varepsilon}(\partial \setminus \partial_0) \cap \partial = \emptyset$ for $\varepsilon >0$ smaller than the injectivity radius, by negativeness of the curvature. We thus have $\xi \in N^*_z\partial = \R \beta(z)$ by Remark \ref{rem:nstar} ; now $\Phi_\varepsilon(\beta(z))$ does not lie in $\R\beta(\varphi_\varepsilon(z))$ by Lemma \ref{lem:phitbeta}, and therefore $\xi = 0$.

If $\Xi \in \Upsilon_\varepsilon$, then there is $T \geqslant \varepsilon$ such that $\Phi_T(z,\xi) = (z',\xi')$ with $\langle \xi, X(z) \rangle = 0$. However by Remark \ref{rem:nstar}, if $(z, \xi) \in N^*_z \partial$ and $\langle \xi, X(z) \rangle = 0$ then $z \in \partial_0$. Thus by what precedes, we obtain $\xi = 0$.
 
Finally, (\ref{eq:e^*}) and Remark \ref{rem:nstar} imply that $N^*\partial \cap E_{\pm,\delta}^* \subset \{0\}.$ Thus we showed $\WF'(\varphi_{\mp\varepsilon}^* R_{\pm, \delta}(s)) \cap N^*(\partial \times \partial) = \emptyset$, which is equivalent to the conclusion of the lemma (indeed, using $\WF$ or $\WF'$ does not matter here since we have $\{(z, \xi, z', \xi'),~(z,\xi, z', -\xi') \in N^*(\partial \times \partial)\} = N^*(\partial \times \partial)$).
\end{proof}

\begin{rem}\label{rem:wfset}
This estimate, combined with \cite[Theorem 8.2.4]{hor1}, implies that the operator $\iota^* \iota_X \varphi_{\mp \varepsilon}^*R_{+,\delta}(s)\iota_*$ is well defined and satisfies
$$
\WF\left(\iota^* \iota_X \varphi_{\mp \varepsilon}^*R_{+,\delta}(s)\iota_*\right) \subset \dd (\iota \times \iota)^{\top} \WF\left(\varphi_{\mp \varepsilon}^*R_{+,\delta}(s)\right)
$$
where $\iota : \partial \hookrightarrow M_\delta$ and $\iota \times \iota : \partial \times \partial \hookrightarrow M_\delta \times M_\delta$ are the inclusions.
\end{rem}

Here the pushforward $\iota_* : \Omega^\bullet(\partial) \to \mathcal{D}'^{\bullet + 1}(M_\delta)$ is defined as follows. If $u \in \Omega^k(\partial)$, we define the current $\iota_*u \in \mathcal{D}'^{k+1}(M_{\delta})$ by
$$
\langle \iota_*u, v \rangle= \int_{\partial} u\wedge \iota^*v, \quad v \in \Omega^{n-k-1}(M_{\delta}).
$$

\section{The scattering operator}\label{sec:scat}

In this section we introduce the dynamical scattering operator $\Sc_\pm(s)$ associated to our problem. By relating the scattering operator to the resolvent described above, we are able to compute its wavefront set. In consequence we obtain that the composition $(\chi \Sc_\pm(s))^n$ is well defined for $\chi \in C^\infty_c(\partial \setminus \partial_0)$, and we give a formula for its flat trace.

For each $x \in \partial \Sigma_\star$, let $\nu(x)$ be the normal outward pointing vector to the boundary of $\Sigma_\star$, and set
$$
\partial_\pm = \{(x,v) \in \partial \Sigma_\star,~\pm \langle \nu(x), v\rangle_g > 0\}.
$$

\subsection{First definitions}
For any $z \in M_\star = S\Sigma_\star$ we define the exit time of $z$ in the future and past by
$
\ell_\pm(z) = \inf \{t>0,~\varphi_{\pm t}(z) \in \partial\}
$
and we set
$$
\Gamma_\pm = \{z \in M,~\ell_\mp(z) = +\infty\}.
$$
Then $\Gamma_+$ (resp. $\Gamma_-)$ is the set of points of $M$ which are trapped in the past (resp. in the future). The scattering map $S_\pm : \partial_\mp \setminus \Gamma_\mp \to \partial_\pm \setminus \Gamma_\pm$ is defined by
$$
S_\pm(z) = \varphi_{\pm \ell_{\pm} (z)}(z), \quad z \in \partial_\mp \setminus \Gamma_\mp,
$$
and satisfies
$
S_\pm \circ S_\mp = \Id_{\partial_{\pm} \setminus \Gamma_\pm}.
$
For $s \in \C$, the scattering operator 
$$\mathcal{S}_\pm(s) : \Omega_c^\bullet(\partial_\mp \setminus \Gamma_\mp) \to \Omega^\bullet_c(\partial_\pm \setminus \Gamma_\pm)$$ is given by
$$
\Sc_\pm(s)\omega = (S_\mp^*\omega) \e^{- s \ell_\pm(\cdot)}, \quad \omega \in \Omega_c^\bullet(\partial_\mp \setminus \Gamma_\mp).
$$

\begin{rem}\label{rem:scatcontinuous}
If $\Re(s)$ is big enough, $\mathcal{S}_\pm(s)$ extends as a map $\Omega^\bullet(\partial) \to C^0(\partial, \bigwedge^\bullet T^*\partial)$ (here $C^0(\partial, \bigwedge^\bullet T^*\partial)$ is the space of continuous forms on $\partial$), by declaring that 
$
\mathcal{S}_\pm(s)\omega(z) = S_\mp^*\omega(z) \e^{- s \ell_\pm(z)}
$
if $z \in \partial_\pm \setminus \Gamma_\pm$ and $\mathcal{S}_\pm(s)\omega(z) = 0$ otherwise. Indeed, this follows from the fact that there is $C > 0$ such that
$$
\|\varphi_t^*\omega\|_{\infty} \leqslant \e^{C|t|} \|\omega\|_{\infty}, \quad t \in \R, \quad \omega \in \Omega^\bullet(M),
$$
where $\|\omega\|_\infty$ is the uniform norm on $C^0(M, \bigwedge^\bullet T^*M).$
\end{rem}

\subsection{The scattering operator via the resolvent}
In this paragraph we will see that $\mathcal{S}_\pm(s)$ can be computed in terms of the resolvent. More precisely, we have the following result.

\begin{prop}\label{prop:scatresolv}
For any $\Re(s)$ large enough we have 
$$
\mathcal{S}_{\pm}(s) = (-1)^N \e^{\pm \varepsilon s}\iota^* \iota_X \varphi_{\mp \varepsilon}^*R_{\pm,\delta}(s)\iota_*
$$
as maps $\Omega^\bullet(\partial) \to \mathcal{D}'^\bullet(\partial)$, where $N : \Omega^\bullet(\partial) \to \N$ is the degree operator, that is, $N(w) = k$ if $w$ is a $k$-form.
\end{prop}

An immediate consequence is the

\begin{corr}\label{cor:mero}
The scattering operator $\mapsto \mathcal{S}_\pm(s) : \Omega^\bullet(\partial) \to \mathcal{D}'^\bullet(\partial)$ extends as a meromorphic family of $s \in \C$ with poles of finite rank, with poles contained in the set of Pollicott-Ruelle resonances of $\Lie_X$, that is, the set of poles of $s \mapsto R_{\pm, \delta}(s)$.
\end{corr}

Before proving Proposition \ref{prop:scatresolv}, we start by an intermediate result.

\begin{lemm} \label{lem:scatresolv}
We have
$
\mathcal{S}_{\pm}(s) = (-1)^N\e^{\pm \varepsilon s}\iota^* \iota_X \varphi_{\mp \varepsilon}^*R_{\pm,\delta}(s)\iota_*
$
 as maps $\Omega^\bullet_c(\partial_{\mp} \setminus \Gamma_\mp) \to \mathcal{D}'^\bullet(\partial_{\pm} \setminus \Gamma_\pm).$
\end{lemm}
\begin{rem}
Note that Proposition \ref{prop:scatresolv} is not a direct consequence of Lemma \ref{lem:scatresolv}. Indeed, the operator $\mathcal{Q}_{\varepsilon, \pm} = \e^{\pm \varepsilon s}\iota^* \iota_X \varphi_{\mp \varepsilon}^*R_{+,\delta}(s)\iota_*$ could hide some singularities near $\Gamma_\pm$ ; Proposition \ref{prop:scatresolv} tells us that is it not the case, at least for $\Re(s)$ large enough.
\end{rem}
\begin{proof}
Let $u \in \Omega^\bullet_c\left(\partial_- \setminus \Gamma_-\right),$ and $U' \subset \partial_-$ be a neighborhood of $\supp u$ such that $\overline U'$ does not intersect $\partial_0$. Let $\varepsilon > 0$ small such that
$$
z \in \partial_- \implies \ell_+(z) > \varepsilon.
$$
The existence of such an $\varepsilon$ follows from the negativeness of the curvature. Let
$$
U = \left\{(t,z) \in \R \times U',~-\ell_{-,\delta}(z) < t < \varepsilon\right\}.
$$
Then $U$ is diffeomorphic to a tubular neighborhood of $U'$ in $M_{\delta}$ via $(t,z) \mapsto \varphi_t(z).$ Let $\chi \in C^\infty(\mathbb{R})$ such that $\chi \equiv 1$ near $(-\infty, 0]$ and $\chi \equiv 0$ on $(\varepsilon / 2, +\infty)$. Set, in the above coordinates,
$$
\psi(t,z) = \chi(t)\e^{-ts}u(z) \in {\bigwedge}^\bullet T_{(t,z)}^*{M_{\delta}},
$$
where we see $u(z)$ as a form on $T^*_{(t,z)}M$ by declaring $\iota_{\partial_t} u(z) = 0$. We extend $\psi$ by $0$ on $M$ and we set
$$
\phi= \psi - R_{+,\delta}(s)(\Lie_X + s) \psi.
$$
Then $\phi$ is smooth (since $\supp \psi \cap \Gamma_- = \emptyset$) and $(\Lie_X + s)\phi= 0$, and we have
$$
\phi|_{\partial_-} = u, \quad \phi|_{\partial_+} = \mathcal{S}_{+}(s)u,
$$
where $\mathcal{S}_{+}(s) = \mathcal{S}_+(s)|_{\Omega^\bullet_c(\partial_{-} \setminus \Gamma_-)}.$ Let $h \in \Omega^\bullet_c(M_{ \delta} \setminus \Gamma_+)$, so that $R_{-,\delta}(s)h$ is smooth. We have
$$
\begin{aligned}
\int_{M_{\delta}} \phi\wedge h &= \int_{M_{\delta}} \psi \wedge h - \int_{M_{\delta}}R_{+,  \delta}(s)(\Lie_X + s)\psi \wedge h \\
&= \int_{M_{\delta}} \psi \wedge h + \int_{M_{\delta}} (\Lie_X + s) \psi \wedge R_{-,\delta}(s) h \\
&= \int_{M_{\delta}} \psi \wedge h - \int_{M_{\delta}} \psi \wedge (\Lie_X - s)R_{-,\delta}(s) h + \int_{\partial M_{ \delta}} \iota_X \left(\psi \wedge R_{-,\delta}(s)h \right) \\
&= \int_{\partial M_{ \delta}} \iota_X \left(\psi \wedge R_{-,\delta}(s)h \right) \\
&=  (-1)^{\deg \psi}\int_{\partial_{-,\delta}} \psi \wedge \iota_X R_{-,\delta}(s)h,
\end{aligned}
$$
since $\iota_X \psi = 0$ and $\psi$ has no support near $\partial_{+, \delta}$. Now we let $\Phi : \partial_{-} \to \partial_{-,\delta}$ be defined by $\Phi(z) = \varphi_{-\ell_{-,\delta}(z)}(z)$. Assume that the support of $h$ does not intersect $U$. Then a change of variable gives
$$
\Phi^*(\iota_X R_{-,\delta}(s) h)|_{\partial_{-, \delta}} = \iota_XR_{-,\delta}(s)h \e^{-s\ell_{-,\delta}(\cdot)},
$$
As we have $\Phi^*(\psi|_{\partial_{-,\delta}}) = (\psi|_{\partial_{-}}) \e^{+s\ell_{-,\delta}(\cdot)} = u \e^{+s\ell_{-,\delta}(\cdot)}$ by definition of $\psi$, we obtain
$$
\int_{M_{\delta}} \phi\wedge h = (-1)^{\deg u} \int_{\partial_{-}} u \wedge \iota^*(\iota_X R_{-,\delta}(s) h).
$$
Now because $(\Lie_X-s)R_{-, \delta}(s) h = h$, we get $(\Lie_X -s)R_{-,\delta}(s)h = 0$ near $U$ and thus $\varphi_{\varepsilon}^*R_{-,\delta}(s)h = \e^{\varepsilon s} R_{-,\delta}(s) h$ near $U$. Let $v \in \Omega^\bullet_c(\partial_{+} \setminus \Gamma_+)$ and $h_n \in \Omega^\bullet_c(M_{\delta}\setminus \Gamma_+), ~n \in \N,$ with $\supp h_n \cap U = \emptyset$, and such that $h_n \to \iota_*v$ in $\mathcal{D}'^\bullet(M_{\delta}).$ Then 
$$
\int_{\partial_{+}} (\mathcal{S}_{+}(s) u) \wedge v = (-1)^{\deg u} \e^{-\varepsilon s} \int_{\partial_{-}} u \wedge \iota^* \iota_X \varphi_{\varepsilon}^* R_{-,\delta}(s)\iota_*v,
$$
because $\phi|_{\partial_{+}} = \mathcal{S}_{+}(s)u.$
Since $\int_{\partial_{+}} \mathcal{S}_{+}(s) u \wedge v = \int_{\partial_{-}} u \wedge \mathcal{S}_{-}(s) v$, we obtain
$$
\mathcal{S}_{-}(s)= (-1)^{\deg u} \e^{-\varepsilon s}\iota^* \iota_X \varphi_{\varepsilon}^* R_{-,\delta}(s)\iota_*
$$
as maps $\Omega^\bullet_c(\partial_+ \setminus \Gamma_+) \to \Omega^\bullet_c(\partial_- \setminus \Gamma_-)$. We can replace $X$ by $-X$ to obtain the desired formula for $\mathcal{S}_{+}(s)$, which concludes.
\end{proof}

\begin{proof}[Proof of Proposition \ref{prop:scatresolv}]

We prove the proposition for $C^\infty(\partial)$, the proof for $\Omega^\bullet(\partial)$ being the same. Let $u \in C^\infty(\partial)$ and write $u = u(\tau, \theta)$. Let $\chi \in C^\infty_c(\R, [0,1])$ such that $\int_\R\chi = 1$, $\chi(0) \neq 0$, $\chi \equiv 0$ on $\R \setminus (-\delta/2, \delta/2)$, and $\chi > 0$ on $(-\delta/2, \delta/2)$. For $n \in \N_{\geqslant 1}$ we set $\chi_n = n\chi(n\cdot)$, so that $\chi_n$ converges to the Dirac measure on $\R$ as $n \to +\infty$. We define $u_n \in \Omega^1_c(M_{\delta})$ in the $(\tau, \rho, \theta)$ coordinates by
$$
u_n= u(\tau, \theta) \chi_n(\rho) \dd \rho.
$$
Then $u_n \to \iota_*u$ in $\mathcal{D}'(M_{\delta})$. Consider
$$
f_n = \iota^* \varphi_{-\varepsilon}^* \iota_X R_{+,\delta}(s) u_n, \quad n \geqslant 1.
$$
Then for $\Re(s)$ big enough, we have $f_n \in C^0(\partial)$ for any $n \in \N.$ For $z \in \partial$ let 
$$
\tilde u(z) = \left\{\begin{matrix} u(S_-(z))\e^{-s\ell_-(z)} &\text{ if } &z \in \partial_+ \setminus \Gamma_+, \\ 0 &\text{ if not.}& \end{matrix} \right.
$$
Then $\tilde u$ is continuous and we claim that $f_n \to \tilde u$ in $\mathcal{D}'^0(\partial)$ when $n \to +\infty.$ Indeed, notice that
$$
\iota_X u_n(\tau, \rho, \theta) = u(\tau, \theta) \chi_n(\rho) (X\rho)(\tau, \rho, \theta).
$$
Let $F = \{\rho \leqslant \delta/2\}$. Since the neighborhood $\{|\rho| < \delta / 2\}$ is strictly convex, there exists $L>0$ such that for any $z \in F$ and $T > 0$ such that $\varphi_{-T}(z) \in F$, we have
\begin{equation}\label{eq:L}
\Bigl(\forall t \in (0,T),~\varphi_{-t}(z)\notin F\Bigr) \quad \implies \quad T \geqslant L.
\end{equation}
Now take $z \in \partial \setminus \Gamma_-$. Then the set $\{t \in [\varepsilon, \ell_{+, \delta}(z)],~\varphi_t(z) \in F\}$ is a finite union of closed intervals, say
$$
\{t \geqslant \varepsilon,~\varphi_t(z) \in F\} = \bigcup_{k=0}^{K(z)} [a_k(z), b_k(z)],
$$
with $a_k(z) \leqslant b_k(z)$ and $b_k(z) < a_{k+1}(z)$ for every $k$.
We set $\rho(t) = \rho(\varphi_{-t}(z))$ for any $t \geqslant 0$ ; we have
$$
\begin{aligned}
|f_n(z)| &\leqslant \|u\|_{\infty} \int_{\varepsilon}^{\ell_{-,\delta}} \left|(\chi_n\circ \rho)(\varphi_{-t}(z))\right| \left|(X\rho)(\varphi_{-t}(z))\right| \e^{-ts} \dd t
\\ 
& \leqslant \|u\|_\infty \sum_{k=0}^{K(z)} \e^{-sa_k(z)} \int_{a_k(z)}^{b_k(z)} \chi_n(\rho(t)) |\rho'(t)| \dd t
\end{aligned}
$$
Looking at the geodesic equation for the metric (\ref{eq:metric}), we see that $\pm X^2\rho > 0$ if $\pm \rho > 0$; thus we may separate each interval $[a_{k}(z), b_k(z)]$ into two subintervals on which $|\rho'| > 0$ and change variables to get 
$$
\int_{a_k(z)}^{b_k(z)} \chi_n(\rho(t)) |\rho'(t)| \dd t \leqslant 2\int_\R \chi_n(\rho) \dd \rho \leqslant 2.
$$
By (\ref{eq:L}), we have $a_k(z) \geqslant kL$ for any $k$. Therefore we obtain, since each $f_n$ is continuous,
$$
|f_n(z)| \leqslant \displaystyle{\frac{2\|u\|_\infty}{1-\e^{-sL}}}, \quad z \in \partial, \quad n \geqslant 1.
$$
For any $z \in \partial_-$ we have $\{\varphi_t(z),~ t\geqslant \varepsilon\} \cap \partial_- = \emptyset$ by negativeness of the curvature. Thus $f_n(z) \to 0$ as $n \to +\infty$ for any $z \in \partial_-$, and by dominated convergence we have
$$
\int_{\partial_{-}} f_n  v \to 0, \quad n \to \infty, \quad v \in \Omega^\bullet(\partial).
$$
for any $v \in \Omega^\bullet(\partial)$. Now let $\eta  > 0$. Let $\chi_\pm \in C^\infty_c(\partial_{\pm} \setminus \Gamma_\pm)$ such that $\chi_- \equiv 1$ on $\supp (\chi_+ \circ S_+)$, and $\mathrm{vol}(\supp(1-\chi_+)) < \eta$. Such functions exist as $\mathrm{Leb}(\Gamma_+ \cap \partial) = 0$, cf. \cite[\S2.4]{guillarmou2017lens}\footnote{Actually \cite{guillarmou2017lens} implies $\mathrm{Leb}(\Gamma_{+, \delta} \cap \partial_{+, \delta}) = 0$. However $\partial_+$ is diffeomorphic to $\partial_{+, \delta}$ via $z \mapsto \varphi_{\ell_{+, \delta}(z)}(z)$, and this map sends $\Gamma_+ \cap \partial$ on $\Gamma_{+, \delta} \cap \partial_{+,\delta}$.}. We have
$$
\int_{\partial_{+}} f_n v = \int_{\partial_{+}} f_n \chi_+ v + \int_{\partial_{+}} f_n (1-\chi_+)v.
$$
Thus, up to replacing $u$ by $\chi_-u$, we have by Lemma \ref{lem:scatresolv}
$$
\int_{\partial_{+}} f_n \chi_+ v \to \int_{\partial_{+}} \tilde u \chi_+ v.
$$
By what precedes there is $C > 0$ such that for any $n \geqslant 1$
$$
\left|\int_{\partial_{+}} \tilde u(1-\chi_+)v\right| < C \eta, \quad \left|\int_{\partial_{+}} f_n(1-\chi_+)v\right| < C\eta.
$$
Summarizing the above facts, we obtain that for $n \geqslant 1$ big enough, one has
$$
\left| \int_{\partial} f_n v - \int_\partial \tilde u v \right| \leqslant 4 C \eta.
$$
Thus $f_n \to \tilde u$ in $\mathcal{D}'^0(\partial)$, which concludes the proof.
\end{proof}

\subsection{Composing the scattering maps}

Recall that $\partial$ has two connected components $\partial^{(1)}$ and $\partial^{(2)}$ that we can identify in a natural way. We denote by $\psi : \partial \to \partial$ the map exchanging those components via this identification (in particular $\psi (\partial_\pm) = \partial_\mp$), and we set 
$$\tilde \Sc_\pm(s) = \psi^* \circ \Sc_\pm(s).$$
Also we denote by $\Psi = T^*\partial \to T^*\partial$ the symplectic lift of $\psi$ to $T^*\partial,$ that is 
$$\Psi(z,\xi) = (\psi(z), \dd \psi_z^{-\top} \xi), \quad (z, \xi) \in T^*\partial.$$

\begin{lemm}\label{lem:composition}
Let $\chi \in C^\infty_c(\partial \setminus \partial_0)$. Then for any $n \geqslant 1$, the composition $\left(\chi \tilde \Sc_\pm(s)\right)^n : \Omega^\bullet(\partial) \to \mathcal{D}'^\bullet(\partial)$ is well defined.
\end{lemm}

\begin{proof}
We first prove the lemma for $n = 2$. According to \cite[Theorem 8.2.14]{hor1}, it suffices to show that
\begin{equation}\label{eq:intersection}
\begin{aligned}
\{(z,\xi),~&\exists z' \in \partial,~(z',0,z,\xi) \in \WF'(\chi \tilde \Sc_{\pm}(s))\}  \\
&\cap \{(z,\xi),~\exists z' \in \partial,~(z,\xi,z',0) \in \WF(\chi \tilde \Sc_{\pm}(s))\} = \emptyset.
\end{aligned}
\end{equation}
We have
$$
\WF(\Sc_\pm(s)) \subset \dd (\iota \times \iota)^\top \left(\Delta_\varepsilon \cup \Upsilon_\varepsilon \cup (E_{+, \delta}^* \times E_{-, \delta}^*)\right),
$$
where $\Delta_\varepsilon, \Upsilon_\varepsilon$ are defined in the proof of Lemma \ref{lem:wfset}.
As $\chi$ is supported far from $\partial_0$, we have $(\varphi_\varepsilon(z'), z') \notin \partial \times \partial$ for any $z' \in \supp \chi$, and for any $\eta \in T^*_{z'}M_\delta$ such that $\langle X(z'), \eta \rangle = 0$, we have
\begin{equation}\label{eq:diinj}
\dd \iota^\top(z',\eta) = 0 \implies \eta = 0.
\end{equation}
This implies that the first term (denoted $A$) of the intersection (\ref{eq:intersection}) is contained in $E_{\mp,\partial}^*$
while the second term (denoted $B_1$) is contained in $\Psi(E_{\pm,\partial}^*)$, where $E_{\pm, \partial}^* = (\dd \iota)^\top (E_{\pm, \delta}^*)$. Now we claim that $\Psi(E_{\pm, \partial}^*) \cap E_{\mp, \partial}^* \subset \{0\}$ far from $\partial_0$. By Lemma \ref{lem:coordinates} and \S\ref{subsec:theanosovproperty} one has, for any $z = (\tau, 0, \theta) \in \partial^{(j)} \cap \Gamma_\pm$, 
$$
E_{\pm, \partial}^*(z) = \R (\dd \iota)^{\top}_z (r_\pm(z) \beta(z) - \psi(z)) = \R(- \sin(\theta) r_\pm(z) \dd \tau - \dd \theta),
$$
since $\iota(z,\theta) = (z, 0, \theta)$. Now we claim that $r_\mp(\psi(z)) \neq r_\pm(z)$ for all $z$. Indeed, the contrary would mean that $E_s(z') \cap E_u(z') \neq \{0\}$ for some $z' \in M$ (represented by both $z$ and $\psi(z)$ in $M_\delta$), which is not possible. Now we have $\sin(\theta) \neq 0$ for $z \notin \partial_0$. As a consequence (\ref{eq:intersection}) is true, since $\supp \chi \cap \partial_0 = \emptyset$. This concludes the case $n = 2$. 

By \cite[Theorem 8.2.14]{hor1} we also have the bound
$$
\WF\bigl((\chi\tilde \Sc_\pm(s))^2\bigr) \subset \left(\WF'(\chi \tilde \Sc_\pm(s)) \circ \WF'(\chi \tilde \Sc_\pm(s))\right)  \cup (B_1 \times \underline 0) \cup (\underline 0 \times A),
$$
where $\underline 0$ denote the zero section in $T^*\partial.$ This formula gives that the set $B_2$ defined by
$$
B_2 = \left\{(z,\xi),~\exists z' \in \partial,~(z,\xi,z',0) \in \WF\bigl((\chi\tilde \Sc_{\pm}(s)^2)\bigr)\right\}
$$
is equal to 
$$
\begin{aligned}
\Bigl\{(z,\xi)\in T^*\partial,~\exists z', z'' \in \partial,~(z, \xi, &z', -\eta) \in \WF(\chi \tilde \Sc_\pm(s)) \\
&\text{ and } (z', \eta, z'', 0) \in \WF(\chi \tilde \Sc_\pm(s))\Bigr\} \cup B_1.
\end{aligned}
$$
Since $\Psi(E_{\pm,\partial}^*) \cap E_{\mp, \partial}^* \subset \{0\}$ we obtain
$$
B_2 = \{(z, \xi),~(z, \xi, z', \eta) \in \dd(\iota \times \iota)^\top(\Upsilon_\varepsilon) \text{ for some }\eta \in \Psi(E_{\pm, \partial}^*)\} \cup B_1.
$$
Finally, we get, by definition of $\Upsilon_\varepsilon$,
$$
B_2 = \{\Psi \circ (\dd \iota)^{\top} (\Phi_{t}(z, \zeta)),~\langle X(z), \zeta\rangle = 0,~\dd \iota^\top(z, \zeta) \in \Psi(E_{\pm, \partial}^*),~\varphi_t(z) \in \partial,~t \geqslant \varepsilon\} \cup B_1.
$$
By (\ref{eq:diinj}), if $\langle X(z), \eta \rangle = 0$ and $\dd \iota^\top(z, \zeta) \in \Psi(E_{\pm, \partial}^*)$, then $(z, \zeta) \in \Psi(E_{\pm, \delta}^*)$ (of course if $z \in \supp \chi$).
This implies that $B_2$ can intersect $E_{\mp, \partial}^*$ only in a trivial way. Indeed, for any $t \geqslant \varepsilon$ and $(z, \zeta) \in \Psi(E_{\pm, \delta}^*)$ such that $\varphi_t(z) \in \partial$, we have $\Psi(\Phi_t(z, \zeta)) \notin E_{\mp, \partial}^* \setminus \{0\}$, since as before it would mean that $E^u(z') \cap E^s(z') \neq \{0\}$ for $z' \in M$ representing both $\varphi_t(z)$ and $\psi(\varphi_t(z))$. Thus $A \cap B_2 = \emptyset$, which shows that $\bigl(\chi\tilde \Sc_\pm(s)\bigr)^3$ is well defined. We may iterate this process to obtain that $(\chi\Sc_\pm(s))^n$ is well defined for every $n \geqslant 1.$ 

\end{proof}

\subsection{The flat trace of the scattering operator}\label{subsec:flattrace}

Let $\mathcal{A} : \Omega^\bullet(\partial) \to \mathcal{D}'^\bullet(\partial)$ be an operator such that
$
\WF'(\mathcal{A}) \cap \Delta = \emptyset,
$
where $\Delta$ is the diagonal in $T^*(\partial \times \partial)$.
Then the flat trace of $\mathcal{A}$ is defined as
$$
\strf \mathcal{A} = \langle \iota^*_\Delta A, 1 \rangle,
$$
where $\iota_\Delta : z \mapsto (z,z)$ is the diagonal inclusion and $A \in \mathcal{D}'^{n}(\partial \times \partial)$ is the Schwartz kernel of $\mathcal{A}$, i.e.
$$
\int_\partial \mathcal{A}(u) \wedge v = \int_{\partial \times \partial} A \wedge \pi_1^*u \wedge \pi_2^*v, \quad u,v \in \Omega^\bullet(\partial),
$$
where $\pi_j : \partial \times \partial \to \partial$ is the projection on the $j$-th factor ($j=1,2$). In fact we have
\begin{equation}\label{eq:alternatedtrace}
\strf (A) = \sum_{k=0}^2 (-1)^{k+1} \mathrm{tr}^\flat(A_k),
\end{equation}
where $\mathrm{tr}^\flat$ is the transversal trace of Attiyah-Bott \cite{atiyah1967lefschetz} and $A_k$ is the operator $C^\infty\bigl(\partial, \wedge^kT^*\partial\bigr) \to \mathcal{D}'\bigl(\partial, \wedge^k T^*\partial\bigr)$ induced by $A$ on the space of $k$-forms.

The purpose of this section is to compute the flat trace of $\mathcal{S}_\pm(s)$. In what follows, for any closed geodesic $\gamma : \R/\Z \to \Sigma$, we will denote 
$$
I_\star(\gamma) = \{z \in S\Sigma|_{\gamma_\star}, z = (\gamma(\tau), \dot \gamma(\tau)) \text{ for some }\tau \in \R/\Z\}
$$
 the set of incidence vectors of $\gamma$ along $\gamma_\star$, and
$$
 I_{\star, \pm}(\gamma) = p_\star^{-1}(I_\star(\gamma)) \cap \partial_\mp
$$
where $p_\star : S\Sigma_\star \to S\Sigma$ is the natural projection.

 \begin{prop}\label{prop:computeflattrace}
Let $\chi \in C^\infty_c(\partial \setminus \partial_0)$. For any $n \geqslant 1$, the operator $(\chi\tilde{\mathcal{S}}_\pm(s))^n$ has a well defined flat trace and for $\Re(s)$ big enough we have
 \begin{equation}\label{eq:strf}
 \strf \left(\chi \tilde \Sc_\pm(s)\right)^n = n \sum_{i(\gamma) = n} \frac{\ell^\#(\gamma)}{\ell(\gamma)} \e^{-s\ell(\gamma)} \left(\prod_{z \in I_{\star, \pm}(\gamma)} \chi(z)\right)^{\ell(\gamma) / \ell^\#(\gamma)}, 
 \end{equation}
 where the sum runs over all closed geodesics $\gamma$ of $(\Sigma, g)$ (not necessarily primitive) such that $i(\gamma, \gamma_\star) = n$. Here $\ell(\gamma)$ is the length of $\gamma$ and $\ell^\#(\gamma)$ its primitive length.
 \end{prop}
 
\begin{proof}
The proof that the intersection
$
\WF'\bigl((\chi\tilde \Sc_\pm(s))^n\bigr) \cap \Delta
$
is empty is very similar to the arguments we already gave, for example in Lemma \ref{lem:composition}. Since it might be repetitive, we shall omit it.

For any $n \geqslant 1$ we define the set $\tilde \Gamma_n^\pm \subset \partial$ by
$$
\complement \tilde \Gamma_n^\pm = \{z \in \partial, ~\tilde S_{\pm}^k(z) \text{ is well defined for }k=1, \dots, n\},
$$
where $\tilde S = \psi \circ S$. Also we set
$$
\tilde \ell_{\pm,n}(z) =  \ell_\pm(z) + \ell_\pm(\tilde S_\pm(z)) + \cdots + \ell_\pm({\tilde S_{\pm}^{n-1}(z)}), \quad z \in \complement \tilde \Gamma_n^\pm,
$$
where $\ell_\pm(z) = \inf\{t > 0,~\varphi_{\pm t}(z) \in \partial\}$, with the convention that $\tilde \ell_{\pm,n}(z) = +\infty$ if $z \in \tilde \Gamma_n^\pm.$ We will need the following

\begin{lemm}\label{lem:estonell}
Let $n\geqslant 1$. For any $k \geqslant 1$, there exists $C_k > 0$ such that 
$$
\|\dd^k \ell_{\pm,n}(z) \| \leqslant C_k \exp(C_k \ell_{\pm,n}(z)), \quad z \in \complement \tilde \Gamma_n^\pm.
$$
\end{lemm}

\begin{proof}
In what follows, $C_k$ is a constant depending only on $k$, which may change at each line. First, notice that 
$
\|\dd^k \varphi_t(z)\| \leqslant C_k \e^{C_k |t|}
$
for any $t \in \R$ and $z \in M_\delta$ such that $\varphi_t(z) \in M_\delta$, for some constant $C_k$ (see for example \cite[Proposition A.4.1]{bonthonneau2015resonances}). Moreover, we have
$$
\dd S_\pm(z) = \dd [\varphi_{\ell_\pm(z)}](z) + X(S_\pm(z)) \dd\ell_\pm(z), \quad z \notin \tilde \Gamma_1^\pm.
$$
By induction we obtain that for any $k$
\begin{equation}\label{eq:dk}
\|\dd^k S_\pm(z)\| \leqslant C_k \exp(C_k \ell_\pm(z)) + C_k\sum_{j=1}^k \|\dd^j \ell_{\pm}(z)\|^{m_j}, \quad m_j \in \N, \quad j = 1, \dots, k,
\end{equation}
for any $z \notin \tilde \Gamma_1^\pm$. This inequality, combined with the fact that $S_\pm(\complement \tilde \Gamma_k^\pm) = \complement \Gamma_{k-1}^\pm,$ implies that to prove the lemma it suffices to show the estimate
\begin{equation}\label{eq:wanted}
\|\dd^k \ell_\pm(z)\| \leqslant C_k \exp(C_k \ell_\pm(z)), \quad z \notin \tilde \Gamma_1^\pm.
\end{equation}
Let $(\rho, \theta, \tau)$ be the coordinates defined near $\partial$ given by Lemma \ref{lem:coordinates}. Then $\rho(S_\pm(z)) = 0$ for $z \in \tilde \Gamma^\pm_1$ and thus
\begin{equation}\label{eq:3.5}
(X\rho)(S_\pm(z)) \dd \ell_\pm(z) = - \dd \rho(S_\pm(z)) \circ \dd[\varphi_{\ell_\pm(z)}](z), \quad z \notin \tilde \Gamma_1^\pm.
\end{equation}
Now Lemma \ref{lem:coordinates} gives
$
(X\rho)(S_\pm(z)) = \sin \bigl(\theta(S_\pm(z))\bigr).
$
As the curvature is negative, we see from Topogonov's comparison theorem \cite[Theorem 73]{Berger_2003} (and classical trigonometric identities for hyperbolic triangles) we see that for some constant $C$ we have
\begin{equation}\label{eq:sinus}
\left|\sin\bigl(\theta(S_\pm(z))\bigr)\right| \geqslant C \exp(-C\ell_{\pm}(z)), \quad z \notin \tilde \Gamma_1^\pm.
\end{equation}
Therefore, we obtain for any $z \in \tilde \Gamma^\pm_1$,
$$
\begin{aligned}
\|\dd \ell_\pm(z)\| &\leqslant C^{-1} \exp(C\ell_\pm(z))\| \dd \rho(S_\pm(z)) \| \| \dd[\varphi_{\ell_\pm(z)}](z)\| \\
&\leqslant C \e^{C\ell_\pm(z)}.
\end{aligned}
$$
Now, using repetively (\ref{eq:dk}), (\ref{eq:3.5}) and (\ref{eq:sinus}), we obtain (\ref{eq:wanted}) by induction.
\end{proof}

Consider $\tilde \chi \in C^\infty(\R, [0,1])$ such that $\tilde \chi \equiv 1$ on $[0,1]$ and $\tilde \chi \equiv 0$ on $[2, +\infty[$, and set $\tilde \chi_L(z) = \tilde \chi(\ell_{\pm,n}(z)-L)$ for $z \in \partial$. Then $\tilde \chi_L \in C^\infty_c(\partial \setminus \tilde \Gamma_n^\pm)$ and by (\ref{eq:alternatedtrace}) we see that the Atiyah-Bott trace formula \cite[Corollary 5.4]{atiyah1967lefschetz} reads in our case
 \begin{equation}\label{eq:atiyah}
 \langle \iota_\Delta^*K_{\chi, \pm,n}(s), \tilde \chi_L \rangle 
  = \sum_{(\tilde S_\mp)^n(z) = z} \e^{-s\ell_{\pm,n}(z)} \tilde \chi_L(z) \prod_{k=0}^{n-1}\chi(\tilde S_\mp^k(z)),
 \end{equation}
where $K_{\chi, \pm, n}(s)$ is the Schwartz kernel of $(\chi\Sc_\pm(s))^n$. Indeed, it is proven in \cite{atiyah1967lefschetz} that for any diffeomorphism $f : \partial \to \partial$ with isolated nondegenerate fixed points, it holds
$$
\mathrm{tr}^\flat(F_k) = \sum_{f(z) = z} \frac{\tr \wedge^k \dd f(z)}{|\mathrm{det}(1-\dd f(z))|}
$$
where $F_k : \Omega^k(\partial) \to \Omega^k(\partial)$ is defined by $F_k\omega = f^*\omega$ and $\wedge^k \dd f(z)$ is the map induced by $\dd f(z)$ on $\wedge^kT^*_z\partial.$ Since $\sum_k (-1)^k \mathrm{tr}(\wedge^k \dd f(z)) = \det(1 - \dd f(z))$ it holds 
\begin{equation}\label{eq:strftotrf}
\strf(F) = \sum_k (-1)^{k+1} \mathrm{tr}^\flat(F_k) = \sum_{f(z) = z} \sgn\mathrm{det}(1 - \dd f(z)).
\end{equation}
Now note that $\tilde \chi_L(\chi\tilde \Sc_\pm(s))^n$ is by definition the operator given by
\begin{equation}\label{eq:operator}
\omega~ \mapsto~ \tilde \chi_L(\cdot) \left(\prod_{k=0}^{n-1} \chi\left({\tilde S_\mp}^k(\cdot)\right) \right) \e^{-s \ell_{\pm,n}(\cdot)}\left({\tilde S_\mp}^n\right)^*w.
\end{equation}
Moreover, $\sgn \mathrm{det}\left(1- \dd \tilde S_\mp^n(z)\right) = -1$ for any $z$ such that $\tilde S_\mp^n(z) = z$. Indeed, for such a $z$, $\dd S_\mp^n(z)$ is conjugated to the linearized Poincar\'e map $P_z = \dd (\varphi_{\ell_{\pm,n}(z)})(z)|_{E^u(z) \oplus E^s(z)}$, which satisfies $\det(1-P_z) < 0$ as the matrix of $P_z$ in the decomposition $E^u(z) \oplus E^s(z)$ reads $\begin{pmatrix} \lambda & 0 \\ 0 & \lambda^{-1} \end{pmatrix}$ for some $\lambda > 1$ (since $\varphi_t$ preserves the volume form $\alpha \wedge \dd \alpha$). Thus (\ref{eq:strftotrf}) and (\ref{eq:operator}) imply (\ref{eq:atiyah}).
 
As $L \to +\infty$, the right hand side of (\ref{eq:atiyah}) converges to
 $$
 n \sum_{i(\gamma, \gamma_\star) = n}\frac{\ell^\#(\gamma)}{\ell(\gamma)} \e^{-s\ell(\gamma)} \left(\prod_{z \in I_{\star, \pm}(\gamma)} \chi(z)\right)^{\ell(\gamma) / \ell^\#(\gamma)},
 $$
 since for any closed geodesic $\gamma : \R/\Z \to \Sigma$ such that $i(\gamma, \gamma_\star) = n$ we have
 $$
\# \{z \in \partial,~z = (\gamma(\tau), \gamma'(\tau)) \text{ for some }\tau\} = n \ell^\#(\gamma)/ \ell(\gamma).
 $$
It remains to see that $\langle i_\Delta^*K_{\chi, \pm,n}(s), 1- \tilde \chi_L\rangle \to 0$ as $L \to +\infty$. Note that Lemma \ref{lem:estonell} gives
\begin{equation}\label{eq:estchil}
\left\|\dd^k\tilde \chi_L\right\| \leqslant C_k \e^{C_k L}.
\end{equation}
By Remark \ref{rem:scatcontinuous}, if $s_0 > 0$ is large enough, one has $\Sc_\pm(s_0) : \Omega^\bullet(\partial) \to C^0\left(\partial, \wedge^\bullet T^*\partial\right)$. Also for any $s \in \C$ with $\Re(s) > 0$ we have
\begin{equation}\label{eq:s_0+s}
\Sc_\pm(s_0 + s)w = (\Sc_\pm(s_0)w) \e^{-s\ell_\pm(\cdot)}, \quad w \in \Omega^\bullet(\partial).
\end{equation}
Let $N \in \N$ such that $\iota_\Delta^*K_{\chi, \pm,n}(s_0)$ extends as a continuous linear form on $C^N(\partial).$ Then Lemma \ref{lem:estonell} and (\ref{eq:estchil}) imply that if $\Re(s)$ is large enough, the product $\e^{-s\ell_{\pm, n}(\cdot)}\iota_\Delta^*K_{\chi, \pm,n}(s_0)$ is well defined and 
$$
\begin{aligned}
\left|\left\langle \e^{-s\ell_{\pm, n}(\cdot)}\iota_\Delta^*K_{\chi, \pm,n}(s_0), (1-\tilde \chi_L)\right\rangle\right| &= \left|\left\langle \iota_\Delta^*K_{\chi, \pm,n}(s_0), (1-\tilde \chi_L)\e^{-s\ell_{\pm,n}(\cdot)} \right\rangle\right| \\
&\leqslant C \left\|(1-\tilde \chi_L)\e^{-s\ell_{\pm,n}(\cdot)}\right\|_{C^N(\partial)} \\
&\leqslant C_N \e^{(C_N - \Re(s))L},
\end{aligned}
$$
since $\ell_{\pm,n} \geqslant L$ on $\supp(1-\tilde \chi_L)$. Therefore, to obtain that $\langle i_\Delta^*K_{\chi, \pm,n}(s_0 + s), 1- \tilde \chi_L\rangle \to 0$ as $L \to +\infty$, it suffices to show that 
$$
\e^{-s\ell_{\pm, n}(\cdot)}\iota_\Delta^*K_{\chi, \pm,n}(s_0) = \iota_\Delta^*K_{\chi, \pm,n}(s_0+s).
$$
This equality is a consequence of (\ref{eq:s_0+s}) and Lemma \ref{lem:elementary}, since we can take $s$ arbitrarily large to make $\exp(-s\ell_{\pm,n}(\cdot)) \in C^N(\partial)$ for any $N > 0$.
\end{proof}
 
As a consequence we have the

\begin{corr}
The function $s \mapsto \eta_{\pm, \chi, n}(s)$ defined for $\Re(s) \gg 1$ by the right hand side of (\ref{eq:strf}) extends to a meromorphic function on the whole complex plane.
\end{corr}
 
To prove Theorem \ref{thm:main} we now want to use a standard Tauberian argument near the first pole of $\eta_{\pm, \chi, n}$ to obtain the growth of $N(n,L)$. Indeed, it is known (see \S\ref{sec:tauberian}) that $s \mapsto R_{\pm, \delta}(s)$ has a pole at $s=h_\star$. However since $\eta_{\pm, \chi, n}$ is given by the trace of the restriction to $\partial$ of $R_{\pm, \delta}$, it is not clear a priori that $\eta_{\pm, \chi, n}$ will have the right behavior at $s = h_\star$. However in the next section we obtain some priori bounds on $N(n,L)$; this will imply that $\eta_{\pm, \chi, n}$ has indeed a pole at $s = h_\star$ of order $n$.

 \section{A priori bounds on the growth of geodesics with fixed intersection number with $\gamma_\star$}\label{sec:apriori}
 
The purpose of this section is to get \textit{a priori} bounds on $N(1,L)$ (and $N(2,L)$ in the case where $\gamma_\star$ is separating), using Parry-Pollicott's bound for Axiom A flows \cite{parry1983analogue}.

Choose some point $x_\star \in \gamma_\star$. Let $\grm$ the genus of $\Sigma$ and $(a_1,b_1, \dots, a_\grm, b_\grm)$ the natural basis of generators of $\Sigma$, so that the fundamental group of $\Sigma$ is the finitely presented group given by
$$\pi_1(\Sigma) = \langle a_1, b_1, \dots, a_\grm, b_\grm,~ [a_1,b_1]\cdots [a_\grm, b_\grm] = 1\rangle,$$
where we set $\pi_1(\Sigma) = \pi_1(\Sigma, x_\star).$

%%%

\subsection{The case $\gamma_\star$ is not separating}\label{subsec:apriorinotseparating}

\subsubsection{Lower bound}\label{subsubsec:lowernotseparating}

In this paragraph we will prove the

\begin{prop}\label{prop:lowerbound}
If $\gamma_\star$ is not separating, then there is $C > 0$ such that for $L$ big enough,
$$
C\e^{h_\star L}/L \leqslant N(1,L).
$$
\end{prop}

Note that the bound given in Theorem \ref{thm:main} is actually $N(1,L) \sim c_\star \e^{h_\star L}$. We could obtain a bound of this order with the methods presented in \S\ref{subsec:aprioriseparating} ; however the bounds given by Proposition \ref{prop:lowerbound} are sufficient for our purpose.

Up to applying a diffeomorphism to $\Sigma$, we may assume that $\gamma_\star$ is represented by $a_\grm \in \pi_1(\Sigma)$. In particular, $\Sigma_\star$ is a surface of genus $\grm - 1$ with $2$ punctures and the fundamental group $\pi_1(\Sigma_\star) = \pi_1(\Sigma_\star, x_\star')$ (here $x_\star'$ is some choice of point on $\partial \Sigma_\star$) is the free group given by $\langle a_1, b_1, \dots, a_\grm \rangle$.  Let $\widetilde \Sigma_\star$ denote the universal cover of $\Sigma_\star$ and let $\tilde x'_\star \in \widetilde \Sigma_\star$ such that $\pi(\tilde x'_\star) = x'_\star$ where $\pi : \widetilde \Sigma_\star \to \Sigma_\star$ is the natural projection. Then $\pi_1(\Sigma)$ acts on $\widetilde \Sigma_\star$ by deck transformations and we set
$$
\ell_\star(w) = \mathrm{dist}(\tilde x'_\star, w \tilde x'_\star), \quad w \in \pi_1(\Sigma_\star),
$$
where the distance comes from the metric $\pi^*g$ on $\widetilde \Sigma_\star.$ Note that if $\gamma_{[w]}$ denotes the unique geodesic in the free homotopy class of $w$ (which is represented by the conjugacy class $[w]$), we have $\ell(\gamma_{[w]}) \leqslant \ell_\star(w)$. We also denote 
$$
\wl(w) = \inf \left\{n \geqslant 0,~\alpha_{1} \dots \alpha_n = w,~\alpha_j \in \{a_k, a_k^{-1}, b_k, b_k^{-1},~k = 1, \dots, \grm-1\} \cup \{a_\grm, a_\grm^{-1}\}\right\}
$$
the word length of an element $w \in \pi_1(\Sigma_\star)$. It follows from the Milnor-\v{S}varc lemma \cite[Proposition 8.19]{bridson2013metric} that for some constant $D > 0$ we have
\begin{equation}\label{eq:milnor0}
\frac{1}{D} \wl(w) - D \leqslant \ell_\star(w) \leqslant D \wl(w) + D, \quad w \in \pi_1(\Sigma_\star).
\end{equation}
Also recall that we have the classical orbital counting (see e.g. \cite{roblin2003ergodicite})
\begin{equation}\label{eq:margulis}
\#\{w' \in \pi_1(\Sigma_\star, x_\star),~\ell_\star(w') \leqslant L \}\sim A \e^{h_\star L}, \quad L \to \infty
\end{equation}
for some $A >0$, where $h_\star$ is the topological entropy of the geodesic flow $(\Sigma_\star, g)$ restricted to the trapped set (see the introduction).

\begin{lemm}\label{lem:constructgeodesics}
Take $w,w' \in \pi_1(\Sigma_\star)$. Then $[b_\grm w] = [b_\grm w']$ (as conjugacy classes of $\pi_1(\Sigma)$) if and only if $w = b_\grm^{-1} a_\grm^n b_\grm w' a_\grm^{-n}$ in $\pi_1(\Sigma)$ for some $n \in \Z$.
\end{lemm}

\begin{proof}
If $w = b_\grm^{-1} a_\grm^n b_\grm w' a_\grm^{-n}$, then clearly $b_\grm w$ and $b_\grm w'$ are conjugated in $\pi_1(\Sigma, x_\star).$ Reciprocally, assume that $[b_\grm w] = [b_\grm w']$, and take smooth paths $\gamma$ and $\gamma'$ representing $b_\grm w$ and $b_\grm w'$. Then there is a smooth homotopy  $H : [0,1] \times \R/\Z \to \Sigma$ such that $H(0, \cdot) = \gamma$ and $H(1, \cdot) = \gamma'$. We may assume that $H$ is transversal to $\gamma_\star$ so that 
$H^{-1}(\gamma_\star)$
is a smooth submanifold of $[0,1] \times \R/\Z$. It is clear that we may deform a little bit the paths $\gamma$ and $\gamma'$ (in $\pi_1(\Sigma, x_\star$)) so that $\gamma$ and $\gamma'$ intersect transversaly $\gamma_\star$ exactly once, so that $H^{-1}(\gamma_\star) \cap (\{j\} \times \R/\Z) = \{j\} \times \{[0]\}$ for $j = 0,1.$ Thus there is an embedding $F : [0,1] \to [0,1] \times \R/\Z$ such that $\mathrm{Im}(F) \subset H^{-1}(\gamma_\star)$ and $F(0) = (0,[0])$ and $F(1) = (1,[0])$. Write $F = (S,T)$. Then set 
$$
\tilde H(s,t) = H(S(s), [T(s) + t]), \quad (s,t) \in [0,1] \times [0,1].
$$
It is immediate to check that $\tilde H$ realizes an homotopy between $\gamma$ and $\gamma'$ with $\tilde H(s, 0) \in \gamma_\star$ for any $s \in [0,1]$. Thus, we obtain 
$
b_\grm w = a_\grm^{-n} b_\grm w' a_\grm^{n}
$
for some $n \in \Z.$
\end{proof}

\begin{proof}[Proof of Proposition \ref{prop:lowerbound}]
In what follows, $C$ is a constant that may change at each line. For any $w' \in \pi_1(\Sigma_\star)$ and $n \in \Z$, we have by (\ref{eq:milnor0})
$$
\begin{aligned}
\ell_\star(a_\grm^n b_\grm w' a_\grm^{-n}) &\geqslant \frac{1}{D} \wl(a_\grm^n b_\grm w' a_\grm^{-n}) - D
& \geqslant \frac{2|n|}{D} - \frac{\wl(w')}{D} - D.
\end{aligned} 
$$
In particular, for any $L$ and $w'$ such that $\ell_\star(w') \leqslant L$, we have
\begin{equation}\label{eq:estnumberclass}
\left| \left\{n \in \Z,~ \ell_\star(a_\grm^n b_\grm w' a_\grm^{-n}) \leqslant L \right\} \right| \leqslant CL + C.
\end{equation}
Now for $w \in \pi_1(\Sigma_\star)$ set $\Ccal_w = \{a_\grm^n b_\grm w a_\grm^{-n},~n \in \Z\} \subset \pi_1(\Sigma_\star)$ and denote by $\Cscr$ the set of such classes. For $\Ccal \in \Cscr$ we set $\ell_\star(\Ccal) = \inf_{w \in \Ccal} \ell_\star(w)$. Now by Lemma \ref{lem:constructgeodesics} we have an injective map
$$
\{\Ccal \in \Cscr,~\ell_\star(\Ccal) \leqslant L\} \to \{\gamma \in \Pcal_1,~\ell(\gamma) \leqslant L + C\}, \quad \Ccal_w \mapsto [b_\grm w].
$$
where $\Pcal_1$ denotes the set of primitive geodesics $\gamma$ such that $i(\gamma, \gamma_\star) = 1$. In particular we get with (\ref{eq:estnumberclass}) and (\ref{eq:margulis})
$$
\begin{aligned}
N(1,L) &\geqslant \sum_{\substack{\Ccal \\ \ell_\star(\Ccal)\leqslant L - C}} 1 \\
&\geqslant \frac{1}{CL + C} \sum_{\substack{\Ccal \\ \ell_\star(\Ccal)\leqslant L - C}}\left|\left\{w \in \Ccal,~\ell_\star(w) \leqslant L - C\right\}\right|\\
&=\frac{1}{CL + C} \left|\left\{w \in \pi_1(\Sigma_\star),~ \ell_\star(w) \leqslant L-C\right\}\right| \\
&\geqslant \frac{1}{CL+C} \exp(h_\star (L-C)),
\end{aligned}
$$
which concludes the proof.
\end{proof}

\subsubsection{Upper bound}\label{subsubsec:uppernotseparating}
Each $\gamma \in \Pcal_{1}$ with $\ell(\gamma) \leqslant L$ lies in the free homotopy class of $b_\grm^{\pm1} w'$ for some $w' \in \pi_1(\Sigma_\star, x_\star')$ and $\ell_\star(w) \leqslant L + C$. In particular (\ref{eq:margulis}) gives the bound
$$
N(1,L) \leqslant C \exp(h_\star L)
$$
for large $L$. Now let $\gamma \in \Pcal_{2}$ with $\ell(\gamma) \leqslant L$. Then $\gamma$ is in the conjugacy class of some concatenation $b_\grm^{\pm 1} w' b_\grm^{\pm 1} w''$, where $w', w'' \in \pi_1(\Sigma_\star)$ satisfy $\ell_\star(w') + \ell_\star(w'') \leqslant L + C$. Thus we get
$$
\begin{aligned}
N(2, L) &\leqslant C\sum_{\substack{w',w'' \in \pi_1(\Sigma_\star) \\ \ell_\star(w') + \ell_\star(w'') \leqslant L +C}} 1 \\
&\leqslant \sum_{k = 0}^{L+C} C \exp({h_\star k}) C \exp({h_\star (L+C-k)}) \\
&\leqslant C' L \exp(h_\star L). 
\end{aligned}
$$
Iterating this process we finally get, for large $L$,
$$
N(n, L) \leqslant C L^{n-1} \exp(h_\star L).
$$

%%%
\subsection{The case $\gamma_\star$ is separating}\label{subsec:aprioriseparating}

Assume now that $\gamma_\star$ is separating and write $\Sigma \setminus \gamma_\star = \Sigma_1 \sqcup \Sigma_2$ where the surfaces $\Sigma_j$ are connected. Up to applying a diffeomorphism to $\Sigma$, we may assume that $\gamma_\star$ represents the class
\begin{equation}\label{eq:star}
[a_1, b_1] \cdots [a_{\grm_1}, b_{\grm_1}] = [a_{\grm}, b_\grm]^{-1} \cdots [a_{\grm_1 + 1}, b_{\grm_1 + 1}]^{-1} \in \pi_1(\Sigma).
\end{equation}
Here $\grm_1$ is the genus of the surface $\Sigma_1$, and the genus $\grm_2$ of $\Sigma_2$ satisfies $\grm_1 + \grm_2 = \grm$. 

 We set $\pi_1(\Sigma) = \pi_1(\Sigma, x_\star)$ and $\pi_1(\Sigma_j) = \pi_1(\Sigma_j, x_\star)$ for $j=1,2$ (we see $\Sigma_j$ as a compact surface with boundary $\gamma_\star$ so that $x_\star$ lives on both surfaces). Then $\pi_1(\Sigma_1)$ (resp. $\pi_1(\Sigma_2)$) is the free group generated by $a_1, b_1, \dots, a_{\grm_1}, b_{\grm_1}$ (resp. $a_{\grm_1 + 1}, b_{\grm_1 + 1}, \dots, a_\grm, b_\grm$), and we denote by $w_{\star,1}$ and $w_{\star,2}$ the two natural words given by (\ref{eq:star}) representing $\gamma_\star$ in $\pi_1(\Sigma_1)$ and $\pi_1(\Sigma_2)$. 
Note that we have a well defined map
$$
\begin{matrix}
\pi_1(\Sigma_1) \times \pi_1(\Sigma_2) &\longrightarrow& \pi_1(\Sigma) \\
(w_1, w_2) & \longmapsto  & w_2 w_1
\end{matrix}
$$
given by the composition of two curves. 

For any $w \in \pi_1(\Sigma),$ we will denote by $[w]$ its conjugacy class, and $\gamma_w$ the unique geodesic of $\Sigma$ such that $\gamma_w$ is isotopic to any curve in $w$ (in fact we will often identify $[w]$ and $\gamma_w$). Let $(\widetilde \Sigma, \tilde g)$ be the universal cover of $(\Sigma, g)$, and choose $\tilde x_\star \in \widetilde \Sigma$ some lift of $x_\star$. Then $\pi_1(\Sigma)$ acts as deck transformations on $\widetilde \Sigma$ and we will denote
$$
\ell_\star(w) = \mathrm{dist}_{\widetilde \Sigma}(\tilde x_\star, w \tilde x_\star), \quad w \in \pi_1(\Sigma).
$$
As in the preceding section, we have the orbital counting (see e.g. \cite{roblin2003ergodicite})
$$
\#\{w_j \in \pi_1(\Sigma_j),~\ell_\star(w_j) \leqslant L\} \sim A_j \e^{h_j L}, \quad L \to \infty, \quad j = 1,2,
$$
for some $A_1, A_2 > 0.$
\subsubsection{Lower bound}
Unlike the case $\gamma_\star$ not separating, we will need a sharp lower bound. Namely, we prove here the following result.

\begin{prop}\label{prop:lowerbound2}
Assume that $\gamma_\star$ is separating, and that $h_1 = h_2 = h_\star$. Then there is $C > 0$ such that for $L$ large enough,
$$
N(2,L) \geqslant CL \e^{h_\star L}.
$$
\end{prop}
Let us briefly describe the strategy used to prove Proposition \ref{prop:lowerbound2}. We denote by $\Pcal(\Sigma_j)$ the set of primitive closed geodesics of $\Sigma_j$. Then we know that 
\begin{equation}\label{eq:boundopen}
N_j(L) \sim \frac{\e^{h_\star L}}{h_\star L}, \quad L \to +\infty, \quad j = 1,2,
\end{equation}
where $N_j(L) = \#\{\gamma \in \Pcal(\Sigma_j),~\ell(\gamma) \leqslant L\}$. In particular we have for any $L$ large enough
\begin{equation}\label{eq:firstbound}
\sum_{\substack{\gamma \in \Pcal(\Sigma_j) \\ \ell(\gamma) \leqslant L}} \ell(\gamma) \geqslant \frac{L}{4} \sum_{\substack{\gamma \in \Pcal(\Sigma_j) \\ L/4 < \ell(\gamma) \leqslant L}} 1 = \frac{L}{4}(N_j(L) - N_j(L/4)) \geqslant C \exp(h_\star L)
\end{equation}
for some constant $C>0$. Therefore we have for $L$ large enough
$$
\begin{aligned}
\sum_{\substack{(\gamma_1, \gamma_2) \in \Pcal_1 \times \Pcal_2 \\ \ell(\gamma_1) + \ell(\gamma_2) \leqslant L}} \ell(\gamma_1) \ell(\gamma_2) &\geqslant  \sum_{\substack{\gamma_1 \in \Pcal_1 \\ L/4 \leqslant \ell(\gamma_1) \leqslant L}} \ell(\gamma_1) \sum_{\substack{\gamma_2 \in \Pcal_2 \\ \ell(\gamma_2) \leqslant L - \ell(\gamma_1)}} \ell(\gamma_2) \\
& \geqslant  C \sum_{\substack{\gamma_1 \in \Pcal_1 \\ L/4 \leqslant \ell(\gamma_1) \leqslant L}} \ell(\gamma_1) \e^{h(L-\ell(\gamma_1))}\\ 
&  \geqslant  C\sum_{L/4 < k \leqslant L-1} [N_1(k+1) - N_1(k)] k \e^{h_\star (L-k-1)}.
\end{aligned}
$$
Now note that (\ref{eq:boundopen}) implies that $N(k+1) - N(k) \displaystyle{\geqslant C \frac{\e^{h_\star k}}{h_\star k}}$ for any $k$ large enough. Therefore we get for $L$ large enough (with some different constant $C$)
\begin{equation}\label{eq:firstest}
\begin{aligned}
\sum_{\substack{(\gamma_1, \gamma_2) \in \Pcal_1 \times \Pcal_2 \\ \ell(\gamma_1) + \ell(\gamma_2) \leqslant L}} \ell(\gamma_1) \ell(\gamma_2) &\geqslant CL\e^{h_\star L}.
\end{aligned}
\end{equation}
As a consequence, if given geodesics $\gamma_j \in \Pcal(\Sigma_j)$, we are able to construct about $\ell(\gamma_1) \ell(\gamma_2)$ new geodesics of $\Sigma$, intersecting $\gamma_\star$ exactly twice and of length bounded by $\ell(\gamma_1) + \ell(\gamma_2),$ then Proposition \ref{prop:lowerbound2} will follow. An idea would be to choose $w_j \in \pi_1(\Sigma_j)$ such that $[w_j]$ represents $\gamma_j$, and to consider the geodesics given by the conjugacy classes $[\bar w_2 \bar w_1]$ where $\bar w_j$ is a cyclic permutation of the word $w_j$ (there are about $\ell(\gamma_j)$ of those). However this process may not be injective (see Lemma \ref{lem:2}), and so more work is needed.

\begin{rem}
If $h_1 \neq h_2$, then adapting the proof presented below would show
$$
N(2, L) \geqslant C \e^{h_\star L}
$$
for large $L$, where $h_\star = \max(h_1, h_2)$.
\end{rem}

We start by the following lemma, which shows that the described procedure will give us indeed geodesics intersecting $\gamma_\star$ exactly twice, provided the geodesics $\gamma_j$ are not multiples of $\gamma_\star$.

\begin{lemm}\label{lem:1}
For two elements $w_j \in \pi_1(\Sigma_j),$ $j=1,2$, we have
$
i(\gamma_{w_2 w_1}, \gamma_\star) = 2
$
except if $w_j = w_{\star,j}^k$ in $\pi_1(\Sigma_j)$ for some $k \in \Z$ and $j \in \{1,2\}.$
\end{lemm}

\begin{proof}
Let $\gamma : \R/\Z \to \Sigma$ be a smooth curve in the free homotopy class of $w_2w_1$ such that $\{\tau \in \R/\Z,~\gamma(\tau) \in \gamma_\star\} = \{\tau_1, \tau_2\}$ for some $\tau_1 \neq \tau_2 \in \R/\Z.$ We may also choose $\gamma$ so that $\gamma|_{[\tau_1, \tau_2]}$ (resp. $\gamma|_{[\tau_2, \tau_1]}$) is homotopic to some representative $\gamma_1 : [0,1] \to \Sigma$ of $w_1$ (resp. some representative $\gamma_2 : [0,1] \to \Sigma$ of $w_2$) relatively to $\gamma_\star$, meaning that there is a homotopy between $\gamma|_{[\tau_1, \tau_2]}$ and $\gamma_1$ with endpoints (not necessarily fixed) in $\gamma_\star$. Here $[\tau_1, \tau_2]  \subset \R/\Z$ is the interval linking $\tau_1$ and $\tau_2$ in the counterclockwise direction. 

As $\gamma_{w_2w_1}$ minimizes the quantity $i(\gamma, \gamma_\star)$ for $\gamma \in [\gamma_{w_2w_1}]$, we have either $i(\gamma_{w_2w_1}) = 0$ or $i(\gamma_{w_2w_1}) = 2$. If $i(\gamma_{w_2w_1}, \gamma_\star) = 0$ then there exists a homotopy $H: [0,1] \times \R/\Z \to \Sigma$ such that $H(0, \cdot) = \gamma$ and $H(1, \tau) \notin \gamma_\star$ for any $\tau$. Moreover we may assume that $H$ is transversal to $\gamma_\star$, so that the preimage
$$
H^{-1}(\gamma_\star) \subset [0,1] \times \R/\Z
$$
is an embedded submanifold of $[0,1] \times \R/\Z$. As $H^{-1}(\gamma_\star) \cap \{s=0\} = \{\tau_1, \tau_2\}$ and $H^{-1}(\gamma_\star) \cap \{s=1\} = \emptyset$ it follows that there is an embedding $F : [0,1] \to [0,1] \times \R/\Z$ such that $F(0) = (0, \tau_1)$, $F(1) = (0, \tau_2)$ and
$$
F(t) \in H^{-1}(\gamma_\star), \quad t \in [0,1].
$$
As $F$ is an embedding, we have that $F$ is homotopic either to $J_{[\tau_1, \tau_2]}$ or to $J_{[\tau_2, \tau_1]}$, where $J_{[\tau, \tau']} : [0,1] \to [0,1] \times \R/\Z$ is the natural map that sends $[0,1]$ to ${0} \times [\tau, \tau']$. We may assume without loss of generality that $F \sim J_{[\tau_1, \tau_2]}$. Writing $F = (S, T)$ we have in particular that $T$ is homotopic to $I_{[\tau_1, \tau_2]} = p_2 \circ J_{[\tau_1, \tau_2]}$, where $p_2 : [0,1] \times \R/\Z \to \R/\Z$ is the projection over the second factor. This means that there is $G : [0,1] \times [0,1] \to \R/\Z$ such that for any $s,t \in [0,1]$,
$$
G(s,0) = \tau_1, \quad G(s,1) = \tau_2, \quad G(0,t) = \tau_1 + t(\tau_2 - \tau_1), \quad G(1,t) = T(t).
$$
Now we set 
$
\tilde H(s,t) = H(sS(t), G(s,t))
$
for $s,t \in [0,1].$ Then
$$
\tilde H(0,t) = \gamma(\tau_1 + t(\tau_2 - \tau_1)), \quad \tilde H(1,t) = (H\circ F)(t), \quad t \in [0,1],
$$
and 
$$
\tilde H(s,0) = H(0, \tau_1) = x_1, \quad \tilde H(s,1) = H(0, \tau_2) = x_2, \quad s \in [0,1].
$$
We conclude that $t \mapsto \gamma|_{[\tau_1, \tau_2]}(\tau_1 + t(\tau_2 - \tau_1))$, and thus $\gamma_1$, is homotopic (relatively to $\gamma_\star$) to some curve contained in $\gamma_\star$. Thus $w_1 = w_\star^k$ for some $k \in \Z$, in $\pi_1(\Sigma)$. As the inclusion $\pi_1(\Sigma_j) \to \pi_1(\Sigma)$ is injective (since $\grm_j > 0$ for $j = 1,2$), the lemma follows.
\end{proof}

Now we need to understand when the geodesics given by $[w_2 w_1]$ and $[w_2' w_1']$ are the same. This is the purpose of the following

\begin{lemm}\label{lem:2}
Take $w_j, w_j' \in \pi_1(\Sigma_j)$, $j=1,2$ such that $i(\gamma_{[w_2w_1]}, \gamma_\star) = 2$. Then $[w_2 w_1] = [w_2' w_1']$ as conjugacy classes of $\pi_1(\Sigma)$ if and only if there are $p,q \in \Z$ such that
\begin{equation}\label{eq:pq}
w_2 = w_{\star,2}^p w_2' w_{\star,2}^q, \quad w_1 = w_{\star,1}^{-q} w_1' w_{\star, 1}^{-p}.
\end{equation}
\end{lemm}

\begin{proof}
Again, let $\gamma : \R/\Z \to \Sigma$ be a smooth curve intersecting transversely $\gamma_\star$ such that $\{\tau \in \R/\Z,~\gamma(\tau) \in \gamma_\star\} = \{\tau_1, \tau_2\}$ for some $\tau_1 \neq \tau_2 \in \R/\Z$, such that $\gamma([\tau_1, \tau_2]) \subset \Sigma_1$ and $\gamma([\tau_2, \tau_1]) \subset \Sigma_2$. Let $x_j = \gamma(\tau_j)$ for $j=1,2$ and chose arbitrary paths $c_j$ contained in $\gamma_\star$ linking $x_j$ to $x_\star$. All the preceding choices can be made so that the curve $\gamma_1 = c_2 \gamma|_{[\tau_1, \tau_2]} c_1^{-1}$ (resp. $\gamma_2 = c_1 \gamma|_{[\tau_2, \tau_1]} c_2^{-1}$) represents $\gamma_*^pw_1\gamma_*^q$ (resp. $\gamma_*^{-q}w_2\gamma_*^{-p}$) for some $p,q \in \Z.$ We may proceed in the same way to obtain $\gamma', \tau_1', \tau_2', c_1', c_2', p',q'$ so that the same properties hold with $w_1, w_2$ replaced by $w_1',w_2'.$ By hypothesis, we have that $\gamma$ is freely homotopic to $\gamma'$. Thus we may find a smooth map $H : [0,1] \times \R/\Z \to \Sigma$ such that $H(0, \cdot) = \gamma$ and $H(1, \cdot) = \gamma'$. As in Lemma \ref{lem:1}, $H$ may be chosen to be transversal to $\gamma_\star$, so that 
$$
H^{-1}(\gamma_\star) \subset [0,1] \times \R/\Z
$$
is a finite union of smooth embedded submanifolds of $[0,1] \times \R/\Z$. Let $(x, \rho) : \Sigma \to \R/\Z \times (-\varepsilon, \varepsilon)$ be coordinates near $\gamma_\star$ such that $\{\rho = 0\} = \gamma_\star$ and $|\rho| = \mathrm{dist}(\gamma_\star, \cdot)$ and such that $\{(-1)^{j-1}\rho \geqslant 0\} \subset \Sigma_j$. As $H^{-1}(\gamma_\star) \cap \{s=0\} = \{\tau_1, \tau_2\}$ and $H^{-1}(\gamma_\star) \cap \{s=1\} = \{\tau_1', \tau_2'\}$, we have two smooth embeddings $F_1, F_2 : [0,1] \to [0,1] \times \R/\Z$ such that $F_j([0,1]) \subset H^{-1}(\gamma_\star)$ and
$
F_j(0) = (0, \tau_j)
$
for $j=1,2$, with (indeed we have $i(\gamma, \gamma_\star) = 2$ and thus there is a path in $H^{-1}(\gamma_\star)$ linking $\{s=0\}$ to $\{s=1\}$, since otherwise we could proceed as in the proof of Lemma \ref{lem:1} to obtain that $i(\gamma, \gamma_\star) = 0$). In fact we have $F_1(1) = (1, \tau_1')$ and $F_2(1) = (1, \tau_2')$ (we shall prove it later). Set $F_j = (S_j, T_j)$. Set, with the same notations as in the proof of Lemma \ref{lem:2},
$$
\tilde H(s, t) = H\bigl((1-t)S_1(s) + tS_2(s), ~T_1(s) + t(T_2(s) - T_1(s))\bigr), \quad s,t \in [0,1].
$$
Then $H$ is smooth as $T_1(s) \neq T_2(s)$ for any $s$ (as $H^{-1}(\R/\Z$ is smooth), and thus
$$
\tilde H(0, t) = \gamma(\tau_1 + t(\tau_2 - \tau_1)), \quad \tilde H(1, t) = \gamma(\tau_1' + t(\tau_2'- \tau_1')),
$$
and 
$$
\tilde H(s, 0) = H(S_1(s), T_1(s)), \quad \tilde H(s,1) = H(S_2(s), T_2(s)), \quad s \in [0,1].
$$
For $j=1,2$ let $c_j(s),~ s\in[0,1],$ be paths contained in $\gamma_\star$ depending continuously on $s$ and linking $T_j(s)$ to $x_\star$, such that $c_j(0) = c_j$. Then the construction of $\tilde H$ shows that
$$
c_2(0)\gamma|_{[\tau_1, \tau_2]}c_1(0)^{-1} \sim c_2(1) \gamma|_{[\tau_1', \tau_2']}c_1(1)^{-1}, 
$$
and reversing the role of $\tau_1$ and $\tau_2$ in the constructions made above,
$$
c_1(0)\gamma|_{[\tau_2, \tau_1]}c_2(0)^{-1} \sim c_1(1) \gamma|_{[\tau_2', \tau_1']}c_2(1)^{-1}.
$$
Thus we obtain 
$$\gamma_\star^{p}w_1\gamma_\star^q = c_2(1)c_2'^{-1}\gamma_*^{p'}w_1' \gamma_{*}^{q'}c_1' c_1(1)^{-1}, \quad \gamma_\star^{-q}w_2\gamma_\star^{-p} = c_1(1)c_1'^{-1} \gamma_*^{-q'}w_2\gamma_\star^{-p'}c_2'c_2(1)^{-1},$$ which is the conclusion of Lemma \ref{lem:2} as the paths $c_1(1)c_1'^{-1}$ and $c_2(1)c_2'^{-1}$ are contained in $\gamma_\star$ (and again, the inclusions $\pi_1(\Sigma_j) \to \pi_1(\Sigma)$, $j = 1,2,$ are injective).

Thus it remains to show that $F_j(1) = (1, \tau_j')$ for $j=1,2$. We extend $\rho$ into a smooth function $\rho : \Sigma \to \R$ such that $(-1)^{j-1}\rho > 0$ on $\Sigma_j \setminus \gamma_\star$. Now there exists a continuous path $G : [0,1] \to ([0,1] \times \R/\Z) \setminus H^{-1}(\gamma_\star)$ such that $G(0) \in \{0\}\times ]\tau_1, \tau_2[$ and $G(1) \in \{1\}\times (\R/\Z \setminus \{\tau_1', \tau_2'\})$ (otherwise it would mean that there is a continuous path in $[0,1] \times \R/\Z$ linking $(0, \tau_1)$ to $(0, \tau_2)$, which would imply, as in Lemma \ref{lem:1}, that $i(\gamma, \gamma_\star) = 0$). In particular we have $\rho \circ H \circ G > 0$ since $\rho(H(0, \tau)) > 0$ for $\tau \in ]\tau_1, \tau_2[.$ Thus necessarily $G(1) \in \{1\}\times]\tau_1', \tau_2'[$ since $\rho(H(1, \tau)) < 0$ for $\tau \in ]\tau_2', \tau_1'[.$  Now, as $\mathrm{Im}(F_1) \cap \mathrm{Im}(F_2) = \emptyset$ (again, if the intersection was not empty we could find a path linking $(0, \tau_1)$ to $(0, \tau_2)$), we have that $G(1)$ lies in $]T_1(1), T_2(1)[.$ Since $(\rho \circ H \circ H)(1) > 0$, it follows that $T_1(1) = \tau_1'$ and $T_2(1) = \tau_2'.$ The lemma is proven.
\end{proof}

Before starting the proof of Proposition \ref{prop:lowerbound2}, we state a technical result that will be useful to show that there are not too many elements $w_j, w_j' \in \pi_1(\Sigma_j)$ such that $[w_2 w_1] = [w_2' w_1']$. For any element $w_j \in \pi_1(\Sigma_j),$ we denote by $\ell([w_j])$ its translation length, that is
$$
\ell([w]) = \inf_{\tilde x \in \tilde \Sigma} \dd_{\tilde \Sigma}(\tilde x, w \tilde w).
$$
Of course this length coincide with the length of $\gamma_{w}$. 

\begin{lemm}\label{lem:0}
There exists $C > 0$ such that the following holds. For any ${w} \in \pi_1(\Sigma_j)$, there exists $n_{w} \in \Z$ such that
$$
\ell([w_{\star,j}^{n + n_w} w]) \geqslant \ell([w_{\star,j}^{n_w}w]) + \ell(\gamma_\star)|n|- C, \quad n \in \Z.
$$
\end{lemm}

\begin{proof}
The method presented here was inspired by Fr\'ed\'eric Paulin. We fix $j \in \{1,2\}$ and denote $w_\star = w_{\star, j}$. First note that if $w = w_\star^k$ for some $k \in \Z$ then the conclusion is clear with $n_w = -k$ and $C = 0$. Next assume that $w \neq w_\star^k$ for any $k$. In particular $w$ is not the trivial element and is thus hyperbolic. Let $(\tilde \Sigma, \tilde g)$ denote the universal cover of $(\Sigma, g)$ ; then $\pi_1(\Sigma)$ acts as deck transformations on $(\tilde \Sigma, \tilde g)$. For any $w \in \pi_1(\Sigma) \setminus \{1\},$ we denote by (here $z$ denotes any point in $\tilde \Sigma$)
$$
w_\pm = \lim_{k \to +\infty} (w^{\pm 1})^k(z)
$$
the two distinct fixed points of $w$ in the boundary at infinity $\partial_\infty \tilde \Sigma$ of $\tilde \Sigma.$ We also denote by $A_{w}$ the translation axis of $w$, that is, the unique complete geodesic of $(\tilde \Sigma, \tilde g)$ converging towards $w_+$ (resp. $w_-$) in the future (resp. in the past).
As $i(w, w_\star) = 0$ (since $w$ is not a power of $w_\star$ which represents the boundary of $\Sigma_j$), we have $A_w \cap A_{w_\star} = \emptyset.$ Moreover, $w_\pm \notin \{w_{\star, -}, w_{\star, +}\}$ (indeed if it were the case, then $w$ would be equal to some power of $w_\star$, as $\pi_1(\Sigma)$ acts properly and discontinuously on $\tilde \Sigma$). 

In a first step, we will assume that the family $(w_{\star, -}, w_-, w_+, w_{\star,+})$ is cyclically ordered in $\partial_\infty \tilde \Sigma \simeq S^1$, and we denote by $w \uparrow w_\star$ this property. Consider $z \in A_{w_\star}$ and $z' \in A_w$ such that $\dist(A_{w_\star}, A_w) = \dist(z, z')$. For any $x \neq y \in \tilde \Sigma$, we denote by $[x,y]$ the unique geodesic segment joining $x$ to $y$, by $(x,y)$ the unique complete oriented geodesic ray passing through $x$ and $y$ and by $(x,y)_\pm$ the future and past endpoints of $(x,y)$. Then we claim that the following holds (see Figure \ref{fig:geometry}):
\begin{enumerate}
\item \label{length} For any $n \geq 1$ we have $\dist(z, w_\star^n w z) \geqslant n\ell([w_\star]) + \ell([w])$;
\item \label{angle} There is $c > 0$, independent of $w$ satisfying $w \uparrow w_\star$, such that for any $n \geqslant 1$, the angle (taken in $[0, \pi]$) between the segments $[w^{-1} w_\star^{-n} z, z]$ and $[z, w_\star^n w z]$ is greater than $c$ (denoted $\alpha$ on Figure)
\end{enumerate}

\begin{figure}
\includegraphics[scale=1.15]{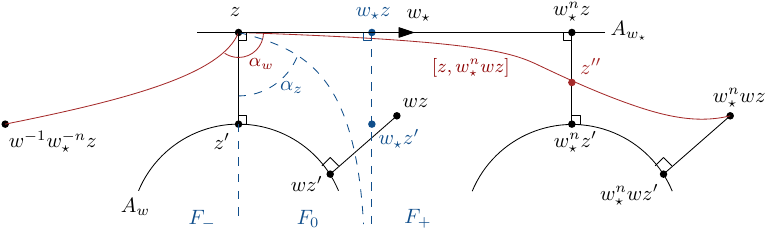}
\caption{Proof of Lemma \ref{lem:0}.}
\label{fig:geometry}
\medskip
\small{}
\end{figure}
To see that (\ref{length}) holds, first note that the segment $[z, w_\star^n w z]$ intersects $[w_\star^n z, w_\star^n z']$, because $w \uparrow w_\star$, and denote by $z''$ the intersection point. Then, as $[w_\star^n z, w_\star^n z']$ is orthogonal to $A_{w_\star}$, we have $\ell([z, z'']) \geqslant n \ell([w_\star]).$ Moreover, as both $[w_\star^n z, w_\star^n z']$ and $[w_\star^n w z', w_\star^n w z]$ are orthogonal to $w_\star^n A_{w}$, we have $\dist(z'', w_\star^n w z) \geqslant \ell([w])$.

We prove (\ref{angle}) as follows. We have a decomposition in connected sets
$$\tilde \Sigma \setminus \bigl( (z,z') \cup (w_\star z, w_\star z') \bigr)  = F_- \cup F_0 \cup F_+,$$
where $w_{\star, \pm} \in \overline{F_\pm}.$ Then since $w \uparrow w_\star$, we have $w_\star^n w z \in F_+$ for any $n \geqslant 1$ (since $w z \in F_0 \cup F_+$) and thus the angle $\alpha_w$ between $[z', z]$ and $[z, w_\star^n w z]$ is greater than the angle $\alpha_z$ between $(z, z')$ and the ray joining $z$ to $(w_\star z, w_\star z')_+$. Now $\alpha_z$ only depends on $z$ (and not on $w$), and we set $c = \inf_{y \in A_{w_\star}} \alpha_y > 0$ (indeed $y \mapsto \alpha_y$ is continuous and $\alpha_y = \alpha_{w_\star y}$ for any $y \in A_{w_\star}$). As $w^{-1} w_\star^{-n} z \in F_-$, we get (\ref{angle}).

Now it is a classical fact from the theory of CAT($-\kappa$) spaces ($\kappa > 0$)  that the following holds. For $c > 0$ as above, there is $C > 0$ such that if $\eta \in \pi_1(\Sigma) \setminus \{1\}$ and $z \in \tilde \Sigma$ satisfy that the angle (taken in $[0, \pi]$) between $[\eta^{-1} z, z]$ and $[z, \eta z]$ is greater or equal than $c$, then $\ell(\eta) \geqslant \dist(z, \eta z) - C$ (see for example \cite[Lemma 2.8]{paulin2012equilibrium}). Applying this to $\eta = w_\star^n w$, we get with (\ref{length}) and (\ref{angle})
\begin{equation}\label{eq:estdirection}
\ell([w_\star^n w]) \geqslant \dist(z, w_\star^n w z) - C \geqslant n \ell([w_\star]) + \ell([w]) - C, \quad n \geqslant 1.
\end{equation}
Here $C$ does not depend on $w$ such that $w \uparrow w_\star$. Now note that for $n > 0$ one has\footnote{\label{foot:2} Indeed, we have $(w_\star^nw)_+ = \lim_{k \to +\infty} (w_\star^n w)^k \cdot w_+ \subset [w_+, w_{\star,+}]$ (the interval joining $w_+$ to $w_{\star, +}$ in $\partial_\infty \tilde \Sigma$ but not containing $w_-$ nor $w_{\star, -}$) as $n > 0$. Similarly $(w_\star^nw)_- \subset [w_{\star, -}, w_-]$ and thus $w_\star w^n \uparrow w_\star.$}
\begin{equation}\label{eq:uparrow}
w \uparrow w_\star \quad \implies \quad  w_\star^n w \uparrow w_\star.
\end{equation}
Moreover, for $n > 0$ large enough (depending on $w$), we have\footnote{This is a consequence of footnote \ref{foot:2}, which implies that if $w \uparrow w_\star^{-1}$ and $w_\star^n w \uparrow w_\star^{-1}$ then $[(w_\star^n w)_-, (w_\star^n w)_+] \subset [w_-, w_+].$ By looking at the action of $w_\star$ on $\partial_\infty \tilde \Sigma$, we see that it is not possible if $n$ is large enough. Similarly, we have $w_\star^{-n}w \uparrow w_\star^{-1}$ for $n$ large enough.}
$$
w_\star^{\pm n} w \uparrow w_\star^{\pm 1}.
$$
Therefore, if $n_w = \inf\{n \in \Z,~w_\star^n w \uparrow w_\star\}$ we have, for any $n \geqslant 0$,
$$
w_\star^n w_\star^{n_w} w \uparrow w_\star \quad \text{ and } \quad w_\star^{-n} w_\star^{n_w - 1} w \uparrow w_\star^{-1}.
$$
Applying (\ref{eq:estdirection}) with $w$ replaced by $w_\star^{n_w} w$ we get 
$$
\ell([w_\star^n w_\star^{n_w} w]) \geqslant \ell([w_\star^{n_w} w]) + n\ell([w_\star]) - C, \quad n \geqslant 1.
$$
Now applying (\ref{eq:estdirection}) with $w_\star$ replaced by $w_\star^{-1}$ and $w$ replaced by $w_{\star}^{n_w - 1} w$, we get 
$$
\ell([(w_\star^{-1})^nw_\star^{n_w - 1} w]) \geqslant \ell([w_\star^{n_w - 1} w]) + n \ell([w_\star]) - C', \quad n \geqslant 1.
$$
The Lemma easily follows from the last two estimates, up to changing $C$ and $C'$ and replacing $n_w$ by $n_w - 1$.
\end{proof}

\begin{proof}[Proof of Proposition \ref{prop:lowerbound2}]

Let $j \in \{1,2\}$. For any primitive geodesic $\gamma_j \in \Pcal(\Sigma_j)$, we choose some $w_{\gamma_j} \in \pi_1(\Sigma_j)$ such that $\gamma_j$ corresponds to the conjugacy class $[w_{\gamma_j}]$. We may assume that $\wl(w_{\gamma_j}) = \wl([w_{\gamma_j}])$ where
$$
\wl([w_{\gamma_j}]) = \inf \left\{\wl(w'_j),~w_j' \in [w_{\gamma_j}] \right\}.
$$
As $\pi_1(\Sigma_j)$ is free, the element $w_{\gamma_j}$ is unique up to cyclic permutations. We denote $n_j(\gamma_j) = \wl([w_{\gamma_j}])$ ; then the Milnor-\v{S}varc lemma implies, for any $\gamma_j \in \Pcal(\Sigma_j)$ and $w_j \in [w_{\gamma_j}],$
$$
\ell(\gamma_j) = \ell([w_{\gamma_j}]) = \ell([w_j]) \leqslant \ell_\star(w_j) \leqslant D \wl(w_j) + D, 
$$
which gives
\begin{equation}\label{eq:milnor}
\ell(\gamma_j) \leqslant D n_j(\gamma_j) + D, \quad \gamma_j \in \Pcal(\Sigma_j).
\end{equation}
Our goal is now the following. Starting from geodesics $\gamma_j \in \Pcal(\Sigma_j)$, $j=1,2$, we want to construct about $\ell(\gamma_1) \ell(\gamma_2)$ distinct geodesics in $\Pcal$, by considering the conjugacy classes $[\bar w_{\gamma_2} \bar w_{\gamma_1}]$ where $\bar w_{\gamma_j}$ runs over all cyclic permutation of $w_{\gamma_j}$. However, as explained before, this process may conduct to produce several times the same geodesic in $\Pcal$ (recall Lemma \ref{lem:2}) so we are led to show estimates on the growth number of families of geodesics, as follows. For any $\gamma_j \in \Pcal(\Sigma_j)$, we define the family of conjugacy classes
$$
\Ccal_{\gamma_j} = \{ [w_{\star,j}^n w_{\gamma_j}],~[w_{\star,j}^n w_{\gamma_j}] \text{ is primitive},~n\in \Z\}.
$$
Here a class $[w]$ is said to be primitive if the closed geodesic corresponding to $[w]$ is primitive. We denote by $\Cscr_j$ the set of such families, and for each $\Ccal_j \in \Cscr_j$ we set
$$
\ell(\Ccal_j) = \min_{c \in \Ccal_j} \ell(c).
$$
The minimum exists by Lemma \ref{lem:0}. We have the following
\begin{lemm}\label{lem:thereisc}
There is $C > 0$ such that for $L$ big enough,
$$
\# \{\Ccal_j \in \Cscr_j,~\ell(\Ccal_j) \leqslant L\} \geqslant C \e^{h_\star L} / L.
$$
\end{lemm}
\begin{proof}
By Lemma \ref{lem:0} we have for any $\gamma_j \in \Pcal(\Sigma_j)$
$$
\#\{n \in \Z,~\ell([w_{\star,j}^n w_{\gamma_j}]) \leqslant L\} \leqslant C(L - \ell(\Ccal_{\gamma_j}) +C).
$$
It follows that for large $L$,
$$
N_j(L) = \sum_{\substack{\gamma_j \in \Pcal(\Sigma_j) \\ \ell(\gamma_j) \leqslant L}} 1 
= \sum_{\substack{\Ccal_j \in \Cscr_j \\ \ell(\Ccal_j) \leqslant L}} \#\{c \in \Ccal_j,~\ell(c) \leqslant L\} 
 \leqslant C\sum_{\substack{\Ccal_j \in \Cscr_j \\ \ell(\Ccal_j) \leqslant L}} (L - \ell(\Ccal_j) + C).
$$
Let $\tilde N_j(L) = \#\{\Ccal_j \in \Cscr_j,~\ell(\Ccal_j) \leqslant L\}.$ Then an Abel transformation gives
$$
\begin{aligned}
\sum_{\substack{\Ccal_j \in \Cscr_j \\ \ell(\Ccal_j) \leqslant L}} (L - \ell(\Ccal_j) + C) &\leqslant \sum_{k=1}^{L} \left(\tilde N_j(k) - \tilde N_j(k - 1)\right) (L - k + C)
& \leqslant C' \sum_{k=1}^L \tilde N_j(k),
\end{aligned}
$$
and thus $\sum_{k=1}^L \tilde N_j(k) \geqslant C \exp({h_\star L}) / L$ for large $L$. On the other hand we have for $M > 0$
$$
\sum_{k=1}^{L-M} \tilde N_j(k) \leqslant \sum_{k=1}^{L-M} N_j(k) \leqslant C \sum_{k=1}^{L-M} \frac{\e^{h_\star k}}{k} \leqslant C' \frac{\e^{h_\star (L-M)}}{L-M}.
$$
Therefore, if $M$ is big enough, we have for any $L$ large enough
$$
M \tilde N_j(L) \geqslant \sum_{k=L-M}^L \tilde N_j(k) \geqslant \frac{N_j(L)}{2} \geqslant C'' \e^{h_\star L}/L,
$$
which concludes.
\end{proof}

For any $\Ccal_j \in \Cscr_j$, we choose some class $[w_{\Ccal_j}] \in \Ccal_j$ such that $\ell(\Ccal_j) = \ell([w_{\Ccal_j}])$. Also $w_{\Ccal_j} \in \pi_1(\Sigma_j)$ may be chosen cyclically reduced, meaning that $\wl(w_{\Ccal_j}) = \wl([w_{\Ccal_j}])$. Let $W_{\Ccal_j} \subset \pi_1(\Sigma_j)$ denote the set of cyclic permutations of $w_{\Ccal_j}$ Then $|W_{\Ccal_j}| = \wl([w_{\Ccal_j}])$ since $w_{\Ccal_j}$ is primitive (see \cite{lyndon1962equation}). 

\begin{lemm}\label{lem:W}
For any $\Ccal_j \in \Cscr_j$, there exists a subset $W_{\Ccal_j}' \subset W_{\Ccal_j}$ with 
$$|W_{\Ccal_j}'| \geqslant (|W_{\Ccal_j}| - 3)/4$$ and the following property. For any $p,q \in \Z$ and $w \in W_{\Ccal_j}'$,
$$
(w_{\star, j})^p w (w_{\star,j})^q \in W_{\Ccal_j}' \quad \implies \quad p = q = 0.
$$
\end{lemm}
\begin{proof}
We prove the lemma for $j = 1$. Let $\Ccal_1 \in \Cscr_1$ ; we set $\grm = \grm_1, w_\star = w_{\star, 1}$ and $W = W_{\Ccal_1}$ to simplify notations. For $w \in W$, we will say that $w$ is of type $A$ if $(w_{\star})^pw(w_{\star})^{q} \in W_{\Ccal_1}$ for some $p,q\in \Z \setminus (0, 0)$. If $w$ is of type $A$, then exactly $2 \grm(|p| + |q|) \geqslant 2$ simplifications occur in the word $w' = (w_\star)^pw(w_\star)^{q}$, since $\wl(w_{\star}) = 4 \grm$. As $w_\star = a_1 b_1 a_1^{-1}b_1^{-1} \cdots a_\grm^{-1} b_\grm^{-1}$ and at least $2$ simplifications occur in $w$', we see that $w$ has necessarily one of the following forms :

\begin{align*}
&a_1b_1 \cdots                    &\text{(1)}&&\qquad
&a_1 \cdots b_\grm^{-1}       &\text{(3)}&&\qquad
&\cdots b_1^{-1} a_1^{-1}       &\text{(5)}
 \\[1em]
&b_\grm a_\grm \cdots  &\text{(2)}&&\qquad
&b_\grm \cdots a_1^{-1}                 &\text{(4)}&&\qquad
&\cdots a_\grm^{-1} b_\grm^{-1}            &\text{(6)}
\end{align*}

Denote $n = \wl(w)$ and $w = u_1 \cdots u_n$ with $u_j \in \{a_k,b_k, a_k^{-1}, b_k^{-1},~k = 1, \dots, \grm\}$.
Set $w_k = u_{\sigma^k(1)} \cdots u_{\sigma^k(n)}$ for $k \in \N$ where $\sigma$ is the permutation sending $(1, \dots, n)$ to $(n, 1, \dots, n-1)$, so that $W = \{w_k,~k = 1, \dots,n\}.$

Assume that $w_k$ is of type $A$. If $w_k$ is of the form (5) or (6), it is clear that $w_{k+1}$ cannot be of type $A$. If $w_k$ is of the form (3) or (4), and if $w_{k+1}$ is of type $A$, $w_{k+1}$ is necessarily the form (5) or (6), so that $w_{k+2}$ cannot be of type $A$.  Finally assume that $w_k$ is of the form (1) or (2). Then we see that $w_{k+1}$ cannot be of the form (1) or (2) except if $\grm = 1$. Therefore if $w_{k+1}$ is still of type $A$ and $\grm > 1$, it has one of the forms (3), (4), (5) or (6) and $w_{k+2}$ or $w_{k+3}$ is not of type $A$ by what precedes. We showed that if $\grm > 1$ and $w_k$ is of type $A$, one of the words $w_{k+1}, w_{k+2}, w_{k+3}$ is not of type $A$. Therefore by denoting $W'$ the set of words which are not of type $A$, the conclusion of the lemma holds.

Now suppose $\grm = 1$ so that $w_\star = a_1b_1a_1^{-1}b_1^{-1}$. If $w_k$ is of type $A$ and not of the form (1) or (2), then $w_{k+1}$ or $w_{k+2}$ is not of type $A$ by what precedes. Thus we assume that $w_k$ is of the form (1) or (2), but not of the form (3), (4), (5) or (6) (such words will be called of type $B$). In particular, we have $w_k = \cdots u_{\sigma^k(n)}$ with $u_{\sigma^k(n)} \neq b_1^{-1}, a_1^{-1}$ (as $\grm = 1$). Thus, in the word $(w_\star)^pw_k(w_\star)^q$, simplifications can only occur between $(w_\star)^p$ and $w$, and it is not hard to see that $\#(\dom_k) \leqslant 2$ where
$$
\dom_k = \{(w_\star)^p w_k (w_\star)^q,~p, q \in \Z\} \cap W.
$$
Denote $\mathscr{O} = \{\dom_k,~k = 1, \dots,n,~w_k \text{ is of type } B\}.$ For any $\dom \in \mathscr{O}$ we choose some $w_\dom \in \dom$. Then
$$
W' = \{w \in W,~w \text{ is not of type }A\} \cup \{w_\dom,~\dom \in \mathscr{O}\}
$$
satisfies the conclusion of the lemma.
\end{proof}

Using Lemmas \ref{lem:1} and \ref{lem:2} we thus obtain that the map
$$
\begin{matrix}
\displaystyle{\bigcup_{(\Ccal_1, \Ccal_2) \in \Cscr_1 \times \Cscr_2}}  &W_{\Ccal_1}' \times W_{\Ccal_2}'& \longrightarrow & \Pcal, \\ &(w_1, w_2)& \longmapsto & [w_2 w_1]
\end{matrix}
$$
is injective. Moreover for any $(w_1, w_2)$ we have $\ell([w_2 w_1]) \leqslant \ell([w_1]) + \ell([w_2]) + 4\diam \Sigma + 2$. Indeed, let $\gamma_j \in \Pcal(\Sigma_j)$ denote the unique geodesic corresponding to the class $[w_j]$ for $j=1,2$. Then we may find a smooth curve $\tilde \gamma_j$ based at $x_\star$ such that $\tilde \gamma_j = w_j$ as elements of $\pi_1(\Sigma_j)$ and $\ell(\tilde \gamma_j) \leqslant \ell([w_j]) + 2 \diam \Sigma + 1$ (for example by removing some appropriate small piece of $\gamma_j$ and link the endpoints of the cutted curve to $x_\star$). Thus $\ell([w_2w_1]) \leqslant \ell(\tilde \gamma_2 \tilde \gamma_1) \leqslant \ell(\gamma_1) + \ell(\gamma_2) + 4\diam \Sigma + 2 = \ell([w_1]) + \ell([w_2]) + 4\diam \Sigma + 2.$ We thus obtain, with $R = 4\diam \Sigma + 2$ and $C$ being a constant changing at each line,
$$
\begin{aligned}
\sum_{\substack{\gamma \in \Pcal \\ i(\gamma, \gamma_\star) = 2 \\ \ell(\gamma) \leqslant L}} 1&\geqslant \sum_{\substack{(\Ccal_1, \Ccal_2) \\ \ell(\Ccal_1) + \ell(\Ccal_2) \leqslant L - R}} |W_{\Ccal_1}'| |W_{\Ccal_2}'| \\
& \geqslant \sum_{\substack{(\Ccal_1, \Ccal_2) \\ \ell(\Ccal_1) + \ell(\Ccal_2) \leqslant L - R}} \left(C\ell(\Ccal_1) - C\right)\left(C\ell(\Ccal_2) - C\right) \\
&=  A(L) + C\sum_{\substack{(\Ccal_1, \Ccal_2) \\ \ell(\Ccal_1) + \ell(\Ccal_2) \leqslant L - R}} \ell(\Ccal_1) \ell(\Ccal_2)
\end{aligned}
$$
where we used $|W_{\Ccal_j}'| \geqslant C|W_{\Ccal_j}| - C = C\wl(w_j) - C \geqslant \displaystyle{C\ell(\Ccal_j) - C}$ for $j = 1,2$ (this follows by Lemma \ref{lem:W} and (\ref{eq:milnor})) and the fact that $\ell(w_j) = \ell(\Ccal_j)$ for any $w_j \in W_j$, and where
$$
A(L) \leqslant C \sum_{\substack{(\Ccal_1, \Ccal_2) \\ \ell(\Ccal_1) + \ell(\Ccal_2) \leqslant L - R}} (\ell(\Ccal_1) + \ell(\Ccal_2) + 1).
$$
By Lemma \ref{lem:thereisc} we have
\begin{equation}\label{eq:secondest}
C^{-1} \e^{h_\star L} / L \leqslant \#\{\Ccal_j,~ \ell(\Ccal_j) \leqslant L\} \leqslant C\e^{h_\star L}/L,
\end{equation}
and from this it is not hard to see that $A(L) \ll L\e^{h_\star L}$ as $L \to +\infty.$ Moreover, using (\ref{eq:secondest}) and similar techniques we used to obtain (\ref{eq:firstest}) (for example by noting that there is $C$ such that $\tilde N_j(L) - \tilde N_j(L-C) \geqslant C \e^{h_\star L} / L$ for any $L$ large enough, where $\tilde N_j(L) = \#\{\Ccal_j,~ \ell(\Ccal_j) \leqslant L\}$, as it follows from \eqref{eq:secondest}) we get for $L$ large enough
$$
\sum_{\substack{(\Ccal_1, \Ccal_2) \\ \ell(\Ccal_1) + \ell(\Ccal_2) \leqslant L - R}} \ell(\Ccal_1) \ell(\Ccal_2) \geqslant C (L-R)\e^{h(L-R)},
$$
which concludes the proof of Proposition \ref{prop:lowerbound2}.
\end{proof}

\subsubsection{Upper bound}

Each $\gamma \in \Pcal_2$ with $\ell(\gamma) \leqslant L$ is in the conjugacy class $w_1 w_2$ for some $w_j \in \pi_1(\Sigma_j)$ with $\ell_\star(w_1) + \ell_\star(w_2) \leqslant L +C$. Therefore (\ref{eq:margulis}) implies
$$
\begin{aligned}
N(2, L) &\leqslant \sum_{\substack{w_j \in \pi_1(\Sigma_j) \\ \ell_\star(w_1) + \ell_\star(w_2) \leqslant L + C}} 1 \\
&\leqslant \sum_{k=0}^{L+C} C \exp(h_1 k) \exp(h_2(L-k+C)),
\end{aligned}
$$
which gives for large $L$, if $h_\star = \max(h_1, h_2)$,
$$
N(2, L) \leqslant \left \{ \begin{matrix} C L \exp(h_\star L) &\text{ if } &h_1 = h_2, \vspace{0.15cm}   \\ C \exp(h_\star L) &\text{ if } & h_1 \neq h_2. \end{matrix} \right.
$$
Iterating this process we obtain (with $C$ depending on $n$)
$$
N(n, L) \leqslant \left \{ \begin{matrix} C L^{2n - 1} \exp(h_\star L) &\text{ if } &h_1 = h_2, \vspace{0.15cm}   \\ C L^{n-1} \exp(h_\star L) &\text{ if } & h_1 \neq h_2. \end{matrix} \right.
$$

\subsection{Relative growth of geodesics with small intersection angle}\label{subsec:aprioriangles}

For any $\eta > 0$ small, we consider 
$
N(n, \eta, L) = \#(\Pcal_{\eta,n}(L))
$
where $\Pcal_{\eta,n}(L)$ is the set of closed geodesics $\gamma : \R/\Z \to \Sigma$ of length not greater than $L$, intersecting $\gamma_\star$ exactly $n$ times, and such that there is $\tau$ with $\gamma(\tau) \in \gamma_\star$ and $\mathrm{angle}(\dot \gamma(\tau), T_{\gamma(\tau)}\gamma_\star) < \eta.$ The purpose of this paragraph is to prove the following estimate.

\begin{lemm}\label{lem:estangle}
For any $L_0 > 0$, there is $\eta > 0$ such that for any $L$ big enough
$$
N(n, \eta, L) \leqslant 2n N(n, L - L_0).
$$
\end{lemm}

\begin{proof}

Let $P_{2n}(\Pcal_n)$ denote the set of subsets of $\Pcal_{n}$ which are of cardinal not greater than $2n$. Then for $K \in \N_{\geqslant 1}$ we construct a map
$$
\Psi_K : \Pcal_n \to P_{2n}(\Pcal_n),
$$
as follows. Let $\gamma : \R/\Z \to \Sigma$ be an element of $\Pcal_n$ and let $\tau_1, \dots, \tau_n \in \R/\Z$ be pairwise distinct such that $\gamma(\tau_j) \in \gamma_\star$. For any $j$, choose a path $c_j$ contained in $\gamma_\star$ and linking $x_\star$ to $\gamma(\tau_j).$ Then set $w_j = c_j^{-1} \gamma_j c_j \in \pi_1(\Sigma)$, where $\gamma_j : \R/\Z \to \Sigma$ is defined by $\gamma_j(t) = \gamma(\tau_j + t).$ Then set 
$$
\Psi_K(\gamma) = \left\{\left[w_j w_\star^{\varepsilon K}\right],~ j = 1, \dots, n,~ \varepsilon \in \{-1, 1\}\right\} \in P_{2n}(\Pcal_n).
$$
Here the class $\left[w_j w_\star^{\varepsilon K}\right]$ is identified with the unique geodesic contained in the free homotopy class of $w_j w_\star^{\varepsilon K}$. Note that $\Psi_K$ is well defined: for different choices of $c_j$, we would obtain $w_\star^{k_j} w_j w_\star^{-k_j}$ instead of $w_j$ for some $k_j \in \Z$; however the class $\left[w_\star^{k_j} w_j w_\star^{-k_j} w_\star^{\varepsilon K}\right]$ coincides with $\left[w_j w_\star^{\varepsilon K}\right]$. Moreover, the image of $\Psi_K$ is indeed contained in $\Pcal_n$. Indeed, by similar techniques used to prove Lemma \ref{lem:1}, one can show that the geodesic $[w_j w_\star^{\varepsilon K}]$ intersects $\gamma_\star$ exactly $n$ times, as $\gamma$ does (adding turns aroung $\gamma_\star$ does not change the intersection number with $\gamma_\star$). 

Our goal is now to show that there is $C > 0$ such that for any $K \in \N_{\geqslant 1}$, there is $\eta > 0$ such that
\begin{equation}\label{eq:goal}
\Pcal_{\eta, n}(L) \subset \bigcup_{\gamma \in \Pcal_n(L-K\ell(\gamma_\star) + C)} \Psi_K(\gamma).
\end{equation}
Let $\varepsilon > 0$ smaller than $\rho_g / 2$ where $\rho_g$ is the injectivity radius of $(\Sigma, g)$, and $K \in \N_{\geqslant 1}$. Then there is $\eta > 0$ such that if $z = (0, \tau, \theta) \in \gamma_\star$ (here we use the coordinates of Lemma \ref{lem:coordinates}) with $|\theta| < \eta$ (resp. $|\theta - \pi| < \eta$), then if $z' = (0, \tau, 0)$ (resp. $z' = (0, \tau, \pi)$), we have
$$
\mathrm{dist}_\Sigma(\pi(\varphi_t(z)), \pi(\varphi_t(z'))) \leqslant \varepsilon, \quad t \in [0, K\ell(\gamma_\star)].
$$
Let $c(t) = \pi(\varphi_t(z))$ for $t \in [0, K\ell(\gamma_\star)]$, and close the path $c$ by using the exponential map at $\pi(z)$, to obtain a closed curve $\tilde \gamma : \R/\Z \to \Sigma$ of length not greater than $K\ell(\gamma_\star) + 2\varepsilon$. If $\varepsilon$ (and thus $\eta$) is small enough, we have $\tilde \gamma = \gamma_\star^{\pm K}$ in $\pi_1(\Sigma, \pi(z))$ whenever $|\sin \theta| < \eta$. In particular, if $\gamma : \R/\ell(\gamma)\Z \to \Sigma$ is a closed geodesic intersecting $\gamma_\star$ exactly $n$ times, and with at least one intersection angle smaller than $\eta$, then we can write
\begin{equation}\label{eq:sim}
\gamma \sim w_\star^{\pm K} w'
\end{equation}
for some $w' \in \pi_1(\Sigma)$, satisfying $\ell_\star(w') \leqslant \ell(\gamma) - K\ell(\gamma_\star) + C$ (for some $C > 0$ independent of $\gamma$). Here $a \sim b$ means that $a$ is freely homotopic to $b$. As  before, the unique geodesic contained in the free homotopy class of $w'$ intersects $\gamma_\star$ exactly $n$ times (removing some turns around $\gamma_\star$ does not change the intersection number).

Finally, by similar techniques used in the proofs of Lemmas \ref{lem:constructgeodesics} or \ref{lem:2}, one can see that $[w' w_\star^{\pm K}] \in \Psi_K([w'])$ (again, we identify the geodesic freely homotopic to $w'$ with the class $[w']$) for any $w' \in \pi_1(\Sigma)$ such that $i([w'], \gamma_\star) = n$. Moreover, if $\ell_\star(w') \leqslant \ell(\gamma) - K\ell(\gamma_\star) + C$, then we have of course $\ell([w']) \leqslant \ell_\star(w') \leqslant \ell(\gamma) - K\ell(\gamma_\star) + C$. Thus (\ref{eq:sim}) implies that each $\gamma \in \Pcal_{\eta, n}(L)$ lies  in $\Psi(\gamma')$ for some $\gamma' \in \Pcal_n$ (given by $[w']$) such that $\ell(\gamma') \leqslant L - K\ell(\gamma_\star) + C$. The lemma follows.
\end{proof}

\section{A Tauberian argument}\label{sec:tauberian}

The goal of this section is to give an asymptotic growth of the quantity
$$
N_\pm(n, \chi, t) = \sum_{\substack{\gamma \in \Pcal \\ i(\gamma_\star, \gamma) = n \\ \ell(\gamma) \leqslant t}} I_{\star, \pm}(\gamma, \chi)
$$
as $t \to +\infty$, where $\chi \in C^\infty_c(\partial \setminus \partial_0)$ and 
$
I_{\star, \pm}(\gamma, \chi) = \prod_{z \in I_{\star, \pm}(\gamma)} \chi(z).
$

\subsection{The case $\gamma_\star$ is not separating}\label{subsec:notseparating}

By \cite{dyatlov2016pollicott}, we know that the zeta function
$$
\zeta_{\Sigma_\star}(s) = \prod_{\gamma \in \Pcal_\star} \left(1-\e^{-s\ell(\gamma)}\right)
$$
extends meromorphically to the whole complex plane, and moreover we may write
$$
\zeta_{\Sigma_\star}'(s) / \zeta_{\Sigma_\star}(s) = \sum_{k = 0}^2 (-1)^k \mathrm{tr}^\flat\left(\e^{\pm\varepsilon s} \varphi_{\mp \varepsilon}^*R_{\pm, \delta}(s)|_{\Omega^k_c(M_\delta) \cap \ker \iota_X}\right),
$$
where the flat trace is computed on $M_\delta$. Here $\Pcal_\star$ denote the set of primitive closed geodesics of $(\Sigma_\star, g)$. By \cite{parry1990zeta}, $\zeta_{\Sigma_\star}' / \zeta_{\Sigma_\star}$ is holomorphic in $\{\Re(s) \geqslant h\}$ except for a simple pole at $s = h_\star$, where $h_\star > 0$ is the topological entropy of the geodesic flow of $(\Sigma_\star, g)$ restricted to the trapped set (as defined in the introduction). Moreover, it is shown in \cite{dyatlov2016pollicott} that $s \mapsto R_{\pm, \delta}(s)|_{\Omega^k\cap \ker \iota_X}$ has no pole in $\{\Re(s) > 0\}$ for $k = 0$ and $k = 2$.
Write the Laurent expansion of $R_{\pm, \delta}(s)$ given in \S\ref{subsec:resolv} near $s = h_\star$ as
$$
R_{\pm,\delta}(s)= Y_{\pm, \delta}(s) + \frac{\Pi_{\pm, \delta}(h_\star)}{s-h_\star} : \Omega^\bullet_c(M_\delta) \to \mathcal{D}'^\bullet(M_\delta).
$$
Denote $\Omega^k = \Omega^k_c(M_{\delta})$ and $\Omega^k_0 = \Omega^k \cap \ker \iota_X.$ Then the above comments show that 
$$
\rank(\Pi_{\pm, \delta}|_{\Omega^1_0}) = 1.
$$
As $R_{\pm, \delta}(s)$ commutes with $\iota_X$, it preserves the spaces $\Omega^k_0$. Writing $\Omega^k = \Omega^k_0 \oplus \alpha \wedge \Omega^{k-1}_0$ we have for any $w = u + \alpha \wedge v$ with $\iota_Xu = 0$ and $\iota_X v = 0$,
$$
\Pi_{\pm, \delta}(h_\star)|_{\Omega^2}(u + \alpha \wedge v) = \Pi_{\pm, \delta}(h_\star)|_{\Omega^2_0}(u) + \alpha \wedge \Pi_{\pm, \delta}(h_\star)|_{\Omega^1_0}(v).
$$
Thus $\Pi_{\pm, \delta}|_{\Omega^2} = \alpha \wedge \iota_X \Pi_{\pm, \delta}|_{\Omega^1_0}$. By Proposition \ref{prop:scatresolv} and the fact that $\varphi_{\pm \varepsilon}^* \Pi_{\pm, \delta}(h_\star) = \e^{\pm \varepsilon h_\star} \Pi_{\pm, \delta}(h_\star),$ we have near $s = h_\star$
$$
\tilde \Sc_\pm(s) = Y_{\pm}(s) + \frac{\psi^* \iota^*\iota_X\Pi_{\pm,\delta}(h_\star)\iota_*}{s-h_\star},
$$
where $Y_\pm(s)$ is holomorphic near $s = h_\star$. We denote
$$
\Pi_{\pm, \partial} = \psi^*\iota^*\iota_X \Pi_{\pm,\delta}(h_\star)\iota_* : \Omega^\bullet(\partial) \to \mathcal{D}'^\bullet(\partial).
$$
Then by what precedes, and since $\iota_X \Pi_{\pm, \delta}(h_\star)|_{\Omega^1} = 0$, we obtain that $\rank(\Pi_{\pm, \partial}) = 1$. Finally for any $\chi \in C^\infty_c(\partial \setminus \partial_0)$ we set
$$c_{\pm}(\chi) = \strf(\chi \Pi_{\pm,\partial}).$$ 

\begin{lemm}\label{lem:estnotseparating}
Let $\chi \in C^\infty_c(\partial \setminus \partial_0)$ such that $c_\pm(\chi) > 0$. Then it holds
$$
N_\pm(n,\chi, t) \sim \frac{(c_\pm(\chi) t)^n}{n !} \frac{\e^{h_\star t}}{h_\star t}, \quad t \to +\infty.
$$
\end{lemm}

\begin{proof}
Because $\chi \Pi_{\pm, \partial}$ is of rank one, it follows that $\strf((\chi \Pi_{\pm,\partial})^n) = c_\pm(\chi)^n$ for any $n \geqslant 1$ (since the flat trace of finite rank operator coincide with its usual trace) and thus
$$
\strf((\chi\tilde \Sc_\pm(s))^n) = \frac{c_\pm(\chi)^n}{(s-h_\star)^n} + \mathcal{O}((s-h_\star)^{-n+1}), \quad s \to h_\star.
$$
We set $\eta_{n,\chi}(s) = \strf((\chi\tilde \Sc_\pm(s))^n),$ and 
$$
g_{n,\chi}(t) = \sum_{\substack{\gamma \in \mathcal{P} \\ i(\gamma, \gamma_\star) = n}} \ell(\gamma) \sum_{\substack{k \geqslant 1 \\ k\ell(\gamma)\leqslant t}} I_{\star, \pm}(\gamma, \chi)^k, \quad t \geqslant 0,
$$
Now if $G_{n,\chi}(s) = \displaystyle{\int_0}^{+\infty}g_{n,\chi} (t) \e^{-ts} \dd t$, a simple computation leads to
$$
G_{n,\chi}(s) = \frac{1}{s} \sum_{i(\gamma, \gamma_\star) = n} \ell^\#(\gamma) \e^{-s \ell(\gamma)} I_{\star, \pm}(\gamma, \chi)^{\ell(\gamma) / \ell^\#(\gamma)} = -\frac{\eta_{n,\chi}'(s)}{ns},
$$
where the last equality comes from Proposition \ref{prop:computeflattrace}. Because one has the expansion $\eta_{n,\chi}'(s) = -nc_\pm(\chi)^n(s-h_\star)^{-(n+1)} + \mathcal{O}((s-h_\star)^{-n})$ as $s \to h_\star$, we obtain
$$
G_{n,\chi}(h_\star s) = \frac{c_\pm(\chi)^n}{h_\star^{n+2}(s-1)^{n+1}} + \mathcal{O}((s-h_\star)^{-n}), \quad s \to h_\star.
$$
Then applying the Tauberian theorem of Delange \cite[Th\'eor\`eme III]{delange1954generalisation}, we have
$$
\frac{1}{h_\star}g_{n,\chi}(t/h_\star) \sim \frac{c_\pm(\chi)^n}{h_\star^{n+2}} \frac{\e^{t}}{n!}t^{n}, \quad t \to +\infty,
$$
which reads 
\begin{equation}\label{eq:fn}
\displaystyle{g_{n,\chi}(t) \sim \frac{(c_\pm(\chi) t)^n}{n! h_\star} \exp(h_\star t).}
\end{equation}
Now note that, if $\Pcal_n$ is the set of primitive closed geodesics $\gamma$ with $i(\gamma, \gamma_\star) = n$ one has
$$
g_{n,\chi} (t) \leqslant \sum_{\substack{\gamma \in \mathcal{P}_n \\ \ell(\gamma) \leqslant t}} \ell(\gamma) \lfloor t/\ell(\gamma)\rfloor I_\gamma(\chi) \leqslant t N(n, \chi, t).
$$
As a consequence we have 
\begin{equation}\label{eq:lower}
\liminf_{t \to +\infty} N_\pm(n, \chi, t) \frac{n! h_\star t}{(c_\pm(\chi) t)^n\e^{h_\star t}}  \geqslant 1.
\end{equation}
For the other bound, we use the a priori bound obtained in \S\ref{subsubsec:uppernotseparating}
$$
N_\pm(n, \chi, t) \leqslant N(n, t) \leqslant \frac{Ct^n}{n!} \frac{\e^{h_\star t}}{h_\star t}
$$
to deduce that for any $\sigma > 1$
\begin{equation}\label{eq:t/sigma}
\limsup_{t \to +\infty} N_\pm(n, \chi, t/\sigma) \frac{n!}{t^n} \frac{h_\star t}{\e^{h_\star t}} = 0.
\end{equation}
Now we may write
$$
\begin{aligned}
N_\pm(n, \chi, t) &= N_\pm(n, \chi, t/\sigma) + \sum_{\substack{\gamma \in \Pcal \\ i(\gamma_\star, \gamma) = n \\ t/\sigma \leqslant \ell(\gamma) \leqslant t}} I_{\star,\pm}(\gamma, \chi)  \\
&\leqslant N_\pm(n, \chi, t/\sigma) + \frac{\sigma}{t} \sum_{\substack{\gamma \in \Pcal \\ i(\gamma_\star, \gamma) = n \\ t/\sigma  \leqslant \ell(\gamma) \leqslant t}} I_{\star, \pm}(\gamma, \chi) \ell(\gamma) \\
& \leqslant N_\pm(n, \chi, t/\sigma) + \frac{\sigma}{t} g_{n, \chi}(t),
\end{aligned}
$$
which gives with (\ref{eq:t/sigma})
$$
\limsup_{t \to +\infty} N_\pm(n, \chi, t) \frac{n!}{(c_\pm(\chi) t)^n} \frac{h_\star t}{\e^{h_\star t}} \leqslant \sigma.
$$
As $\sigma > 1$ is arbritrary, the Lemma is proven.
\end{proof}

\subsection{The case $\gamma_\star$ is separating}

In that case, $\Sigma_\delta$ consists of two surfaces $\Sigma_{\delta}^{(1)}$ and $\Sigma_{\delta}^{(2)}$. We write $M_\delta = M_\delta^{(1)} \sqcup M_\delta^{(2)}$ where $M_\delta^{(j)} = S\Sigma_\delta^{(j)}$, $j=1,2,$ and $\partial = \partial^{(1)} \sqcup \partial^{(2)}$ with $\partial^{(j)} \subset M_\delta^{(j)}$. Note that, if
$
\tilde \Sc_\pm^{(j)}(s)
$
denotes the restriction of $\tilde \Sc_\pm(s)$ to $\partial^{(j)},$ we have
$$
\tilde \Sc_\pm^{(1)}(s) : \Omega^\bullet(\partial^{(1)}) \to \mathcal{D}'^\bullet(\partial^{(2)}), \quad \tilde \Sc_\pm^{(2)}(s) : \Omega^\bullet(\partial^{(2)}) \to \mathcal{D}'^\bullet(\partial^{(1)}).
$$
As in \S\ref{subsec:notseparating}, we have 
$$
\tilde \Sc_\pm^{(j)}(s) = Y_\pm^{(j)}(s) + \frac{\Pi_{\pm,\partial}^{(j)}}{s-h_j}, \quad s \to h_j,
$$
where $Y_\pm^{(j)}(s)$ is holomorphic near $s = h_j$ and $h_j$ is the topological entropy of the geodesic flow of $\Sigma_\delta^{(j)}.$ As before we fix $\chi \in C^\infty_c(\partial \setminus \partial_0)$.

\subsubsection{The case $h_1 \neq h_2$}\label{subsubsec:h_1neqh_2}

 We may assume $h_1 > h_2$ and we set $c_\pm(\chi) = \strf\left(\chi \tilde \Sc_\pm^{(2)}(h_1)\chi \Pi_{\pm,\partial}^{(1)}\right).$ Because $\Pi_{\pm,\partial}^{(1)}$ is of rank one, it follows that $\strf\left(\bigl(\chi \tilde \Sc_\pm^{(2)}(h_1)\chi \Pi_{\pm,\partial}^{(1)}\bigr)^n\right) = c_\pm(\chi)^n$ for any $n \geqslant 1$ and thus, by cyclicity of the flat trace (as the flat trace coincide with the real trace for operators of finite rank), we have as $s \to h_1$,
$$
\begin{aligned}
\strf\left((\chi\tilde \Sc_\pm(s))^{2n}\right) &= \strf\left( \bigl(\chi \tilde \Sc_\pm^{(1)}(s)\chi \tilde \Sc_\pm^{(2)}(s)\bigr)^n + \bigl(\chi \tilde \Sc_\pm^{(2)}(s)\chi\tilde \Sc_\pm^{(1)}(s)\bigr)^n\right) \\
&= \frac{2c_\pm(\chi)^n}{(s-h_1)^n} + \mathcal{O}((s-h_1)^{-n+1}).
\end{aligned}
$$
Now we may proceed exactly as in \S\ref{subsec:notseparating} to obtain that, if $c_\pm(\chi) > 0$,
$$
N(2n, \chi, t) \sim \frac{(c_\pm(\chi) t)^n}{n!} \frac{\e^{h_\star t}}{h_\star t}, \quad t \to +\infty.
$$

\subsubsection{The case $h_1 = h_2 = h_\star$}\label{subsubsec:h_1=h_2}

 In that case, by denoting $c_\pm(\chi) = \strf(\Pi_{\pm,\partial}^{(1)}\Pi_{\pm,\partial}^{(2)})$ we have
$$
\strf\left((\chi \tilde \Sc_\pm(s))^{2n}\right) = \frac{2c_\pm(\chi)^n}{(s-h_\star)^{2n}} + \mathcal{O}((s-h_\star)^{-2n +1}), \quad s \to h_\star.
$$
Again, provided that $c_\pm(\chi) \neq 0$, we may proceed exactly as in \S\ref{subsec:notseparating} to obtain
$$
N(2n, \chi, t) \sim 2\frac{(c_\pm(\chi) t^2)^n}{(2n)!}\frac{e^{h_\star t}}{h_\star t}.
$$

\section{Proof of theorem \ref{thm:main}}\label{sec:proofthm}

In this section we prove Theorem \ref{thm:main}. We will apply the asymptotic growth we obtained in the last section to some appropriate sequence of functions in $C^\infty_c(\partial \setminus \partial_0)$. Let $F \in C^\infty(\R, [0,1])$ be an even function such that $F \equiv 0$ on $[-1,1]$ and $F \equiv 1$ on $]-\infty, -2] \cup [2, +\infty[$. For any small $\eta  >0$, set in the coordinates from Lemma \ref{lem:coordinates}
$$
\chi_\eta(z) = F(\theta / \eta), \quad z = (\tau, 0, \theta) \in \partial.
$$
Then $\chi_\eta \in C^\infty_c(\partial \setminus \partial_0)$ for any $\eta > 0$ small. The function $\chi_\eta$  forgets about the trajectories passing at distance not greater than $\eta$ from the "glancing" $S\gamma_\star$.

\subsection{The case $\gamma_\star$ is not separating} \label{subsec:notseparating2}
Recall from \S\ref{sec:apriori} that we have the \textit{a priori} bounds
\begin{equation}\label{eq:aprioriproof}
C^{-1} \frac{\e^{h_\star L}}{h_\star L} \leqslant N(1, L) \leqslant C \e^{h_\star L}
\end{equation}
for $L$ large enough. This estimate implies the following fact\footnote{Indeed, if it does not hold, then there is $\varepsilon > 0$ such that for any $L_0 > 0$ there is $L_1$ such that for any $n \geqslant 0$, it holds
$$
\varepsilon < \frac{N(1, L_1 + nL_0)}{N(1, L_1 + (n+1)L_0)},
$$
which gives
$
N(1, L_1 + (n+1) L_0) \varepsilon^n < N(1, L_1)
$
for each $n.$ As $L_0$ can be chosen arbitrarily, we see that (\ref{eq:aprioriproof}) cannot hold.}
:
$$
\forall \varepsilon > 0, \quad \exists L_0 > 0, \quad \forall L_1 > 0, \quad \exists L > L_1, \quad N(1, L-L_0) \leqslant \varepsilon N(1, L).
$$
In particular, we see with Lemma \ref{lem:estangle} that for any $\eta > 0$ small enough, one has
\begin{equation}\label{eq:estangle2}
\liminf_{L \to +\infty} \frac{N(1, \eta, L)}{N(1, L)} \leqslant \frac{1}{2},
\end{equation}
where $N(1, \eta, L)$ is defined in \S\ref{subsec:aprioriangles}.

For $\eta > 0$ small and $L > 0$, neither $c_{\pm}(\chi_\eta)$ nor $N_\pm(n, \chi_\eta, L)$ (see \S\ref{subsec:notseparating}) depend on $\pm$, since $F$ is an even function. We denote them simply by $c(\eta)$ and $N(n, \chi_\eta, L)$ respectively. We claim that $c(\eta) > 0$ if $\eta  > 0$ is small enough. Indeed, reproducing the arguments from \S\ref{subsec:notseparating} we see that $c(\eta) = 0$ implies 
\begin{equation}\label{eq:ll}
N(1, \chi_\eta, L) \ll \exp(h_\star L)/h_\star L, \quad L \to +\infty.
\end{equation}
On the other hand we have $N(1, L) = N(1, \chi_\eta, L) + R(\eta, L)$ with
$$
R(\eta, L) = N(1,L) - N(1, \chi_\eta, L) \leqslant N(1, 2\eta, L),
$$
and thus, if $\eta$ is small enough, (\ref{eq:estangle2}) gives
$$
\limsup_{L \to +\infty} \frac{N(1, \chi_\eta, L)}{N(1, L)} \geqslant \frac{1}{2},
$$
Since $C^{-1} \exp(h_\star L)/(h_\star L) \leqslant N(1, L)$, we obtain that (\ref{eq:ll}) cannot hold, and thus $c(\eta) > 0.$

In particular we can apply Lemma \ref{lem:estnotseparating} to get $\displaystyle{\lim_{L} N(n, \chi_\eta, L) \frac{n!}{(c(\eta) L)^n}\frac{h_\star L}{\e^{h_\star L}} = 1}$. As $N(n,L) \geqslant N(n, \chi_\eta, L)$ we obtain that for $L$ large enough
$$
C^{-1}\frac{L^n}{n !} \frac{\e^{h_\star L}}{h_\star L} \leqslant N(n, L) \leqslant C\frac{L^n}{n !} \frac{\e^{h_\star L}}{h_\star L}
$$
(the upper bound comes from \S\ref{subsubsec:uppernotseparating}).
Let $\varepsilon > 0$. The last estimate combined with Lemma \ref{lem:estangle} implies that for $\eta > 0$ small enough, one has
$$
\limsup_L R(n, \eta, L) \frac{n!}{L^n}\frac{h_\star L}{\e^{h_\star L}} < \varepsilon,
$$
where $R(n, \eta, L) = N(n, L) - N(n, \chi_\eta, L).$ Thus writing $N(n, \chi_\eta, L) \leqslant N(n,L) \leqslant N(n, \chi_\eta, L) + R(n, \eta, L)$ we obtain
$$
c(\eta)^n \leqslant \liminf_L N(n,L) \frac{n!}{L^n}\frac{h_\star L}{\e^{h_\star L}}  \leqslant \limsup_L N(n,L) \frac{n!}{L^n}\frac{h_\star L}{\e^{h_\star L}} \leqslant c(\eta)^n + \varepsilon
$$
for any $\eta$ small enough. As $\varepsilon > 0$ is arbitrary, we finally get
$$
N(n,L) \sim \frac{(cL)^n}{n !} \frac{\e^{h_\star L}}{h_\star L}, \quad L \to +\infty
$$
where $c_\star = \lim_{\eta \to 0} c(\eta) < +\infty$ (the limit exists as $\eta \mapsto c(\eta)$ is nondecreasing and bounded by above). \newline

\subsection{The case $\gamma_\star$ is separating}

\subsubsection{The case $h_1 \neq h_2$} In that case recall from \S\ref{sec:apriori} that we have the bound
$$
C^{-1} \e^{h_\star L} \leqslant N(2,L) \leqslant C \e^{h_\star L}
$$
for $L$ large enough. In particular, using Lemma \ref{lem:estangle} and \S\ref{subsubsec:h_1neqh_2} we may proceed exactly as in \S\ref{subsec:notseparating2} to obtain
$$
N(2,L) \sim \frac{(c_\star L)^n}{n!} \frac{\e^{h_\star L}}{h_\star L}, \quad L \to +\infty
$$
where $c_\star = \lim_{\eta \to 0} c_{\pm}(\chi_\eta)$.

\subsubsection{The case $h_1 = h_2 = h$} In that case recall from \S\ref{sec:apriori} that we have the bound
$$
C^{-1} L \e^{h_\star L} \leqslant N(2,L) \leqslant C L\e^{h_\star L}
$$
for $L$ large enough. In particular, using Lemma \ref{lem:estangle} and \S\ref{subsubsec:h_1=h_2} we may proceed exactly as in \S\ref{subsec:notseparating2} to obtain
$$
N(2,L) \sim 2\frac{(c_\star L)^n}{(2n)!} \frac{\e^{h_\star L}}{h_\star L}, \quad L \to +\infty
$$
where $c_\star = \lim_{\eta \to 0} c_{\pm}(\chi_\eta)$.

\section{A Bowen-Margulis type measure}\label{sec:bowen}

\subsection{Description of the constant $c_\star$}
In this subsection we describe the constant $c_\star$ in terms of Pollicott-Ruelle resonant states of the open system $(M_\delta, \varphi_t)$, assuming for simplicity that $\gamma_\star$ is not separating. By \S\ref{subsec:resolv} we may write, since $\Pi_{\pm, \delta}(h_\star)$ is of rank one by \S\ref{subsec:notseparating},
$$
\Pi_{\pm,\delta}(h_\star)|_{\Omega^1(M_\delta)} = u_\pm \otimes (\alpha \wedge s_\mp), \quad u_\pm, \in \mathcal{D}'^1_{E_{\pm,\delta}^*}(M_\delta), \quad s_\mp \in \mathcal{D}'^1_{E_{\mp, \delta}^*}(M_\delta),
$$
with $\supp(u_\pm, s_\pm) \subset \Gamma_{\pm, \delta}$ and $u_\pm, s_\mp \in \ker(\iota_X).$ Using the Guillemin trace formula \cite{guillemin1977lectures} and the Ruelle zeta function $\zeta_{\Sigma_\star}$, we see that the Bowen-Margulis measure $\mu_0$ (see \cite{bowen1972equidistribution})  of the open system $(M_\delta, \varphi_t)$, which is given by Bowen's formula
$$
\mu_0(f) = \lim_{L \to +\infty} \sum_{\substack{\gamma \in \Pcal_\delta \\ \ell(\gamma) \leqslant L}} \frac{1}{\ell(\gamma)} \int_0^{\ell(\gamma)} f(\gamma(\tau), \dot \gamma(\tau)) \dd \tau, \quad f \in C^\infty_c(M_\delta),
$$
coincides with the measure
$
\displaystyle{f \mapsto \strf(f \Pi_{\pm,\delta}(h)) = \int_{M_\delta} f ~u_\pm \wedge \alpha \wedge s_\mp}.
$
Note that $\supp(u_\pm \wedge \alpha \wedge s_\mp) \subset K_\star$, where $K_\star \subset S\Sigma_\star$ is the trapped set. On the other hand we have by definition of $\Pi_{\pm, \partial}$,
$$
c_\star = \lim_{\eta \to 0}\strf(\chi_\eta \Pi_{\pm, \partial}) = -\lim_{\eta \to 0} \int_\partial \chi_\eta \psi^* \iota^*u_\pm \wedge \iota^*s_\mp.
$$

\subsection{A Bowen-Margulis type measure}
In what follows we set
$
S_{\gamma_\star} \Sigma = \{(x,v) \in S\Sigma,~x \in \gamma_\star\}
$
and for any primitive geodesic $\gamma : \R/\Z \to \Sigma$, 
$$
I_{\star}(\gamma) = \{z \in S_{\gamma_\star} \Sigma, ~z = (\gamma(\tau), \dot \gamma(\tau)) \text{ for some }\tau\}.
$$
For any $n \geqslant 1$ we define the set $\Gamma_n \subset S_{\gamma_\star} \Sigma$ by
$$
\complement \Gamma_n = \{z \in S_{\gamma_\star} \Sigma, ~\tilde S_{\pm}^k(z) \text{ is well defined for }k=1, \dots, n\}.
$$
Also we set $\ell_n(z) = \max(\ell_{+,n}(z), \ell_{-,n}(z))$ where
$$
\ell_{\pm,n}(z) =  \ell_\pm(z) +  \ell_\pm(\tilde S_\pm(z)) + \cdots + \ell_\pm({\tilde S_{\pm}^{n-1}(z)}), \quad z \in \complement \Gamma_n,
$$
where $\ell_\pm(z) = \inf\{t > 0,~\varphi_{\pm t}(z) \in S_{\gamma_\star}\Sigma\}$.

We will now prove Theorem \ref{thm:equidistribution} which says that for any $f \in C^\infty(S_{\gamma_\star}\Sigma)$ the limit
\begin{equation}\label{eq:bowen}
\mu_n(f) = \lim_{L \to +\infty} \frac{1}{N(n,L)} \sum_{\gamma \in \Pcal_n} \frac{1}{n}\sum_{z \in I_{\star}(\gamma)} f(z)
\end{equation}
exists and defines a probability measure $\mu_n$ on $S_{\gamma_\star}\Sigma$ supported in $\Gamma_n$. We will also prove that, in the separating case,
$$
\mu_n(f) = c_\star^{-n} \lim_{\eta \to 0} \strf(f(\chi_\eta \Pi_{\pm, \partial})^n),
$$
where $c_\star > 0$ is the constant appearing in Theorem \ref{thm:main}. Note that here we identify $f$ as its lift $p_\star^* f$ which is a function on $\partial$, so that the above formula makes sense (recall that $p_\star : S\Sigma_\star \to S\Sigma$ is the natural projection which identifies both components of $\partial S\Sigma_\star = \partial$).
We have of course such a formula in the non separating case but we omit it here.

\begin{proof}[Proof of Theorem \ref{thm:equidistribution}]
Let $f \in C^\infty(S_{\gamma_\star} \Sigma)$. Then reproducing the arguments in the proof of Proposition \ref{prop:computeflattrace}, we get for $\Re(s)$ big enough,
$$
\strf\left(f (\chi_\eta\tilde \Sc_\pm(s))^n\right) = \sum_{i(\gamma, \gamma_\star) = n} \left(\sum_{z \in I_{\star}(\gamma)} f(z)\right) \e^{-s \ell(\gamma)} I_{\star}(\gamma, \chi_\eta),
$$
where $\chi_\eta$ is defined in \S\ref{sec:proofthm} and $I_\star(\gamma, \chi_\eta) = I_{\star, \pm}(\gamma, \chi_\eta)$ (see \S\ref{sec:tauberian}; this does not depend on $\pm$ as the function $F$ used to construct $\chi_\eta$ is even).
Now we may proceed exactly as in \S\ref{sec:tauberian}, replacing $g_{n,\chi}(t)$ by
$$
g_{n, \chi_\eta, f}(t) = \sum_{\substack{\gamma \in \Pcal \\ i(\gamma, \gamma_\star) = n}} \left(\sum_{z \in I_{\star}(\gamma)} f(z)\right) \sum_{\substack{k \geqslant 1 \\ k\ell(\gamma) \leqslant t}} I_\star(\gamma, \chi_\eta), \quad t \geqslant 0,
$$
to obtain that the limit (\ref{eq:bowen}) exists, and is equal to $\lim_{\eta \to 0} c_\star^{-n}\mathrm{Res}_{s=h_\star} \strf(f(\chi_\eta \tilde \Sc_\pm(s))^n)$ provided $\gamma_\star$ is separating. Finally, if $f \in C^\infty_c(S_{\gamma_\star}\Sigma \setminus \Gamma_n)$ then there is $L > 0$ such that
$$
\ell_n(z) \leqslant L, \quad z \in \supp(f).
$$
In particular for any $\gamma \in \Pcal$ such that $i(\gamma, \gamma_\star) = n$ and $\ell(\gamma) > L$, we have $f(z) = 0$ for any $z \in I_\star(\gamma)$. This shows that $\mu_n(f) = 0$ and the support condition for $\mu_n$ follows.
\end{proof}

\section{Application to geodesic billards}\label{sec:billard}
We prove here Corollary \ref{thm:3}. Take $(\Sigma', g')$ a compact oriented negatively curved surface with totally geodesic boundary $\partial \Sigma'.$ We can double the surface to obtain a closed surface $\Sigma,g)$, and the doubled metric $g$ which is smooth outside $\partial \Sigma$ (it is of class $C^{3-\varepsilon}$ near $\partial \Sigma'$ for every $\varepsilon > 0$). However the geodesic flow on $(\Sigma, g)$ remains $C^1$ and Anosov, and one can see that the construction of the scattering operator
$$
\Sc_\pm(s) : \Omega^\bullet(\partial) \to \mathcal{D}'^\bullet(\partial), \quad \partial = \{(x,v),~x \in \partial \Sigma'\} \subset S \Sigma
$$
is still valid in this context\footnote{Indeed we may embed $\Sigma$ into a slightly larger smooth surface $\Sigma_\delta$ with strictly convex boundary to prove (exactly as before) that the scattering operator $\Sc_\pm(s)$ (which does not depend on the extension !) extends meromorphically to the whole complex plane.}, as well as the considerations on its wavefront set. Now $\partial \Sigma$ is a disjoint union of closed geodesics $\gamma_{\star, 1}, \dots, \gamma_{\star,r}$, and the two open surfaces $\Sigma', \Sigma''$ which are the connected components of $\Sigma \setminus \partial \Sigma'$ are smooth and have same entropy. Now, instead of $\tilde \Sc_\pm(s) = \psi^* \circ \Sc_\pm(s)$, consider
$$
\hat \Sc_\pm(s) = R^* \circ \Sc_\pm(s), 
$$
where $R : \partial \to \partial$ is the reflexion according to the Fresnel-Descartes' law. Note that although the geodesic flow is only $C^1$, the operator $\hat \Sc_\pm(s)$ is a weighted version of the transfer operator of the map $z \mapsto R(S_\mp(z))$, which is smooth where it is defined. Thus as in \S\ref{sec:scat}\footnote{We can check the needed wavefront properties by using the fact that the geodesic flow of the doubled surface is still Anosov, as in \S\ref{sec:scat}.}, for any $\chi \in C^\infty_c(\partial \setminus \partial_0)$, we have the trace formula
$$
\strf \left((\chi \hat \Sc_\pm(s))^n\right) = 2n \sum_{i(\gamma) = n} \frac{\ell^\#(\gamma)}{\ell(\gamma)} \e^{-s\ell(\gamma)} \prod_{z \in B(\gamma)}\chi(z), 
$$
but here the sum runs over all closed oriented billard trajectories of $\Sigma'$ with $n$ rebounds (here we have a factor $2$ since we count each trajectory twice as the manifold is doubled), and $B(\gamma)$ is the set of inward pointing vectors in $\partial$ given by the rebounds of $\gamma$. Moreover it is clear that, to each oriented periodic billiard trajectory of $\Sigma'$ with two rebounds, correspond exactly two closed geodesics of $\Sigma$ intersecting exactly twice $\partial \Sigma'.$ The methods given in \S\ref{sec:apriori} that led to an priori bound on the number of closed geodesics intersecting exactly two times $\gamma_\star$ extends in the context of the multicurve $(\gamma_{\star,1}, \dots, \gamma_{\star,r})$ given by $\partial \Sigma'$, for example by choosing a point $x_\star \in \gamma_{\star, 1}$ and composing elements of $\pi_1(\Sigma', x_\star)$ with elements of $\pi_1(\Sigma \setminus \Sigma', x_\star)$ as in \S\ref{sec:apriori}.  Thus we get an a priori lower bound for the number of closed billiard trajectories with two rebounds and as in \S\ref{sec:tauberian} the order of the pole of $\strf \left( (\chi_\eta \hat \Sc_\pm(s))^2\right)$ at $s=h'$ (the entropy of the open system ($\Sigma', g$)) is exactly two for small $\eta$, which implies that the pole of $\strf \left((\chi_\eta \hat \Sc_\pm(s))^n\right)$ is exactly $n$ for every $n$ (as the residue of $\hat \Sc_\pm(s)$ at $s=h$ is of rank one). Thus reproducing the arguments of \S\ref{sec:proofthm} we get Corollary \ref{thm:3}.

\section{A large deviation result}\label{sec:deviation}
The goal of this last section, which is independent of the rest of this paper, is to prove the following result, which is a consequence of a classical large deviation result by Kifer \cite{kifer1994large}. 

\begin{prop}\label{lem:deviation}
There exists $I_\star > 0$ such that the following holds. For any $\varepsilon > 0$, there is $C, \delta > 0$ such that for large $L$
$$
\frac{1}{N(L)} \# \left\{\gamma \in \Pcal,~\left| \frac{i(\gamma, \gamma_\star)}{\ell(\gamma)} - I_\star\right| \geqslant \varepsilon \right\} \leqslant C \exp(-\delta L).
$$
\end{prop}

In fact, $I_\star = 4i(\bar m, \delta_{\gamma_\star})$ where $i$ is the Bonahon's intersection form \cite{bonahon1986bouts}, $\delta_{\gamma_\star}$ is the Dirac measure on $\gamma_\star$ in and $\bar m$ is the renormalized Bowen-Margulis measure on $M$ (here we see the intersection form as a function on the space of $\varphi$-invariant measures on $S\Sigma$, as described below). Lalley \cite{lalley1996self} showed a similar result for self-intersection numbers; see also \cite{pollicott2006angular} for self intersection numbers with prescribed angles.

\subsection{Bonahon's intersection form}

Let $\Mcal_\varphi(S\Sigma)$ be the set of finite positive measures on $S\Sigma$ invariant by the geodesic flow, endowed with the vague topology. For any closed geodesic $\gamma$, we denote by $\delta_\gamma \in \Mcal_\varphi(S\Sigma)$ the Lebesgue measure of $\gamma$ parametrized by arc length (thus of total mass $\ell(\gamma)$). Let $\mu \in \Mcal_\varphi(S\Sigma)$ be the Liouville measure, that is, the measure associated to the volume form
$
\displaystyle{\frac{1}{2} \alpha \wedge \dd \alpha}.
$

\begin{prop}[Bonahon \cite{bonahon1988geometry}, see also Otal \cite{otal1990spectre}]\label{prop:bonahon}
There exists a continuous function
$$
i : \Mcal_\varphi(S\Sigma) \times \Mcal_\varphi(S\Sigma) \to \R_+
$$
which is additive and positively homogeneous with respect to each variable, such that $i(\mu, \mu) = 2 \pi \vol(\Sigma)$ and 
$$
i(\delta_\gamma, \delta_{\gamma'}) = i(\gamma, \gamma'), \quad 
i(\mu, \delta_\gamma) = 2\ell(\gamma),
$$
for any closed geodesics $\gamma, \gamma'$.
\end{prop}

\begin{rem}
\begin{enumerate}[label=(\roman*)]
\item Actually, Bonahon's intersection form is a pairing on the space of \textit{geodesic currents}. This space is naturally identified with the space of $\varphi$-invariant measure on $S\Sigma$ which are also invariant by the flip $R : (x,v) \mapsto (x,-v)$.  What we mean here by $i(\nu, \nu')$ for general $\nu, \nu' \in \Mcal_\varphi(S\Sigma)$ is simply $i(\Phi(\nu), \Phi(\nu'))$ where $\Phi : \nu \mapsto \nu + R^*\nu$ (note that $\varphi_t R=R\varphi_{-t}$ for $t \in \R$). 
\item Note that the formulae for $i(\mu, \mu)$ and $i(\mu, \delta_{\gamma})$ differ from \cite{bonahon1988geometry} ; it is due to our convention since here the Liouville measure $\mu$ corresponds to twice the Liouville current considered in \cite{bonahon1988geometry}.
\end{enumerate}
\end{rem}

\subsection{Large deviations}\label{sec:bonahon}

For any $\nu \in \Mcal_\varphi(S\Sigma)$ we denote by $h(\nu)$ the measure-theoretical entropy of $\varphi$ with respect to $\nu$. Then we have the following result.

\begin{prop}[Kifer \cite{kifer1994large}]
Let $F \subset \Mcal_\varphi^1(S\Sigma)$ be a closed set, where $\Mcal_\varphi^1(S\Sigma)$ is the set of $\varphi$-invariant probability measures on $S\Sigma$. Then
$$
\limsup_L  \frac{1}{L}\log \frac{1}{N(L)} \# \{\gamma \in \Pcal,~\delta_\gamma / \ell(\gamma) \in F \} \leqslant \sup_{\nu \in F} h(\nu) - h,
$$
where $h$ is the entropy of the geodesic flow.
\end{prop}
\begin{proof}[Proof of Lemma \ref{lem:deviation}]
We denote by $\bar m \in \Mcal_\varphi^1(S\Sigma)$ the unique probability measure of maximal entropy, that is 
$$
\bar m = \lim_{L \to +\infty} \sum_{\substack{\gamma \in \Pcal \\ \ell(\gamma) \leqslant L}} \frac{\delta_\gamma}{\ell(\gamma)},
$$
where the convergence holds in the weak sense.
Let $\varepsilon > 0$. Define 
$$
F_\varepsilon = \{\nu \in \Mcal_\varphi^1(S\Sigma),~\left| i(\nu, \delta_{\gamma_\star} ) - i(\bar m, \delta_{\gamma_\star}) \right|\geqslant\varepsilon\}.
$$
Then $F_\varepsilon$ is closed and $\bar m \in  \complement F_\varepsilon$ so that
$
\delta = h - \sup_{\nu \in F_\varepsilon} h(\nu) > 0.
$
In particular we obtain for large $L$
$$
\frac{1}{N(L)} \# \{\gamma \in \Pcal,~\delta_\gamma / \ell(\gamma) \in F_\varepsilon\} \leqslant C \exp(-\delta' L)
$$
for some $0 < \delta' < \delta$ and $C > 0$. Now, by Proposition \ref{prop:bonahon}, $\delta_\gamma / \ell(\gamma) \in F_\varepsilon$ is equivalent to $|i(\gamma, \gamma_\star) / \ell(\gamma) - i(\bar m, \delta_{\gamma_\star})| \geqslant \varepsilon.$ Now let $I_\star = i(\bar m, \delta_{\gamma_\star})$. It is a well known fact that $\bar m$ have full support in $S\Sigma$, which implies $I_\star > 0$ by definition of $i(\bar m, \delta_{\gamma_\star})$ (see \cite{otal1990spectre}). This concludes.
\end{proof}

\begin{rem}
\begin{enumerate}[label=(\roman*)]
\item It is not hard to see that Lemma \ref{lem:deviation} implies
$$
\frac{1}{N(L)} \sum_{\ell(\gamma) \leqslant L} i(\gamma, \gamma_\star) \sim I_\star L
$$
as $L \to +\infty$. Thus we recover \cite[Theorem 4]{pollicott1985asymptotic}. 
\item If $(\Sigma, g)$ is hyperbolic then $\bar m$ is the renormalized Liouville measure and we find, with Proposition \ref{prop:bonahon},
$$
I_\star = \frac{\ell(\gamma_\star)}{2 \pi^2 (\grm - 1)}.
$$
\end{enumerate}
\end{rem}

\section{Extension to multi-curves}\label{sec:multi}
In this last section, we explain how the methods used until there allow to derive Theorem \ref{thm:multi}. Let $\gamma_{\star, 1}, \dots, \gamma_{\star, r}$ be pairwise disjoint closed geodesics of $(\Sigma, g)$, and denote by $\Sigma_1, \dots, \Sigma_q$ the connected components of $\Sigma \setminus \bigcup_{i=1}^r \gamma_{\star, i}$. 

\subsection{Notations}

For any $j = 1, \dots, q$, we denote by $h_j > 0$ the topological entropy of the open system $(\Sigma_j, g|_{\Sigma_j})$, and by $B_j$ the set of indexes $i$ such that $\gamma_{\star, i}$ is a boundary component of $\Sigma_j$. We decompose $B_j$ as 
$$
B_j = S_j \sqcup O_j
$$
where $S_j$ is the set of indexes $i$ such that $\gamma_{\star,i}$ lies in $B_{j'}$ for some $j' \neq j$, and $O_j = B_j \setminus S_j$. In fact $S_j$ (resp. $O_j$) is the set of shared (resp. unshared) boundary components of $\Sigma_j$.

For any $\nbf = (n_1, \dots, n_r) \in \N^r$ we define
$$
\langle \nbf, \Sigma_j \rangle = \sum_{i=1}^r n_i \left( \frac{1_{S_j}(i)}{2} + 1_{O_j}(i)\right), \quad j = 1, \dots, q.
$$
This quantity represents the number of times a curve has to travel through $\Sigma_j$ if it intersects $n_i$ times $\gamma_{\star, i}$.

An \textit{admissible} path $(u,v)$ is the collection of two words $u = u_1 \cdots u_n$ and $v= v_1 \cdots v_n$ with $u_\ell \in \{1, \dots, r\}$ and $v_\ell \in \{1, \dots, q\}$ for $\ell = 1,\dots, n$, and with the following property. For any $\ell \in \Z/n\Z$ we have
$
u_{\ell}, u_{\ell + 1} \in B_{v_\ell}
$
and 
$$
v_{\ell} = v_{\ell + 1} \implies u_{\ell + 1} \in O_{v_\ell}.
$$
For any admissible path $\omega = (u,v)$ we denote $\nbf(\omega) = (n_1, \dots, n_r)$ where we set $n_i = \#\{\ell,~u_\ell = i\}.$ An admissible path $\omega$ will be called \textit{primitive} if every non trivial cyclic permutation of $\omega$ is distinct from $\omega$.

An element $\nbf \in \N^r$ will be called admissible if $\nbf = \nbf(\omega)$ for some admissible path $\omega$. For any admissible $\nbf \in \N^r$ we set

$$
h_\nbf = \max\{h_j,~\langle \nbf, \Sigma_j \rangle > 0\} \quad \text{ and } \quad d_\nbf = \sum_{h_j = h_\nbf} \langle \nbf, \Sigma_j\rangle.
$$
The number $h_\nbf$ is the maximum of the entropies encountered by a closed geodesic $\gamma$ satisfying $i(\gamma, \gamma_\star) = n_i$ for $i = 1, \dots, r$, while $d_\nbf$ is the number of times $\gamma$ will travel through a surface $\Sigma_j$ with $h_j = h_\nbf$.

\subsection{Statement}

For any primitive geodesic $\gamma \in \Pcal$ we denote 
$$
\ibf(\gamma, \vec \gamma_\star) = (i(\gamma, \gamma_{\star, 1}), \dots, i(\gamma, \gamma_{\star, r})).
$$
Note that each closed geodesic $\gamma : \R/\Z \to \Sigma$ gives rise to an admissible path $\omega(\gamma)$ (which is unique up to cyclic permutation) defined as follows. Let $(\tau_1, \dots, \tau_n) \in (\R/\Z)^n$ be a cyclically ordered sequence such that $\gamma^{-1}\left(\bigcup_i \gamma_{\star, i}\right) = \{\tau_1, \dots, \tau_n\}$. Then there are words $u_1 \cdots u_n$ and $v_1 \cdots v_n$ such that $\gamma(\tau_\ell) \in \gamma_{\star, u_\ell}$ and $\gamma(\tau) \in \Sigma_{v_\ell}$ for any $\tau \in  ]\tau_\ell, \tau_{\ell + 1}[$ and we set $\omega(\gamma) = (u,v).$ For two paths $\omega,\omega'$, we will write $\omega \sim \omega'$ if $\omega$ is a cyclic permutation of $\omega'$.

\begin{theo}\label{thm:multi2}
Let $\omega$ be an admissible and primitive path. Then there is $c_\omega > 0$ such that for any $k \geqslant 1$
$$
\#\{\gamma \in \Pcal,~\ell(\gamma) \leqslant L,~\omega(\gamma) \sim \omega^k\} \sim d_{\nbf(\omega)}\frac{\displaystyle{\left(c_\omega L^{d_{\nbf(\omega)}}\right)^k}}{(k d_{\nbf(\omega)})!} \frac{\e^{h_{\nbf(\omega)} L}}{h_{\nbf(\omega)} L}
$$
In particular we obtain for any admissible $\nbf \in \N^r$
$$
\#\left\{\gamma \in \Pcal,~\ell(\gamma) \leqslant L,~\mathbf{i}(\gamma, \gamma_\star) = \nbf\right\} \sim C_{\nbf}  L^{d_{\nbf}} \frac{\e^{h_\nbf L}}{h_\nbf L}
$$
where $\displaystyle{C_{\nbf} = d_{\nbf} \sum_{\substack{[\omega] : \nbf(\omega) = \nbf}} c_\omega}$. Here the sum runs over classes $[\omega] = \{\omega',~\omega' \sim \omega\}$.

\end{theo}

\subsection{Proof of Theorem \ref{thm:multi2}}

We let $\Sigma_\star = \bigsqcup_{j=1}^q \Sigma_j$ denote the compact surface with geodesic boundary obtained by cutting $\Sigma$ along $\gamma_{\star, 1}, \dots, \gamma_{\star, r}$, and set 
$$
\partial = \{(x,v) \in S\Sigma_\star,~x \in \partial \Sigma_\star\}.
$$
Then the construction of \S\ref{sec:scat} applies perfectly in this context, and we denote by 
$$
\Sc_\pm(s) : \Omega^\bullet(\partial) \to \mathcal{D}'^\bullet(\partial)
$$
the Scattering operator. For any $i = 1, \dots, r$, we let $F_i \in C^\infty(\partial)$ defined by $F_i(z) = 1$ if $\pi(p(z)) \in \gamma_{\star,i}$ and $F_i(z) = 0$ if not. Here we recall that $p_\star : S\Sigma_\star \to S\Sigma$ and $\pi : S\Sigma \to \Sigma$ are the natural projections. Also we denote $\psi : \partial \simeq \partial$ the smooth map which exchanges the connected components of $(\pi \circ p_\star)^{-1}(\gamma_{\star, i})$ via the natural identification, and we set
$$
\tilde \Sc_\pm(s) = \psi^*\Sc_\pm(s).
$$
Let $\omega = (u, v)$ be a primitive admissible word of length $n \geqslant 1$ and $\chi \in C^\infty_c(\partial \setminus \partial_0)$ (recall that $\partial_0 = \cup_{i} p^{-1}(S\gamma_{\star, i})$ is the tangential part of $\partial$). Then set
$$
\tilde \Sc_\pm(\chi, \omega, s) = F_{u_1} \chi \tilde \Sc_\pm^{(v_n)}(s) F_{u_n} \cdots \chi \tilde \Sc_\pm^{(v_1)}(s) F_{u_1} : \Omega^\bullet(\partial_{u_1}) \to \mathcal{D}'^\bullet(\partial_{u_1}).
$$
Here $\tilde S_\pm^{(v_\ell)}$ is the scattering operator associated to the surface $\Sigma_{v_\ell}$ for $\ell = 1, \dots, n$, and $\partial_{u_1} = (\pi \circ p_\star)^{-1}(\gamma_{\star, u_1})$. As in \S\ref{subsec:flattrace}, we find
$$
\strf\left(\tilde \Sc_\pm(\chi, \omega, s)\right) = \sum_{\omega(\gamma) \sim \omega}  \e^{-s \ell(\gamma)} \prod_{z \in I_{\star, \pm}(\gamma)} \chi(z),
$$
where for a closed geodesic $\gamma : \R/\Z \to \Sigma$ we denoted 
$$
I_{\star, \pm}(\gamma) = \{z \in \partial_\pm,~\pi \circ p_\star(z) = \gamma(\tau) \text{ for some }\tau \in \R/\Z\}.
$$
More generally, for $k \geqslant 1$ we have
\begin{equation}\label{eq:finalmulti}
\sum_{\omega'  \sim \omega^k} \strf\left(\tilde \Sc_\pm(\chi, \omega', s)\right) = |\omega| \sum_{\omega(\gamma) \sim \omega^k} \frac{\ell^\#(\gamma)}{\ell(\gamma)} \e^{-s \ell(\gamma)} \left(\prod_{z \in I_{\star, \pm}(\gamma)} \chi(z)\right)^{\ell(\gamma) / \ell^\#(\gamma)}.
\end{equation}
where $|\omega| = n$ is the length of $\omega$, and where the sum runs over all the path that are cyclic permutations of $\omega^k$ (there are $|\omega|$ of them as $\omega$ is primitive).

Note that $\max_\ell \{h_{v_\ell}\} = h_{\nbf(\omega)}$ and
$$
\#\{\ell \in \{1, \dots, n\},~h_{v_\ell} = h_{\nbf(\omega)}\} = d_{\nbf(\omega)}.
$$
Moreover, as in \S\ref{subsec:notseparating}, the following holds. For any $\ell$ such that $h(v_\ell) = h_{\nbf(\omega)}$ we have 
$$
F_{u_{\ell + 1}}\chi \tilde \Sc_\pm^{(v_\ell)}(s) = \frac{F_{u_{\ell + 1}}\chi \tilde \Pi_{\pm, \partial_{v_\ell}} F_{u_\ell}}{s-h_{\nbf(\omega)}}+ \dom_{\Omega^\bullet(\partial_{u_\ell}) \to \mathcal{D}'^\bullet(\partial_{u_{\ell + 1}})}\left(1\right), \quad s \to h_{v_{\nbf(\omega)}},
$$
for some operator $\tilde \Pi_{\pm, \partial_{v_\ell}}$ satisfying that $F_{u_{\ell + 1}}\chi \tilde \Pi_{\pm, \partial_{v_\ell}} F_{u_\ell}$ is of rank one.
$$
\tilde \Sc_\pm(\chi, \omega, s) = \frac{A_\pm(\chi, \omega)}{(s-h_{\nbf(\omega)})^{d_{\nbf(\omega)}}} + \dom_{\Omega^\bullet(\partial_{u_1}) \to \mathcal{D}'^\bullet(\partial_{u_1})}\left((s-h_{\nbf(\omega)})^{1-d_{\nbf(\omega)} }\right), \quad s \to h_{\nbf(\omega)},
$$
for some operator $A_\pm(\chi, \omega) : \Omega^\bullet(\partial_{u_1}) \to \mathcal{D}'^\bullet(\partial_{u_1})$ of rank one. As we obviously have $\tilde \Sc_\pm(\chi, \omega^k, s) = \tilde \Sc_\pm(\chi, \omega, s)^k$ for $k \geqslant 1$, we obtain
$$
\strf\left(\tilde \Sc_\pm(\chi, \omega^k, s)\right) = \frac{c_\pm(\chi, \omega)^k}{(s-h_{\nbf(\omega)})^{k d_{\nbf(\omega)}}} + \dom\left((s-h_{\nbf(\omega)})^{1 - kd_{\nbf(\omega)}}\right), \quad s \to h_{\nbf(\omega)},
$$
where we set $c_\pm(\chi, \omega) = \strf(A_\pm(\chi, \omega)).$ In particular, if we are able to show that for some $C > 0$ we have for $L$ large enough
\begin{equation}\label{eq:apriorimulti}
C^{-1} L^{d_{\nbf(\omega)}-1} \e^{h_{\nbf(\omega)}L} \leqslant \#\{\gamma \in \Pcal,~\ell(\gamma) \leqslant L,~\omega(\gamma) \sim \omega\} \leqslant C L^{d_{\nbf(\omega)}-1} \e^{h_{\nbf(\omega)}L},
\end{equation}
then Theorem \ref{thm:multi} will follow by reproducing the arguments from \S\ref{sec:tauberian},\ref{sec:proofthm} (we also need an estimate on the number of geodesics with $\omega(\gamma) \sim \omega$ intersecting one of the $\gamma_{\star, u_\ell}$ with a small angle as in \S\ref{subsec:aprioriangles}). Those facts may be proven using similar techniques as those presented in \S\ref{sec:apriori}, by writing every $\gamma$ satsisfying $\omega(\gamma) \sim \omega$ as free homotopy classes of elements of the form $w_{1} \cdots w_{n}$ with $w_\ell \in \pi_1(\Sigma_{v_\ell}, x_{v_\ell})$ for some collection of $x_{j} \in \Sigma_{j}$ (the composition is made by using a path linking $x_{v_\ell}$ to $x_{v_{\ell + 1}}$ and passing through $\gamma_{\star, u_{\ell + 1}}$). Indeed, proceeding as in Lemmas \ref{lem:constructgeodesics} and \ref{lem:2}, we obtain that $[w_1 \cdots w_n] = [w_1' \cdots w_n']$ as conjugacy classes in $\pi_1(\Sigma)$ if and only if $w_\ell = (w_{\star, u_\ell, v_\ell})^{-p_\ell} w_\ell' (w_{\star, u_{\ell + 1}, v_{\ell + 1}})^{p_{\ell + 1}}$ for each $\ell$ for some $p_\ell \in \Z$, where $w_{\star, u_\ell, v_\ell}$ is an element of $\pi_1(\Sigma_{v_\ell}, x_{v_\ell})$ representing $\gamma_{\star, u_\ell}$. Now in the same spirit of Lemma \ref{lem:0} one can show that for some $C$, we have for each $\ell$ and $w_\ell'$
$$
\#\{(p,q) \in \Z,~\ell\left(\left[(w_{\star, u_\ell, v_\ell})^{-p} w_\ell' (w_{\star, u_{\ell + 1}, v_{\ell + 1}})^{q}\right]\right) \leqslant L\} \leqslant C(L - \ell(\Ccal_{w_\ell'})-C)^2
$$
where $\ell(\Ccal_{w'_\ell}) = \inf_{p,q}\ell\left(\left[(w_{\star, u_\ell, v_\ell})^{-p} w_\ell' (w_{\star, u_{\ell + 1}, v_{\ell + 1}})^{q}\right]\right)$. Thus by similar computations made in \S\ref{subsec:aprioriseparating} we obtain the lower bound of (\ref{eq:apriorimulti}), by using that 
\begin{equation}\label{eq:roblin}
\#\left\{w_\ell \in \pi_1(\Sigma_{v_\ell}, x_{v_\ell}),~\mathrm{dist}_{\tilde \Sigma_{v_\ell}}(\tilde x_{v_\ell},~w_{\ell} \cdot \tilde x_{v_\ell})\right\} \sim A_\ell \e^{h_{v_\ell} L}
\end{equation}
and $\#\{\ell,~h_{v_\ell} = h_{\nbf(\omega)}\} = d_{\nbf(\omega)}$. Also (\ref{eq:roblin}) gives the upper bound of (\ref{eq:apriorimulti}) (and the desired bound for $\#\{\gamma,~\omega(\gamma) \sim \omega^k,~ \ell(\gamma) \leqslant L\}$, for $k \geqslant 1$). 

Combining (\ref{eq:finalmulti}), (\ref{eq:apriorimulti}) and an appropriate version of Lemma \ref{lem:estangle} (which naturally extends in this context), we obtain Theorem \ref{thm:multi2} by making the support of $1- \chi$ arbitrarily close to $\partial_0$, as in \S\ref{sec:proofthm}, and by setting $c_\omega = \lim_{\supp(1-\chi) \to \partial_0} c_\pm(\chi, \omega).$
\appendix

\section{An elementary fact about pullbacks of distributions}

\begin{lemm}\label{lem:elementary}
Let $K \in \mathcal{D}'(\R^d\times \R^d)$ be a compactly supported distribution. We assume that $\WF(K) \subset \Gamma$ where $\Gamma \subset T^*(\R^d \times \R^d)$ is a closed conical subset such that 
$$
\Gamma \cap N^*\Delta = \emptyset, \quad N^*\Delta = \{(x,\xi, x, -\xi),~(x,\xi) \in T^*\R^d\}.
$$
In particular the pullback $i^*K$, where $i : x \mapsto (x,x)$, is well defined. Then for $N \in \N_{\geqslant 1}$ large enough, the following holds. Let $u \in C^N_c(\R^d)$ and assume that the pullback $i^*(\pi_1^*u K)$ is well defined, where $\pi_1 : (x,x) \mapsto x$ is the projection on the first factor. Then
$$
i^*(\pi_1^*u K) = u (i^* K).
$$
\end{lemm}

\begin{proof}
Let $K_\varepsilon \in C^\infty(\R^d \times \R^d),~\varepsilon \in ]0, 1]$, be a sequence of distributions supported in a fixed compact set such that $K_\varepsilon \to K$ in $\mathcal{D}'_\Gamma(\R^d \times \R^d)$. Let $\Gamma' \subset T^*(\R^{d}\times \R^d)$ an open conical subset containing $N^*\Delta$. As $K_\varepsilon$ is compactly supported we may assume that $|t-q| > \delta_0$ for any $(t, q) \in \Gamma \times \Gamma'$ such that $|t| = |q| = 1$ for some $\delta_0 > 0.$ As a consequence, for every $N$ there is $C_N > 0$ such that for any $\varepsilon > 0$ small enough,
\begin{equation}\label{eq:est1}
\left|\widehat K_\varepsilon(q)\right| \leqslant C_N \langle q \rangle^{-N}, \quad q \in \Gamma',
\end{equation}
Let $\Gamma'' \subset \Gamma'$ another open conical subset containing $N^*\Delta$ and let $\delta > 0$ such that for any $q \in \Gamma''$ and $t \in \R^{2d}$ one has
\begin{equation}\label{eq:est2}
|t - q| < \delta |q| \quad \implies \quad t \in \Gamma'.
\end{equation}
Then for any $q \in \Gamma''$
$$
\begin{aligned}
(2\pi)^{2d}\left|\widehat{K_\varepsilon \pi_1^*u}(q)\right| &\leqslant  \int_{\R^{2d}_t} |\widehat K_\varepsilon(t)| |\widehat{\pi_1^*u}(q-t)| \dd t \\
&\leqslant \int_{|t-q| < \delta |q|} |\widehat K_\varepsilon(t)| |\widehat{\pi_1^*u}(q-t)| \dd t + \int_{|t-q| \geqslant \delta |q|} |\widehat K_\varepsilon(t)| |\widehat{\pi_1^*u}(q-t)| \dd t.
\end{aligned}
$$
Let $N_1,N_2 \in \N_{\geqslant_1}$. We have, with $\langle t \rangle = \sqrt{1+|t|^2}$, using (\ref{eq:est1}) and (\ref{eq:est2}), assuming that $u \in C^{N_2}_c(\R^d)$ with $N_2 \geqslant 2d + 1$,
$$
\begin{aligned}
\int_{|t-q| < \delta |t|} |\widehat K_\varepsilon(t)| |\widehat{\pi_1^*u}(q-t)| \dd t  &\leqslant C_{N_1, N_2} \int_{|t-q| < \delta |q|} \langle t \rangle^{-N_1} \langle q-t\rangle^{-N_2} \dd t\\
&\leqslant C_{N_1, N_2}' \langle q \rangle^{-N_1 + N_2} \int_{\R^{d}} \langle t \rangle^{-N_2} \dd t.
\end{aligned}
$$
where we used Peetre's inequality. On the other hand, we have with $k$ being the order of $K$, and any $N_3 \in \N_{\geqslant 1}$ such that $u \in C^{N_3}_c(\R^{d})$
$$
\begin{aligned}
\int_{|t-q| \geqslant \delta |q|} | \widehat K_\varepsilon(t)| |\widehat{\pi_1^*u}(q-t)| \dd t &\leqslant C_{k, N_3} \int_{|t-q| \geqslant \delta |q|} \langle t \rangle^{k} \langle q-t \rangle^{-N_3} \\
&\leqslant C'_{k, N_3} \langle q \rangle^{-N_3 + (k + 2d + 1)} \int_{\R^{2d}} \langle t \rangle^{-2d-1} \dd t.
\end{aligned}
$$
Therefore, if $u \in C^N(\R^d)$ with $N = k + 2d + 1 + N'$ we have
\begin{equation}\label{eq:est3}
(2\pi)^{2d}\left|\widehat{K_\varepsilon \pi_1^*u}(q)\right| \leqslant C_N \langle q \rangle^{-N'}, \quad q \in \Gamma''.
\end{equation}
Note that for $\varphi \in C^\infty_c(\R^d)$ one has
$$
\langle i^*(K_\varepsilon \pi_1^*u ), \varphi \rangle = \int_{\R^d_x} \varphi(x) \int_{\R^d_\xi \times \R^d_\eta} \widehat{K_\varepsilon \pi_1^*u}(\xi, \eta) \e^{ix(\xi + \eta)} \dd \xi \dd \eta \dd x.
$$
Indeed (\ref{eq:est3}) shows that the integral in $(\xi, \eta)$ converges near $N^*\Delta$ if $N' \geqslant 2d + 1$, and far from $N^*\Delta$ we can use the stationnary phase method to get enough convergence in $(\xi, \eta)$, so that the above integral makes sense as an oscillatory integral and coincides with $\langle i^*(K_\varepsilon \pi_1^*u ),  \varphi \rangle$, since this formula is obviously true if $u$ is smooth. Moreover all the above estimates are uniform in $\varepsilon,$ and thus letting $\varepsilon \to 0$ we obtain the desired result, since obviously one has
$$
i^*(K_\varepsilon \pi_1^*u) = u(i^*K_\varepsilon), \quad \varepsilon \in ]0,1].
$$
\end{proof}

\bibliography{bib.bib}

\begin{thebibliography}{Bow73}

\bibitem[AB67]{atiyah1967lefschetz}
Michael~Francis Atiyah and Raoul Bott.
\newblock A lefschetz fixed point formula for elliptic complexes: I.
\newblock {\em Annals of Mathematics}, pages 374--407, 1967.

\bibitem[Ana00]{anantharaman2000precise}
Nalini Anantharaman.
\newblock Precise counting results for closed orbits of anosov flows.
\newblock In {\em Annales Scientifiques de l’{\'E}cole Normale Sup{\'e}rieure},
  volume~33, pages 33--56. Elsevier, 2000.

\bibitem[Ano67]{anosov1967geodesic}
Dmitry~Victorovich Anosov.
\newblock Geodesic flows on closed riemannian manifolds of negative curvature.
\newblock {\em Trudy Matematicheskogo Instituta Imeni VA Steklova}, 90:3--210,
  1967.

\bibitem[Ber03]{Berger_2003}
Marcel Berger.
\newblock {\em A Panoramic View of Riemannian Geometry}.
\newblock Springer Berlin Heidelberg, 2003.

\bibitem[BH13]{bridson2013metric}
Martin~R Bridson and Andr{\'e} Haefliger.
\newblock {\em Metric spaces of non-positive curvature}, volume 319.
\newblock Springer Science \& Business Media, 2013.

\bibitem[Bon86]{bonahon1986bouts}
Francis Bonahon.
\newblock Bouts des vari{\'e}t{\'e}s hyperboliques de dimension 3.
\newblock {\em Annals of Mathematics}, 124(1):71--158, 1986.

\bibitem[Bon88]{bonahon1988geometry}
Francis Bonahon.
\newblock The geometry of teichm{\"u}ller space via geodesic currents.
\newblock {\em Inventiones mathematicae}, 92(1):139--162, 1988.

\bibitem[Bon15]{bonthonneau2015resonances}
Yannick Bonthonneau.
\newblock {\em R{\'e}sonances du laplacien sur les vari{\'e}t{\'e}s {\`a}
  pointes}.
\newblock PhD thesis, Universit{\'e} Paris Sud-Paris XI, 2015.

\bibitem[Bow72]{bowen1972equidistribution}
Rufus Bowen.
\newblock The equidistribution of closed geodesics.
\newblock {\em American Journal of Mathematics}, 94(2):413--423, 1972.

\bibitem[Bow73]{bowen1973symbolic}
Rufus Bowen.
\newblock Symbolic dynamics for hyperbolic flows.
\newblock {\em American journal of mathematics}, 95(2):429--460, 1973.

\bibitem[Del54]{delange1954generalisation}
Hubert Delange.
\newblock G{\'e}n{\'e}ralisation du th{\'e}oreme de ikehara.
\newblock In {\em Annales scientifiques de l'{\'E}cole Normale Sup{\'e}rieure},
  volume~71, pages 213--242, 1954.

\bibitem[DG16]{dyatlov2016pollicott}
Semyon Dyatlov and Colin Guillarmou.
\newblock Pollicott--ruelle resonances for open systems.
\newblock In {\em Annales Henri Poincar{\'e}}, volume~17, pages 3089--3146.
  Springer, 2016.

\bibitem[DR20]{dang2020poincar}
Nguyen~Viet Dang and Gabriel Rivi{\`e}re.
\newblock Poincar\'e series and linking of legendrian knots.
\newblock {\em arXiv preprint arXiv:2005.13235}, 2020.

\bibitem[ES16]{erlandsson2016counting}
Viveka Erlandsson and Juan Souto.
\newblock Counting curves in hyperbolic surfaces.
\newblock {\em Geometric and Functional Analysis}, 26(3):729--777, 2016.

\bibitem[Gui77]{guillemin1977lectures}
Victor Guillemin.
\newblock Lectures on spectral theory of elliptic operators.
\newblock {\em Duke Math. J.}, 44(3):485--517, 09 1977.

\bibitem[Gui86]{guillope1986distribution}
Laurent Guillop{\'e}.
\newblock Sur la distribution des longueurs des g{\'e}od{\'e}siques ferm{\'e}es
  d’une surface compacte {\`a} bord totalement g{\'e}od{\'e}sique.
\newblock {\em Duke Mathematical Journal}, 53(3):827--848, 1986.

\bibitem[Gui17]{guillarmou2017lens}
Colin Guillarmou.
\newblock Lens rigidity for manifolds with hyperbolic trapped sets.
\newblock {\em Journal of the American Mathematical Society}, 30(2):561--599,
  2017.

\bibitem[H{\"o}r90]{hor1}
L.~H{\"o}rmander.
\newblock {\em The analysis of linear partial differential operators:
  Distribution theory and Fourier analysis}.
\newblock Springer Study Edition. Springer-Verlag, 1990.

\bibitem[Kif94]{kifer1994large}
Yuri Kifer.
\newblock Large deviations, averaging and periodic orbits of dynamical systems.
\newblock {\em Communications in mathematical physics}, 162(1):33--46, 1994.

\bibitem[KS88]{katsuda1988homology}
Atsushi Katsuda and Toshikazu Sunada.
\newblock Homology and closed geodesics in a compact riemann surface.
\newblock {\em American Journal of Mathematics}, 110(1):145--155, 1988.

\bibitem[Lal88]{lalley1988closed}
Steven~P. Lalley.
\newblock {\em Closed geodesics in homology classes on surfaces of variable
  negative curvature}.
\newblock 1988.

\bibitem[Lal89]{lalley1989renewal}
Steven~P Lalley.
\newblock Renewal theorems in symbolic dynamics, with applications to geodesic
  flows, noneuclidean tessellations and their fractal limits.
\newblock {\em Acta mathematica}, 163(1):1--55, 1989.

\bibitem[Lal96]{lalley1996self}
Steven~P Lalley.
\newblock Self-intersections of closed geodesics on a negatively curved
  surface: statistical regularities.
\newblock {\em Convergence in ergodic theory and probability (Columbus, OH,
  1993)}, 5:263--272, 1996.

\bibitem[LS62]{lyndon1962equation}
Roger~C Lyndon and Marcel-Paul Sch{\"u}tzenberger.
\newblock The equation $ a^{M}= b^{N}c^{P}$ in a free group.
\newblock {\em Michigan Mathematical Journal}, 9(4):289--298, 1962.

\bibitem[Mar69]{margulis1969applications}
Gregorii~A Margulis.
\newblock Applications of ergodic theory to the investigation of manifolds of
  negative curvature.
\newblock {\em Functional analysis and its applications}, 3(4):335--336, 1969.

\bibitem[Mir08]{mirzakhani2008growth}
Maryam Mirzakhani.
\newblock Growth of the number of simple closed geodesies on hyperbolic
  surfaces.
\newblock {\em Annals of Mathematics}, 168(1):97--125, 2008.

\bibitem[Mir16]{mirzakhani2016counting}
Maryam Mirzakhani.
\newblock Counting mapping class group orbits on hyperbolic surfaces.
\newblock {\em arXiv preprint arXiv:1601.03342}, 2016.

\bibitem[Ota90]{otal1990spectre}
Jean-Pierre Otal.
\newblock Le spectre marqu{\'e} des longueurs des surfaces {\`a} courbure
  n{\'e}gative.
\newblock {\em Annals of Mathematics}, 131(1):151--162, 1990.

\bibitem[Pol85]{pollicott1985asymptotic}
Mark Pollicott.
\newblock Asymptotic distribution of closed geodesics.
\newblock {\em Israel Journal of Mathematics}, 52(3):209--224, 1985.

\bibitem[Pol91]{pollicott1991homology}
Mark Pollicott.
\newblock Homology and closed geodesics in a compact negatively curved surface.
\newblock {\em American Journal of Mathematics}, 113(3):379--385, 1991.

\bibitem[PP83]{parry1983analogue}
William Parry and Mark Pollicott.
\newblock An analogue of the prime number theorem for closed orbits of axiom a
  flows.
\newblock {\em Annals of mathematics}, pages 573--591, 1983.

\bibitem[PP90]{parry1990zeta}
W.~Parry and M.~Pollicott.
\newblock {\em Zeta Functions and the Periodic Orbit Structure of Hyperbolic
  Dynamics}.
\newblock Soci{\'e}t{\'e} math{\'e}matique de France, 1990.

\bibitem[PPS12]{paulin2012equilibrium}
Fr{\'e}d{\'e}ric Paulin, Mark Pollicott, and Barbara Schapira.
\newblock Equilibrium states in negative curvature.
\newblock {\em arXiv preprint arXiv:1211.6242}, 2012.

\bibitem[PS87]{phillips1987geodesics}
Ralph Phillips and Peter Sarnak.
\newblock Geodesics in homology classes.
\newblock {\em Duke Mathematical Journal}, 55(2):287--297, 1987.

\bibitem[PS06]{pollicott2006angular}
Mark Pollicott and Richard Sharp.
\newblock Angular self-intersections for closed geodesics on surfaces.
\newblock {\em Proceedings of the American Mathematical Society},
  134(2):419--426, 2006.

\bibitem[Rob03]{roblin2003ergodicite}
Thomas Roblin.
\newblock Ergodicit{\'e} et {\'e}quidistribution en courbure n{\'e}gative.
\newblock {\em M{\'e}moire de la Soci{\'e}t{\'e} math{\'e}matique de France},
  (95):A--96, 2003.

\bibitem[Sar80]{sarnak1980prime}
Peter Sarnak.
\newblock {\em Prime geodesic theorems}.
\newblock Stanford University, California, 1980.

\bibitem[ST76]{singer1976lecture}
I.~M. Singer and J.~A. Thorpe.
\newblock {\em Lecture Notes on Elementary Topology and Geometry}.
\newblock Springer Verlag, Berlin, Boston, 1976.

\end{thebibliography}
\bibliographystyle{alpha}

\end{document}